\documentclass{amsart}
\usepackage{amsmath,amsthm,amssymb,latexsym,graphicx,textcomp, hyperref, enumerate,mathtools}
\usepackage[dvipsnames]{xcolor}
\usepackage{circuitikz}
\RequirePackage{fix-cm}
\usepackage[T1]{fontenc}
\usepackage{hyperref}
\usepackage[capitalize]{cleveref}
\usepackage{comment}

\usepackage{tikz}
\usetikzlibrary{decorations.pathreplacing}
\usepackage{float}
\usepackage{standalone}
 \usepackage{graphicx}
  \graphicspath{{pictures/}}
  \usepackage{xspace}
  \usepackage[all]{xy}
  \usepackage{pinlabel}
  \usepackage{enumerate}
  \usepackage[font=small,labelfont=bf]{caption}
  \usepackage{wrapfig}
\usepackage{lipsum}
\usepackage[percent]{overpic}
\usepackage{caption}
\usepackage{subcaption}
\usepackage{stackrel}
\usepackage{accents}
\newcommand*{\dt}[1]{%
  \accentset{\mbox{\large\bfseries .}}{#1}}

\usepackage{multirow}
\usepackage{mathtools}

  \newcommand{\calM}{\mathcal{M}}

  \newcommand{\calT}{\mathcal{T}}

  \renewcommand{\AA}{\mathbb{A}}
  
  \newcommand{\CC}{\mathbb{C}}
  \newcommand{\DD}{\mathbb{D}}

  \newcommand{\HH}{\mathbb{H}}

  \newcommand{\QQ}{\mathbb{Q}}
  \newcommand{\RR}{\mathbb{R}}

  \newcommand{\ZZ}{\mathbb{Z}}
  \newcommand{\clip}{\operatorname{Clip}}

  \newcommand{\figref}[1]{Figure~\ref{#1}}

\theoremstyle{definition}
\newtheorem{proposition}{Proposition}[section]
\newtheorem{corollary}[proposition]{Corollary}
\newtheorem{lemma}[proposition]{Lemma}
\newtheorem{theorem}[proposition]{Theorem}
\newtheorem{definition}[proposition]{Definition}
\newtheorem{notation}[proposition]{Notation}
\newtheorem{example}[proposition]{Example}
\newtheorem{remark}[proposition]{Remark}

\title[Flippered surfaces and renormalized volumes of their moduli spaces]{Flippered hyperbolic surfaces and renormalized volumes of their moduli spaces I}
\author{Yi Huang and Ivan Telpukhovskiy}
\date{\today}

\begin{document}
\begin{abstract}
We introduce a natural class of hyperbolic surfaces called flippered surfaces that generalize crowned hyperbolic surfaces (i.e.: worldsheets for open strings). We develop their Teichm\"uller and moduli-space theory, construct generalized Weil--Petersson volume forms and Chekhov's action, and prove that the resulting generalized Mirzakhani volumes are finite. We establish three geometric recursion formulae---neck chopping, disk excision, and crown extraction---which express these volumes in terms of those of topologically simpler surfaces.
For the fundamental polygonal and annular cases, we derive integral representations involving conical Legendre functions, as well as explicit formulae in terms of elliptic integrals and polylogarithms. We further describe the arithmetic structure of the Taylor coefficients of these volumes, show that suitable specializations are Kontsevich--Zagier periods, and prove identities at the imaginary boundary length $2\pi\sqrt{-1}$ that generalize the Do--Norbury paraphrasing of string and dilaton-type equations.

\end{abstract}
\maketitle

\section{Introduction}

Oriented worldsheets of open strings are Riemann surfaces with boundary. In the puncture description of open--closed string theory, open-string insertions
are represented by boundary punctures, whereas closed-string insertions are represented by interior punctures.
The uniformization (see, e.g., \cite[Pg.~24]{huang_thesis}) of such surfaces are crowned hyperbolic surfaces \cite{cassonbleiler}. The natural generalization of the Weil--Petersson volume form on the moduli spaces of crowned hyperbolic surfaces \cite{chekhov2024,goncharov-sun,HT25} has infinite volume, and Chekhov \cite{chekhov2024} suggested a renormalization via an action functional $S \colon \calM_\Sigma(\vec{b}) \to \RR_+$ on the moduli space of crowned hyperbolic surfaces and defined \emph{Mirzakhani volumes} as the integrals
\[
  V_\Sigma(\vec{b})
  = \int_{\calM_\Sigma(\vec{b})} e^{-S(X)}\,\Omega^{\text{WP}}_\Sigma(\vec{b}).
\]
In \cite{HT25}, we showed that $V_\Sigma(\vec{b})$ lies in $\QQ_{>0}\!\left[b_1^2,\ldots,b_m^2,\,\log 2,\,\zeta(j),\,\beta(k)\right]$,
where the constants are integer-evaluated values of the Riemann zeta function
and Dirichlet beta function.

In this paper, we replace point-like punctures by slits of positive size. We refer to the hyperbolic uniformizations of such Riemann surfaces as \emph{flippered hyperbolic surfaces} (see \cref{{defn:anatomy}}). Roughly speaking, just as uniformization geometrizes interior and boundary punctures respectively as cusps and half-cusps (a.k.a. tines), uniformization geometrizes interior and boundary (removed) intervals as funnels and half-funnels (a.k.a. flippers). See \cref{fig:anatomy} for an example.

Our goal is to define and develop the Weil--Petersson/Mirzakhani volume theory of moduli spaces of flippered hyperbolic surfaces. In particular, we shall see that their moduli spaces naturally arise in various topologically flavored volume recursion formulae in the volumes for crowned hyperbolic surfaces. 

We make special mention of \cite{goncharov-sun}, where Goncharov and Sun study the volumes of certain slices of the decorated moduli space of crowned hyperbolic surfaces and integrate via normalization with respect to a potential based on horocycle lengths (albeit phrased in terms of areas in their paper). There are strong structural similarities between our works which should be the subject of future investigation.

\begin{center}
\begin{figure}[ht!]
    \includegraphics[width=0.8\textwidth]{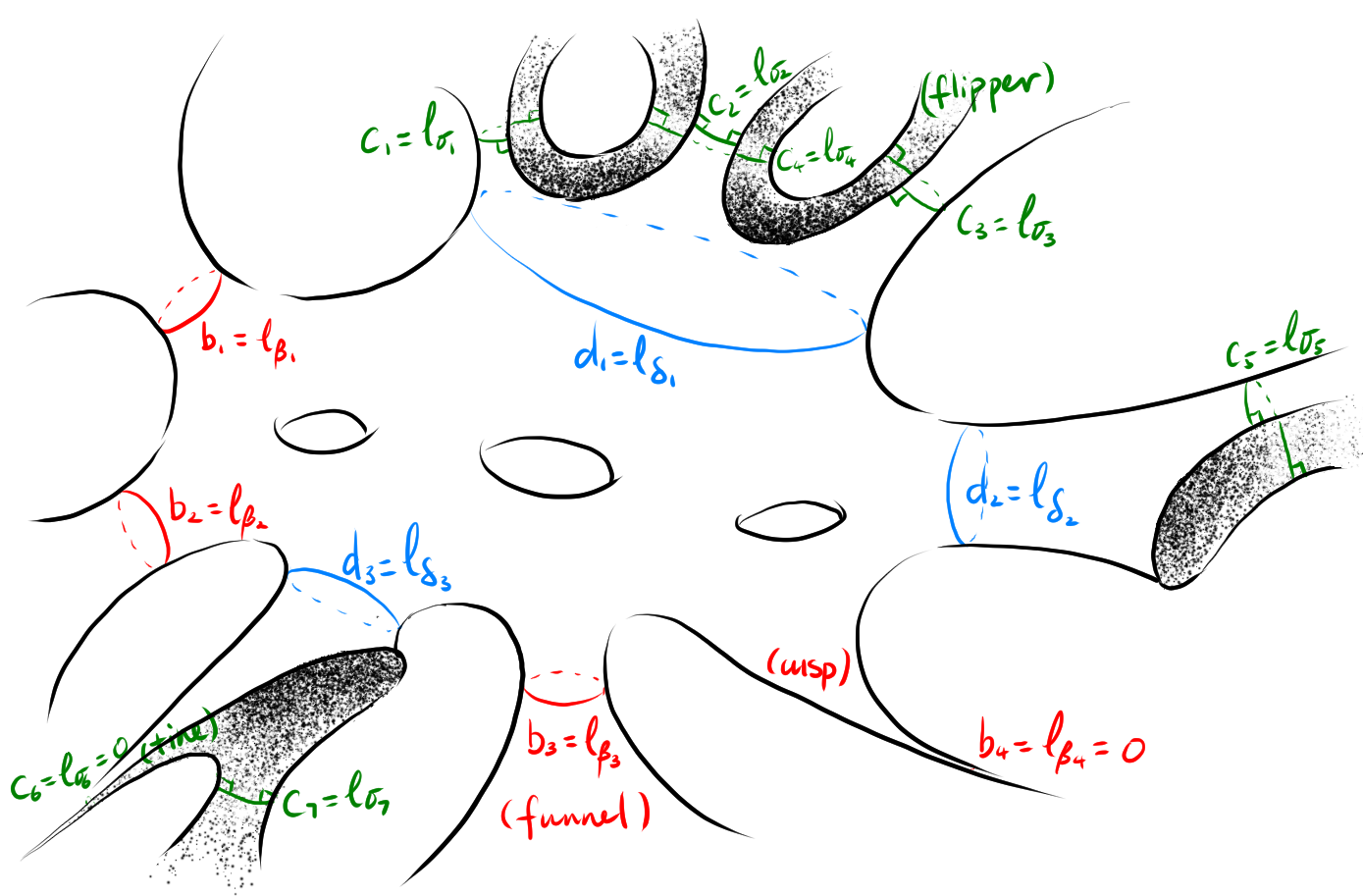}
\caption{A flippered hyperbolic surface with \emph{cuffs} of lengths $\vec{b}=(b_1,b_2,0)$, \emph{straps} of lengths $\vec{c}=(c_1,\ldots,c_5,0,c_7)$ and \emph{necks} of lengths $\vec{d}=(d_1,d_2,d_3)$.}\label{fig:anatomy}
\end{figure}
\end{center}

\subsection{Teichm\"uller and moduli space theory of flippered surfaces}

In Sections~\ref{sec:teich-moduli-spaces-flipsurf}, \ref{sec:gen-chekhov-action} and \ref{sec:WP-volume-form}, we develop the necessary framework for Teichm\"uller and moduli space theory of \emph{flippered} hyperbolic surfaces:
\begin{itemize}
    \item $\Sigma=\Sigma_{g,m,\vec{a}}$, where $\vec{a}=(a_1,\ldots,a_l)$, is an oriented surface with genus $g$, $m$ cuffs, and $l$ crowns respectively with $a_1,\ldots,a_l$ boundary punctures (yielding a total of $n=a_1+\ldots+a_l$ boundary punctures);
    \item the \emph{Teichm\"uller space} $\calT_{g,m,\vec{a}}=\calT_\Sigma$ is the space of isotopy classes of flippered hyperbolic metrics on $\Sigma$;
    \item the \emph{moduli space} $\calM_{g,m,\vec{a}}=\calM_\Sigma$ is $\calT_\Sigma / \text{MCG}_\Sigma$; 
    \item $\calM_\Sigma(\vec{b};\vec{c})$ is the subset of $\calM_\Sigma$ consisting of surfaces with fixed cuff lengths $\vec{b}=(b_1,\ldots,b_m) \in [0,\infty)^m$ and $n$ straps of fixed lengths $\vec{c}=(\vec{c}_1,\ldots,\vec{c}_l)=(c^{(1)}_1,\ldots,c^{(a_1)}_1,\ldots,c^{(1)}_l,\ldots,c^{(a_l)}_l) \in [0,\infty)^n$;
    \item $\calM_\Sigma(\vec{b};\vec{c}\;|\vec{d})$ is the subset of $\calM_\Sigma(\vec{b};\vec{c})$ consisting of surfaces with necks of length $\vec{d}=(d_1,\ldots,d_k)\in(0,\infty)^l$; 
    \item \emph{generalised Weil--Petersson volume forms} $\Omega^{\text{WP}}_\Sigma(\vec{b};\vec{c})$ and $\Omega^{\text{WP}}_\Sigma(\vec{b};\vec{c}\;|\vec{d})$ on the respective moduli spaces $\calM_\Sigma(\vec{b};\vec{c})$ and $\calM_\Sigma(\vec{b};\vec{c}\;|\vec{d})$ (see \cref{eq:vol-form-teich-fixed-hol});
    \item \emph{generalised Chekhov's function} $S \colon \calM_\Sigma(\vec{b};\vec{c}) \to \RR_+$ (see \cref{defn: action for flippered surface}).
\end{itemize}
There are important differences between the Teichm\"uller theory of flippered surfaces and what is available via established cluster algebraic theory: we are unaware of a natural candidate for a ``Weil--Petersson'' Poisson algebraic structure. 

\begin{definition}[Generalised Mirzakhani volume]
\label{def:gen-mirz-vol}
We refer to
\begin{align}
V_\Sigma(\vec{b};\vec{c}\;|\vec{d})
&=
\int_{\mathcal{M}_\Sigma(\vec{b};\vec{c}\;|\vec{d})}
e^{-S(X)}\;\Omega^{\mathrm{WP}}_{\Sigma}(\vec{b};\vec{c}\;|\vec{d})\\
V_\Sigma(\vec{b};\vec{c})
&= 
\int_{\calM_\Sigma(\vec{b};\vec{c})}
e^{-S(X)}\,\Omega^{\text{WP}}_\Sigma(\vec{b};\vec{c}).
\end{align}
respectively as the \emph{generalised Mirzakhani volumes} of $\calM_\Sigma(\vec{b};\vec{c})$ and $\calM_\Sigma(\vec{b};\vec{c}\;|\vec{d})$.
\end{definition}

\begin{theorem}[volume finiteness, \cref{thm:finvolflipper}]
The generalized Mirzakhani volumes of the moduli spaces of flippered surfaces are finite. Moreover, if $c_i\geqslant c_i'$, then
\begin{align}
V_\Sigma(\vec{b};c_1,\dotsc c_i,\dotsc,c_n|\vec{d}) 
& \leqslant 
V_\Sigma(\vec{b};c_1,\dotsc c_i',\dotsc,c_n|\vec{d}) \\
V_\Sigma(\vec{b};c_1,\dotsc c_i,\dotsc,c_n) 
& \leqslant 
V_\Sigma(\vec{b};c_1,\dotsc c_i',\dotsc,c_n).
\end{align}
\end{theorem}

\subsection{Main results}

We obtain \emph{neck (chopping)}, \emph{disk (excision)} and \emph{crown (extraction)} recursions for Mirzakhani volumes.

\begin{proposition}[neck recursion, \cref{prop:neck-recursion} (cf. \cite{chekhov2024,goncharov-sun,HT25})]
\label{prop:neck}
Consider $\Sigma=\Sigma'\cup \mathbb{A}_m$ (see \cref{fig:neckrecursion}) with $\vec{c}:=(\vec{c}\,',c_1,\ldots,c_m)$ and $\vec{d}:=(\vec{d}',\ell)$, then 
\begin{align}
V_\Sigma\bigl(\vec{b};\vec{c}\;|\vec{d}\bigr) 
&=
V_{\Sigma'}\bigl(\vec{b},\ell;\vec{c}\,'|\vec{d}\,'\bigr)
\cdot 
V_{\mathbb{A}_m}\bigl(c_1,\ldots,c_m|\ell\bigr) \cdot \ell 
\text{, and}\\
V_\Sigma\bigl(\vec{b};\vec{c}\bigr) 
&= 
\int_0^\infty 
V_{\Sigma'}\bigl(\vec{b},\ell;\vec{c}\,'\bigr)
\cdot
V_{\mathbb{A}_m}\bigl(c_1,\dots,c_m|\ell\bigr) 
\,\cdot \ell\;\mathrm{d}\ell.
\end{align}
\end{proposition}

\begin{center}
\begin{figure}[ht!]
    \includegraphics[width=0.8\textwidth]{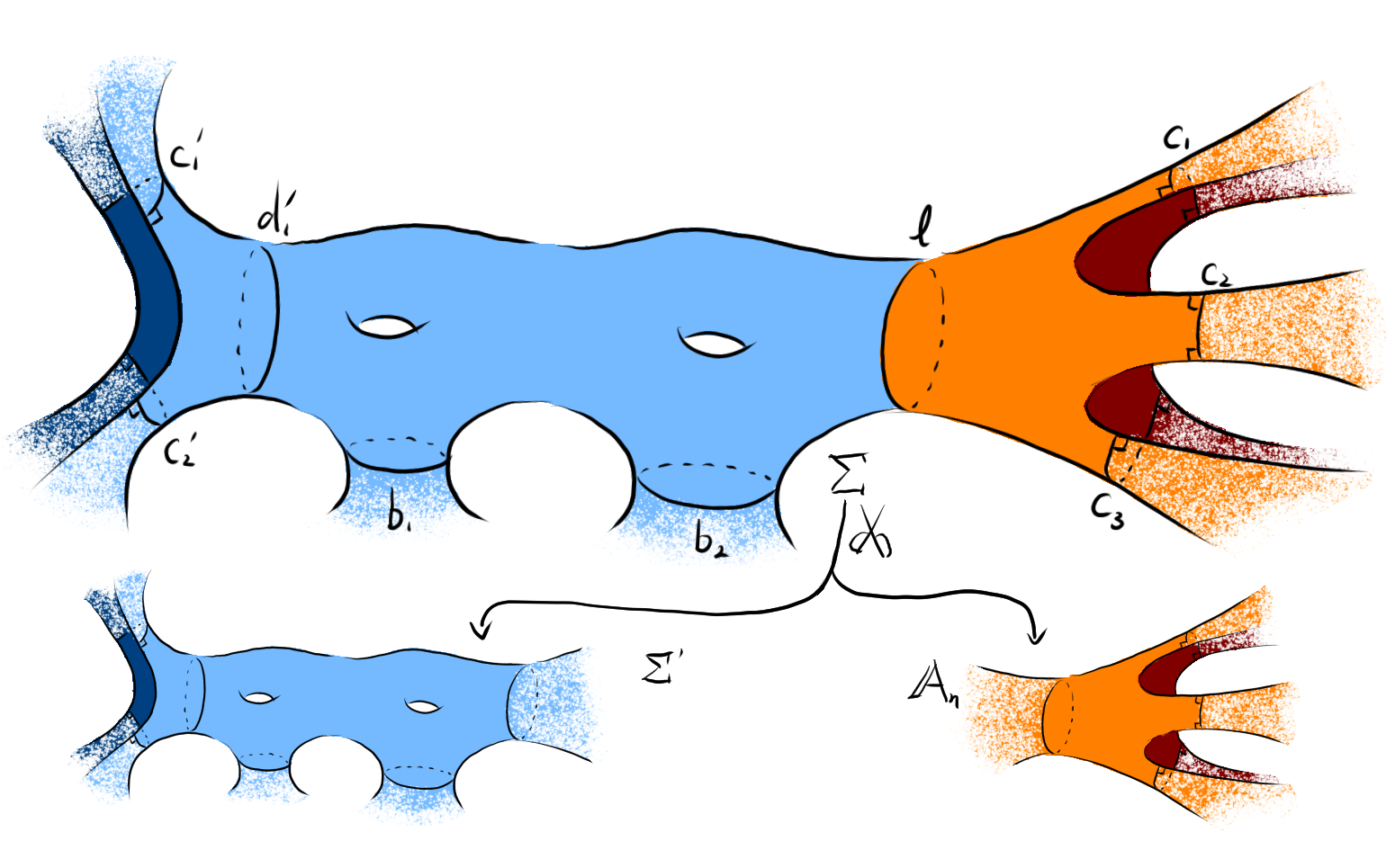}
\caption{A surface $\Sigma$ cut into (Nielsen extensions of) $\Sigma'$ and $\mathbb{A}_n$.}\label{fig:neckrecursion}
\end{figure}
\end{center}

To simplify the statement of the neck recursion formula, we have used a mild form of notation abuse: $\Sigma$ is not the union of $\Sigma'$ and $\mathbb{A}_m$, but rather one has to truncate funnels from $\Sigma'$ and $\mathbb{A}_m$ before gluing them to form $\Sigma$. We will adopt similar notation nomenclature through-out, and a more precise relationship between these surfaces are spelled out in \cref{sec:chekhovactionrecursion}.

\begin{theorem}[disk recursion, \cref{thm:disk-recursion}]
Now consider $\Sigma=\Sigma'\cup \mathbb{D}_n$ (see \cref{fig:diskrecursion}) with $\vec{c}:=(\vec{c}\,',c_1,\ldots,c_n)$, then
\begin{align}
V_{\Sigma}\bigl(\vec{b};\vec{c}\;|\vec{d}\bigr) 
&= \frac{1}{2} \int_0^\infty V_{\Sigma'}\bigl(\vec{b};\vec{c}\,',\ell|\vec{d}\bigr)\cdot
V_{\DD_{n+1}}\bigl(c_1,\ldots,c_n,\ell\bigr) \, \mathrm{d}\cosh\ell\text{, and }\\
V_{\Sigma}\bigl(\vec{b};\vec{c}\bigr) 
&= \frac{1}{2} \int_0^\infty V_{\Sigma'}\bigl(\vec{b};\vec{c}\,',\ell\bigr)\cdot
V_{\DD_{n+1}}\bigl(c_1,\ldots,c_n,\ell\bigr) \, \mathrm{d}\cosh\ell.
\end{align}
\end{theorem}

\begin{center}
\begin{figure}[ht!]
    \includegraphics[width=0.8\textwidth]{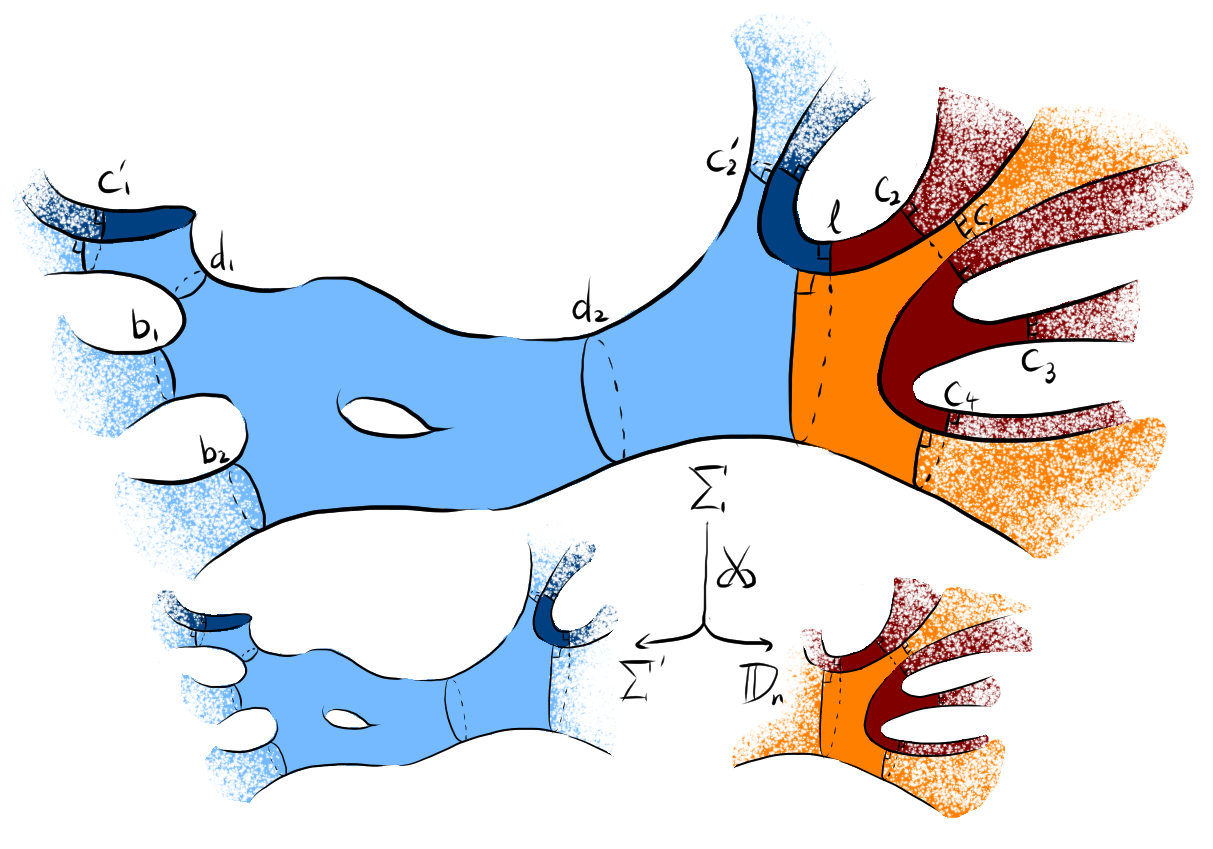}
\caption{A surface $\Sigma$ ``Nielsen cut'' (see \cref{def:cut-orth}) into $\Sigma'$ and $\mathbb{D}_{n+1}$.}\label{fig:diskrecursion}
\end{figure}
\end{center}

\begin{theorem}[crown recursion, \cref{thm:crown-rec}] Consider $\Sigma  = \AA_{n} = \AA_{n-m}^\circ\cup \AA_{m}^\circ$ (see \cref{fig:crownrecursion}) with $\vec{c}:=(\vec{c}\,',c_1,\ldots,c_m)$, then

\begin{align}
V_\Sigma\bigl(\vec{c}\;|d\bigr) 
&=
\int_\RR V_{\AA_{n-m}}\bigl(\vec{c}\,'|\;|d-\ell|\bigr)
\cdot 
V_{\mathbb{A}_m}\bigl(c_1,\ldots,c_m|\;|\ell|\bigr) \;\mathrm{d}\ell
\text{, and}\\
V_\Sigma\bigl(\vec{c}\bigr) 
&= 
2\cdot V_{\AA_{n-m}}\bigl(\vec{c}\,'\bigr) \cdot V_{\AA_{m}}\bigl(c_1,\ldots,c_m\bigr).
\end{align}
\end{theorem}

\begin{center}
\begin{figure}[ht!]
    \includegraphics[width=\textwidth]{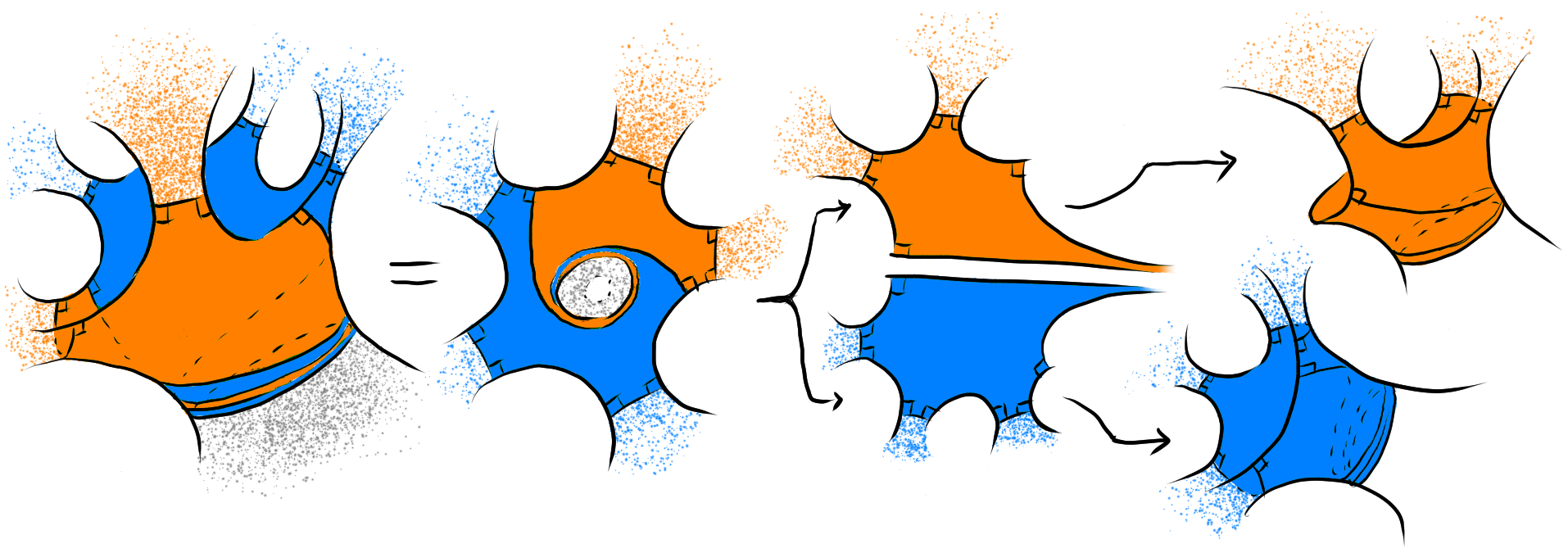}
\caption{An illustration of how one can decompose $\mathbb{A}_5$ into $\mathbb{A}_2$ (red) and $\mathbb{A}_3$ (blue). The initial spiraling direction is chosen arbitrarily, but the final spiraling direction is determined by the geometry of the decomposed pieces. In particular, it is possible to end up with two surfaces with spirals in the opposite direction.}\label{fig:crownrecursion}
\end{figure}
\end{center}

The above recursion formulae assert the importance of the volumes of moduli spaces of topologically simpler surfaces such as $\DD_n$ and $\AA_n$. We show:

\begin{theorem}[\cref{thm:vol-n-gon-1-int}, \cref{thm:vol-n-crown-1-int}]
For $n\geqslant 3$,
\begin{equation}
{V}_{\DD_n}(c_1,\dotsc,c_n) = \pi^{n-2} \int_0^\infty \xi \sinh{2\pi\xi} \cdot \prod_{i=1}^n \frac{P_{-1/2+\sqrt{-1}\xi}(\cosh c_i)}{\cosh{\pi\xi}} \,\mathrm{d}\xi.
\end{equation}    
For $n\geqslant 1$, 
\begin{equation}
V_{\AA_n} (c_1,\dotsc,c_n|d)  =  \pi^{n-1}\int_0^\infty \cos(\omega d) \cdot\prod_{i=1}^n \frac{P_{-1/2+\sqrt{-1}\omega}(\cosh c_i)}{\cosh \pi\omega} \,\mathrm{d}\omega.
\end{equation}
\end{theorem}

\begin{corollary}[\cref{cor:vol-n-crown-no-neck}]
For $n\geqslant 1$,
\begin{equation*}
V_{\AA_n} (c_1,\dotsc,c_n)  = 2^{n-1} \prod_{i=1}^n \frac{K\left(\tanh(c_i/2)\right)}{\cosh(c_i/2)},
\end{equation*}
where $K(s):=\int_0^1\frac{\mathrm{d}t}{\sqrt{(1-t^2)(1-s^2t^2)}}$ is the complete elliptic integral of the first kind. 
\end{corollary}

\begin{theorem}[\cref{thm:vol-str}]

For any $\Sigma$, the Taylor germ of $V_{\Sigma}(\vec{b};\vec{c})$ in $c$-variables has coefficients in 
\[
\QQ[b_1^2,\dotsc, b_m^2,\log 2, \zeta(r),\beta(r)],\quad 1\leqslant r \leqslant \dim_\RR\mathcal{M}_\Sigma(\vec{b};\vec{c}).
\]
\end{theorem}

\begin{remark}
Mirzakhani volume integrals are \emph{periods} in the sense of Kontsevich--Zagier \cite{KontsevichZagier2001} for specialized values of $b_i$ and $c_j$ (e.g.: when $b_i$ and $\cosh c_j$ algebraic numbers, see \cref{thm:vol-periods}). Compare this to the exponential volumes in \cite{goncharov-sun}, which are \emph{exponential} periods \cite[\S4.3]{KontsevichZagier2001}.
\end{remark}

\begin{theorem}[string and dilaton-type equations, \cref{thm:vol-2pi}]

For $\Sigma = \Sigma_{g,m+1,\vec{a}}$ with $2g-2+m+l>0$, and $\vec{b}=(b_1,\dotsc,b_m)$,
\begin{equation}
\begin{split}
V_{\Sigma}(\vec{b},2\pi\sqrt{-1};\vec{c}\;|\vec{d}) = &\sum_{i=1}^m \int_{0}^{b_i} b_i V_{\Sigma\cup p_{m+1}}(\vec{b};\vec{c}\;|\vec{d})\,\mathrm{d}b_i 
\\&+\sum_{j=1}^l d_j V_{\AA_{a_j}}(\vec{c}_j|d_j) \int_0^{d_j} \frac{V_{\Sigma\cup p_{m+1}}(\vec{b};\vec{c}\;|\vec{d})}{V_{\AA_{a_j}}(\vec{c}_j|d_j)} \mathrm{d}d_j.
\end{split}
\end{equation}
For $\Sigma = \AA_n$ with $n\geqslant 3$,
\begin{equation}
\frac{\partial V_{\Sigma}}{\partial d}\left(\vec{c}\;|2\pi\sqrt{-1}\right) = -\pi\sqrt{-1}V_{\Sigma\cup p_1}(\vec{c}).
\end{equation}
For $\Sigma = \Sigma_{g,m+1,\vec{a}}\neq\AA_n$ and $\vec{b}=(b_1,\dotsc,b_m)$,
\begin{align}
\frac{\partial V_\Sigma}{\partial b_{m+1}} \bigl(\vec{b},2\pi\sqrt{-1};\vec{c}\;|\vec{d}\bigr)
&=
2\pi \sqrt{-1} (2g-2+m+l)V_{\Sigma\cup p_{m+1}}\bigl(\vec{b};\vec{c}\;|\vec{d}\bigr)
\text{, and}\\
\frac{\partial V_\Sigma}{\partial b_{m+1}} \bigl(\vec{b},2\pi\sqrt{-1};\vec{c}\bigr)
&= 
2\pi \sqrt{-1} (2g-2+m+l)V_{\Sigma\cup p_{m+1}}\bigl(\vec{b};\vec{c}\bigr).
\end{align}
\end{theorem}

\begin{remark} These identities generalize the identities given in \cite[Theorem~2]{do2009weil} which Do and Norbury show generalize the string and dilaton equations (due to Witten \cite{witten_conjecture}). 
\end{remark}

We conclude this subsection of the introduction by listing some explicit volumes of moduli spaces of flippered surfaces (\cref{thm:vol-fli-triang}, \cref{thm:vol-1-crown}, \cref{thm:vol-flip-2-crown}, \cref{thm:vol-flip-quadr}, \cref{prop:vol-5-gon-2-flip}, \cref{prop:vol-3-crown-formula}, \cref{prop:vol-flip-(1-1)-ann}): 
\begin{table}[ht]
\begin{center}
\begin{tabular}{ |c|c| } 
\hline
    $\Sigma$ & Generalised Mirzakhani Volume \\ \hline
$\mathbb{D}_3$ & $V_\Sigma(c_1,c_2,c_3)=\frac{2
}{\sqrt{\cosh^2c_1+\cosh^2c_2+\cosh^2c_3+2\cosh c_1 \cosh c_2 \cosh c_3-1}}$   \\ \hline
$\mathbb{D}_4$ & ${V}_\Sigma(c,c,0,0)= \frac{2}{1+\cosh c}$   \\ \hline
$\mathbb{D}_4$ & $V_\Sigma(c_1,c_2,0,0)= \frac{2\log\left(\frac{1+\cosh c_1}{1+\cosh c_2}\right)}{\cosh c_1- \cosh c_2}$, if $c_1\neq c_2$   \\ \hline
$\mathbb{D}_4$ & $V_\Sigma(c_1,c_2,c_1,c_2)= \frac{\log \left(\frac{1+\cosh(c_1+c_2)}{1+\cosh(c_1-c_2)}\right)}{\sinh c_1 \sinh c_2}$ \\ \hline
$\mathbb{D}_5$ & ${V}_\Sigma(c_1,c_2,0,0,0)=\frac{2\mathrm{Li}_2(1)+\log^2\frac{\cosh c_1+1}{\cosh c_2+1}+2\sum_{i=1,2}\mathrm{Li}_2\left(\frac{\cosh c_i-1}{\cosh c_i+1}\right)}{\cosh c_1+\cosh c_2}$ \\ \hline
$\mathbb{A}_1$ & $V_\Sigma(c|d)=\frac{1}{\sqrt{2\cosh c +2\cosh d}} $ \\ \hline
$\mathbb{A}_2$ & $V_\Sigma(c,0|d)=\sqrt{\frac{2}{\cosh c-\cosh d}}\text{arccot} \sqrt{\frac{1+\cosh d}{\cosh c-\cosh d}}$, if $c>d$  \\ \hline
$\mathbb{A}_2$ & $V_\Sigma(c,0|d)=\frac{1}{\cosh c/2}$, if $c=d$  \\ \hline
$\mathbb{A}_2$ & $V_\Sigma(c,0|d)=\sqrt{\frac{2}{\cosh d-\cosh c}}\text{arccoth}\sqrt{\frac{\cosh d+1}{\cosh d-\cosh c}}$, if $c<d$  \\ \hline
$\mathbb{A}_3$ & $V_{\Sigma}(c,\vec{0}|d)=\frac{2\text{Li}_2(1)+\frac{1}{2}\sum\log^2\left(\frac{\sqrt{\text{ch} c+\text{ch} d}+\sqrt{\text{ch} d+1}}{\sqrt{\text{ch} c+\text{ch} d}\pm\sqrt{\text{ch} d-1}}\right)+\sum_{1,\dotsc,4}\text{Li}_2\left(\frac{\sqrt{\text{ch} c+\text{ch} d}\pm\sqrt{\text{ch} d\pm1}}{\sqrt{\text{ch} c+\text{ch} d}\pm\sqrt{\text{ch} d\pm1}}\right)}{\sqrt{2(\text{ch} c+\text{ch} d)}}$ \\ \hline

$\mathbb{A}_{1,1}$ & $V_\Sigma(c,c)=\frac{\mathrm{Li}_2(-e^{-c})-\mathrm{Li}_2(-e^c)}{2\sinh c}$   \\ \hline
\end{tabular}
\caption{Some explicit examples of generalised Mirzakhani volumes.}
\label{table:volumes}
\end{center}
\end{table}

Explicit expressions do become quite unwieldy quite quickly, see e.g.: formulas for the volumes $V_{\DD_4}(c_1,c_2,c_3,c_4)$ and $V_{\AA_2}(c_1,c_2|d)$ (respectively \cref{thm:vol-flip-quadr} and \cref{thm:vol-flip-2-crown}).

\subsection*{Organisation of the paper}
 
\cref{sec:teich-moduli-spaces-flipsurf} sets up the geometry of flippered hyperbolic
surfaces: their anatomy, Teichm\"uller and moduli spaces, and the
flipper-clipping and Nielsen cut operations.
\cref{sec:gen-chekhov-action} defines the generalised Chekhov action and
establishes its recursive properties.
\cref{sec:WP-volume-form} constructs the Weil--Petersson volume form via
flipper-clipping and establishes the Nielsen cut recursion for it.
\cref{sec:volumes} proves the finiteness of the renormalised volumes, derives the three volume recursions and computes the volumes of simple surfaces. \cref{sec:structure} proves the string and dilaton-type equations,  determines the structure of the Mirzakhani volume coefficients and establishes that the volumes are periods.
\cref{sec:explicit-volumes} presents explicit computations of volumes for moduli spaces of some simpler surfaces. \cref{app:identity} 
through \cref{app:add-formula} support us with various facts and somewhat more technical computations used along the way.

\subsection*{AI declaration} Version 5.6 of OpenAI’s ChatGPT was used to explore proof ideas, test preliminary arguments, and identify possible references. The authors independently checked every source and argument, wrote the final text, and assume full responsibility for its content.

\subsection*{Acknowledgements}
The authors thank Anton Alekseev, Dylan Allegretti, Norman Do, Maksim Karev, Omar Kidwai, Paul Norbury, Kasra Rafi, and Nicolas Tholozan for useful discussions.
The second listed author gratefully acknowledges support from the Beijing Natural Science Foundation (International Scientists Project), Funding No. 1S24065.

\section{Teichm\"uller and Moduli spaces of flippered surfaces}
\label{sec:teich-moduli-spaces-flipsurf}

We familiarise ourselves with Teichm\"uller and moduli spaces of flippered surfaces: in this section we introduce parametrisations of these spaces, to prepare for upcoming volume calculations.

\subsection{Worldsheets of open strings as hyperbolic surfaces}

\begin{notation}
\label{notn:sigmasurface}
For the remainder of the article let $\Sigma:=\Sigma_{g,m,\vec{a}}$, with $m\in\mathbb{Z}_{\geqslant 0}$ and $\vec{a}=(a_1,\ldots,a_l)\in\mathbb{Z}_{>0}^l$, denote a bordered topological surface of finite type (that admits a complete hyperbolic metric) with
\begin{itemize}
\item genus $g$
\item 
$m$ interior punctures, which we label $\{p_1,\ldots,p_m\}$;
\item
$n=\sum_{k=1}^{l}a_k$ boundary punctures, which we label $\{q_1,\ldots,q_n\}$; 
\item
$\Sigma\cup\{q_1,\ldots,q_n\}$ is a surface with (only) interior punctures and $l$ boundary curves, which we label $\{\delta_1,\ldots,\delta_l\}$, so that the $i$-th boundary $\delta_i$ contains $a_i\geqslant 1$ punctures. 

\end{itemize}
\end{notation}

\begin{remark}
    The condition that $\Sigma$ admits a hyperbolic metric means:
    \begin{itemize}
        \item if $g=0$ and $l=0$, then $m\geqslant3$; 
        \item if $g=0$, $l=1$ and $m=0$, then $n\geqslant3$;
        \item if $g=1$ and $l=0$, then $m\geqslant1$;
        \item if $g\geq2$, there are no additional constraints on $l,m,n$.
    \end{itemize}
\end{remark}

By (canonically) extending geodesic boundaries with funnels, Mirzakhani's moduli spaces $\mathcal{M}_{g,m}(b_1,\ldots,b_m)$ is may be regarded as the moduli space of funnelled hyperbolic surfaces with \emph{fixed boundary holonomy}. This perspective, combined with the surface-doubling approach for studying open strings (see, e.g.: \cite{buryak2017matrix,pandharipande2014intersection}), naturally suggests flippered surfaces as an object of study.

\begin{definition}[Flippered surface]\label{defn:flippered}
We refer to $\Sigma$, equipped with a complete hyperbolic metric $h$ with (necessarily bi-infinite) geodesic boundaries, as a \emph{flippered surface}. We henceforth use $X=(\Sigma,h)$ to denote a generic flippered surface.
\end{definition}

\subsection{The anatomy of a flippered surface}

Although basic, we take a moment to affirm the uniqueness of geodesic representatives of homotopy classes of paths on hyperbolic surfaces.

\begin{theorem}[e.g.: {\cite[Theorems~1.5.2 and 1.5.3]{Buser}, \cite[Lemma~2.10]{huang_thesis}}]
\label{thm:orthogeodesic}
    Any essential isotopy class $[c]$ of paths on a hyperbolic surface $\Sigma$ contains a unique (up to parametrisation) shortest curve $\gamma$. Moreover, $\gamma$ is a geodesic, meets $\partial\Sigma$ perpendicularly (if at all), and is simple if and only if $[c]$ is simple.
\end{theorem}

\begin{definition}[arches, cuffs, flippers, straps, and tines]
\label{defn:anatomy}
We refer to the bi-infinite geodesic boundaries of $X$ as \emph{arches}. The following behaviour can hold at the punctures of a flippered surface $X$:
\begin{itemize}
    \item Interior punctures $p_i$ are geometrised as flares/funnels or cusps. For flares, $\beta_i$ denotes the geodesic representative of a simple loop around $p_i$ and we refer to $\beta_i$ as the $i$-th \emph{cuff}. For cusps, we use $\beta_i$ to denote the cusp $p_i$.
    \item A boundary puncture $q_j$ geometrises to ``half'' of a flare/funnel or cusp when $X$ is regarded as a subset of the doubled surface $DX$. For half-flares, the cuff corresponding to $q_j$ on $DX$ restricts to an orthogeodesic on $X$, and we refer to this orthogeodesic half-cuff as the $j$-th \emph{strap} $\sigma_j$. We refer to the half-funnel on $X$ bounded by $\sigma_j$ as the $j$-th \emph{flipper}. For half-cusps ---which we call a \emph{tine} \cite{huang_thesis}, we use $\sigma_j$ to formally denote the tine $q_j$.     
\end{itemize}
For the remainder of the article, we use $X$ to denote a (finite type) flippered hyperbolic surface (see \cref{fig:anatomy}), with 
\begin{itemize}
    \item $\{\beta_i\}_{i=1,\ldots,m}$ denoting the cuffs for its funnels (and cusps);
    \item $\{\alpha_j\}_{j=1,\ldots,n}$ denoting its boundary arches; 
    \item $\{\sigma_j\}_{j=1,\ldots,n}$ denoting its straps (and tines).
\end{itemize}
To simplify future notation, we label the straps so that $\sigma_i$ always joins $\alpha_i$ and $\alpha_{i+1}$, where the labelling index is cyclic with respect to the number of tines on crown containing these arches and straps.
\end{definition}

\begin{definition}[neck]
    Each boundary curve $\delta_i$ on $X\cup\{q_1,\ldots,q_n\}$ has a unique homotopy representative $\nu_i$ which is a simple closed geodesic on $X$. We refer to $\nu_i$ as the $i$-th \emph{neck} of $X$.
\end{definition}

\begin{notation}[disks and annuli]\label{notn:specialsurfaces}
    There are three classes of examples for which the nomenclature and notation in this paper applies awkwardly or with ambiguity:
    \begin{itemize}
    \item when $\Sigma$ is a disk with $n\geqslant 3$ boundary punctures, with no interior punctures (i.e.: $m=0$) and with no necks. We write $\mathbb{D}_n$ to denote such $\Sigma$, and refer to it as an \emph{$n$-gon}.
    
    \item when $\Sigma$ is disk with $n\geqslant 1$ boundary punctures and $m=1$ interior puncture, the surface contains a unique simple closed geodesic, and this geodesic serves as both the cuff of the interior puncture (or a cusp in the non-generic case) as well as the neck of the punctured boundary. We use the notation $\mathbb{A}_{n}$ to denote $\Sigma$, and refer to $\mathbb{A}_n$ as an \emph{$n$-tined crown}.

    \item when $\Sigma$ is a closed annulus with $a_1$ punctures on the first boundary and $a_2$ boundary punctures on the second boundary, the surface contains no interior punctures (and hence no cuffs). Moreover, the surface supports a unique simple closed geodesic which serves as a neck for both (punctured) boundaries. We use the notation $\mathbb{A}_{a_1,a_2}$ for such $\Sigma$ and refer to it as an \emph{$(a_1,a_2)$-annulus}.
    
    \end{itemize}
\end{notation}

\subsection{Teichm\"{u}ller and moduli spaces of flippered surfaces}
\label{sec:teichandmodspace}

We define the Teichm\"{u}ller and moduli space for flippered surfaces in the standard manner:

\begin{definition}[Teichm\"uller space]
The \emph{Teichm\"uller space} $\mathcal{T}_\Sigma$ is defined as the space of isotopy classes of flippered hyperbolic metrics on $\Sigma$.    
\end{definition}

\begin{proposition}[\cref{cor:topology-teich-space}] 
The Teichm\"uller space for a flippered hyperbolic surface $\mathcal{T}_\Sigma$ is naturally stratified into $2^{n+m}$ strata according to the geometric type of its punctures, described in \cref{defn:anatomy}.
Each stratum is a topological ball, and $\mathcal{T}_\Sigma$ has the structure of a manifold with corners of dimension $\frac{1}{2}(3|\chi(D\Sigma)|+n)$, where $D\Sigma$ is the surface obtained by doubling $\Sigma$ along its boundary. 
\end{proposition}

\begin{definition}[moduli space]
The \emph{moduli space} $\mathcal{M}_\Sigma$ is defined as the space of isometry classes of flippered hyperbolic surfaces.  
\end{definition}

The (pure) mapping class group $\mathrm{MCG}_\Sigma$, consisting of isotopy classes of puncture-preserving homeomorphisms of $\Sigma$, again serves as the group of deck transformations relating $\mathcal{T}_\Sigma$ and $\mathcal{M}_\Sigma$:
\[
\mathcal{M}_\Sigma=\mathcal{T}_\Sigma/\mathrm{MCG}_\Sigma.
\]

\begin{definition}[moduli spaces with fixed holonomy]
\label{def:fixed-holonomy}
For $\vec{b}=(b_1,\ldots,b_m)\in[0,\infty)^m$ and $\vec{c}=(c_1,\ldots,c_n)\in[0,\infty)^n$, define
\begin{align*}
\mathcal{T}_\Sigma(b_1,\ldots,b_m;c_1,\ldots,c_n)
&=:\mathcal{T}_\Sigma(\vec{b};\vec{c})
\subset\mathcal{T}_\Sigma
\quad\text{and}\\
\mathcal{M}_\Sigma(b_1,\ldots,b_m;c_1,\ldots,c_n)
&=:\mathcal{M}_\Sigma(\vec{b};\vec{c})
\subset\mathcal{M}_\Sigma,
\end{align*}
to respectively be the subsets of hyperbolic surfaces in $\mathcal{T}_\Sigma$ and $\mathcal{M}_\Sigma$ whose 
\begin{itemize}
    \item 
    $i$-th cuff $\beta_i$ is of length $b_i$ for all $i=1,\ldots,m$, and
    \item 
    $j$-th strap $\sigma_j$ is length $c_j$ for all $j=1,\ldots,n$.
\end{itemize}
Furthermore, given $\vec{d}=(d_1,\ldots,d_l)\in[0,\infty)^l$, define
\begin{align*}
\mathcal{T}_\Sigma(b_1,\ldots,b_m;c_1,\ldots,c_n|d_1,\ldots,d_l)
&=:\mathcal{T}_\Sigma(\vec{b};\vec{c}\;|\vec{d})
\subset\mathcal{T}_\Sigma(\vec{b};\vec{c})
\quad\text{and}\\
\mathcal{M}_\Sigma(b_1,\ldots,b_m;c_1,\ldots,c_n|d_1,\ldots,d_l)
&=:\mathcal{M}_\Sigma(\vec{b};\vec{c}\;|\vec{d})
\subset\mathcal{M}_\Sigma(\vec{b};\vec{c}).
\end{align*}
to respectively be the subsets of hyperbolic surfaces whose $k$-th neck lengths is $d_k$ for all $k=1,\ldots,l$.
\end{definition}

\begin{notation}[fixed holonomy for special surfaces]
When $\Sigma=\mathbb{D}_n$, $\mathbb{A}_{n}$, or $\mathbb{A}_{a_1,a_2}$, the above notation for Teichm\"uller and moduli spaces with fixed holonomy may not make sense due to the absence of cuffs and/or necks. Instead, we write:
    \begin{itemize}
        \item  $\mathcal{T}_{\mathbb{D}_n}(\vec{c})$ and $\mathcal{M}_{\mathbb{D}_n}(\vec{c})$;
        \item $\mathcal{T}_{\mathbb{A}_{n}}(\vec{c}\;|\vec{d})$ and $\mathcal{M}_{\mathbb{A}_{n}}(\vec{c}\;|\vec{d})$;
        \item $\mathcal{T}_{\mathbb{A}_{a_1,a_2}}(\vec{c};d)$ and $\mathcal{M}_{\mathbb{A}_{a_1,a_2}}(\vec{c};\vec{d})$.
    \end{itemize}
\end{notation}

\begin{remark}
    We show in \cref{cor:top-teich-space-fix-hol} that Teichm\"uller spaces 
    with fixed holonomy are open balls (of various dimensions, including $0$). In particular, $\mathcal{T}_\Sigma(\vec{b};\vec{c})$ embeds as a half-dimensional affine subspace in $\mathcal{T}_{D\Sigma}(b_1,\ldots,b_m,b_1,\ldots,b_m,2c_1,\ldots,2c_n)$, where $D\Sigma$ is the surface obtained from doubling $\Sigma$ along its boundary, and hence has dimension $\frac{1}{2}(3|\chi(D\Sigma)|-2m-n)$. The fixed neck length subset $\mathcal{T}_\Sigma(\vec{b};\vec{c}\;|\vec{d})$ embeds as a $l$ co-dimensional subset of $\mathcal{T}(\vec{b};\vec{c})$, i.e.: 
    \[
    \dim_{\mathbb{R}}\mathcal{T}_\Sigma(\vec{b};\vec{c}\;|\vec{d})
    =
     \tfrac{1}{2}(3|\chi(D\Sigma)|-2m-2l-n).
    \]
\end{remark}

The fact that doubling allows us to embed various Teichm\"uller spaces of flippered (and crowned) surfaces as canonical real analytic subspaces in Teichm\"uller spaces of hyperbolic bordered surfaces means:

\begin{proposition}
    The Teichm\"uller spaces $\mathcal{T}_\Sigma$, $\mathcal{T}_\Sigma(\vec{b};\vec{c})$, $\mathcal{T}_\Sigma(\vec{b};\vec{c}\;|\vec{d})$ and their corresponding moduli spaces all inherit natural and canonical real analytic manifold (possibly with corners) structure as subspaces of $\mathcal{T}_{D\Sigma}$
\end{proposition}

\subsection{Puncture-doubling and flipper clipper}

We define operations called \emph{puncture-doubling} and \emph{flipper-clipping}, which allows us to relate Teichm\"uller/moduli spaces of flippered and crowned hyperbolic surfaces.

\begin{notation}[doubling boundary punctures]
\label{not:doubling}
We write $\dot{\Sigma}(j_1,\dots,j_k)$ to denote the surface obtained from $\Sigma$ by \emph{doubling} the boundary punctures $q_{j_1},\dots, q_{j_k}$ as follows: for each $q_{j_i}$, we remove a small open disk from $\Sigma$ bounded by a contractible arc with endpoints at $q_{j_i}$ (see \cref{fig:puncturedouble}). We refer to $\dot{\Sigma}(j_1,\dots,j_k)$ as the \emph{puncture-double} of $\Sigma$ at $q_{j_1},\dots, q_{j_k}$.
\end{notation}

\begin{center}
\begin{figure}[ht!]
    \includegraphics[width=0.75\textwidth]{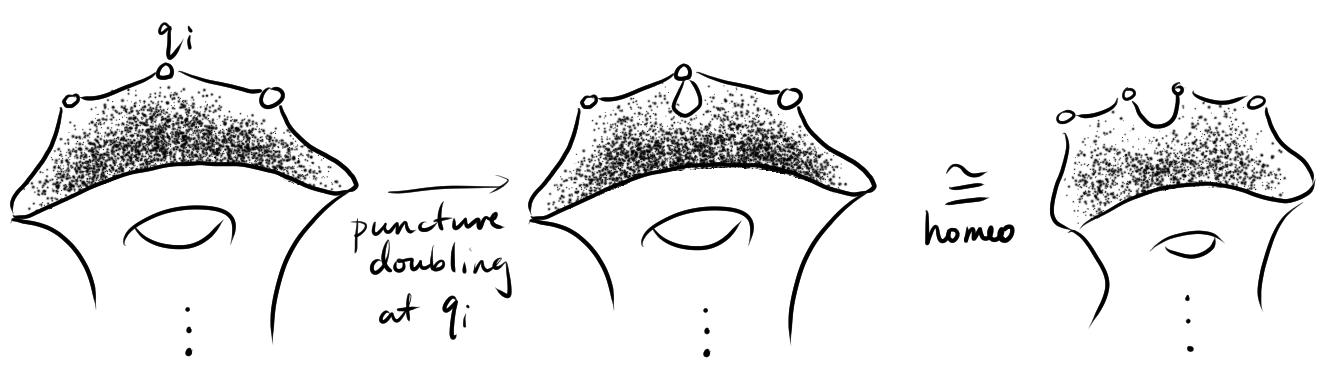}
\caption{Puncture doubling a surface at a boundary puncture labelled $q_i$ increases the number of punctures by $1$.}\label{fig:puncturedouble}
\end{figure}
\end{center}

\begin{definition}[flipper-clipping]
\label{def:flip-clip}
Given a flippered surface $X\in\mathcal{T}_{\Sigma}$, we produce a canonical crowned surface $\clip(X)\in\mathcal{T}_{\dot{\Sigma}(X)}$, where $\dot{\Sigma}(X)$ is the puncture-double of $\Sigma$ at (and only at) all flippered boundary punctures of $X$. The surface $\clip(X)$ is obtained as follows: for every flipper on $X$, remove the largest (by inclusion) open half-plane contained inside the flipper (see \cref{fig:flipper-clip}). We refer to $\clip(X)$ as the \emph{flipper-clipping} of $X$.
\end{definition}

\begin{figure}[H]
  \centering
\resizebox{1\textwidth}{!}{%
\begin{circuitikz}
\tikzstyle{every node}=[font=\Huge]

\draw[rotate=90][line width=3pt,short] (-11,26) .. controls (-4.75,20.25) and (4.75,20.25) .. (11,26);



\draw[rotate=90] [line width=3pt] (-11,5.5) .. controls (-4.75,11) and (4.75,11) .. (11,5.5);

\draw[rotate=90] [line width=3pt,short] (0,9.65) -- (0,21.65);

\draw[rotate=90] [line width=3pt,short] (-3.03+3,21.5-0.3) -- (-2.45+3,21.5-0.3);

\draw[rotate=90] [line width=3pt,short] (-2.5+3,21.5-0.3) -- (-2.5+3,22-0.3);

\draw[rotate=90] [line width=3pt,short] (-3.03+3,21.5-0.3-11) -- (-2.45+3,21.5-0.3-11);

\draw[rotate=90] [line width=3pt,short] (-2.5+3,21.5-0.3-11.5) -- (-2.5+3,22-0.3-11.5);

\node [font=\fontsize{50}{50}] at (-16,-8) {$X$};

\node [font=\fontsize{50}{50}] at (-16,1) {$\sigma_i$};

\node [font=\fontsize{50}{50}] at (-23,-2) {$\alpha_i$};

\node [font=\fontsize{50}{50}] at (-7.5,-2) {$\alpha_{i+1}$};

\end{circuitikz}
}%
   \hspace*{1in}
  \centering
\resizebox{1\textwidth}{!}{%
\begin{circuitikz}
\tikzstyle{every node}=[font=\Huge]

\draw[rotate=90][line width=3pt,short] (-11,26) .. controls (-4.75,20.25) and (4.75,20.25) .. (11,26);


\draw[rotate=90][blue][line width=3pt,short] (11,26) .. controls (6.25,20.5) and (6.25,10.75) .. (11,5.5);

\draw[rotate=90] [line width=3pt] (-11,5.5) .. controls (-4.75,11) and (4.75,11) .. (11,5.5);

\draw[rotate=90] [line width=3pt,short] (0,9.65) -- (0,21.65);

\draw[rotate=90] [line width=3pt,short] (-3.03+3,21.5-0.3) -- (-2.45+3,21.5-0.3);

\draw[rotate=90] [line width=3pt,short] (-2.5+3,21.5-0.3) -- (-2.5+3,22-0.3);

\draw[rotate=90] [line width=3pt,short] (-3.03+3,21.5-0.3-11) -- (-2.45+3,21.5-0.3-11);

\draw[rotate=90] [line width=3pt,short] (-2.5+3,21.5-0.3-11.5) -- (-2.5+3,22-0.3-11.5);

\node [font=\fontsize{50}{50}] at (-16,-8) {$\text{Clip}(X)$};

\node [font=\fontsize{50}{50}] at (-16,1) {$\sigma_i$};

\node [font=\fontsize{50}{50}] at (-23,-2) {$\alpha_i$};

\node [font=\fontsize{50}{50}] at (-7.5,-2) {$\alpha_{i+1}$};

\end{circuitikz}
}%
   \caption{Left: $i$-th flipper in $X$. Right: after clipping of the $i$-th flipper (in \textcolor{blue}{blue}), we obtain two tines.
   }
   \label{fig:flipper-clip}
\end{figure}
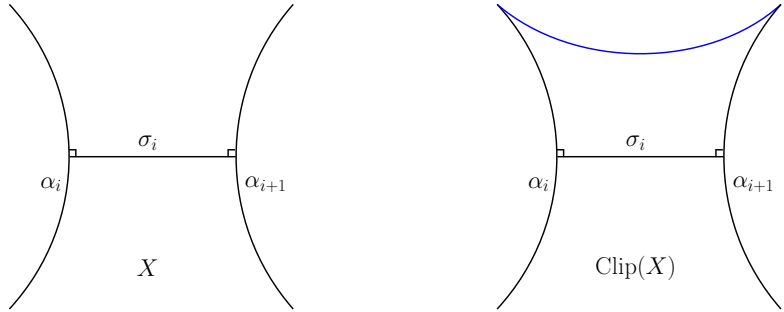

\begin{remark}
Flipper-clipping does not affect cusps. In particular, the flipper-clipping of a crowned hyperbolic surface is unchanged.  
\end{remark}

To remove the dependence on $X$ in the target Teichm\"uller space in \cref{def:flip-clip}, we highlight the following subspaces of $\mathcal{T}_\Sigma$:

\begin{definition}[flipper-based strata in Teichm\"uller space]
\label{def:T_Sigma(A)}
Given a subset $\mathcal{A}\subseteq\{1,\dotsc,n\}$, let $\mathcal{T}_\Sigma^\mathcal{A} \subset \mathcal{T}_\Sigma$ be the subset of hyperbolic surfaces in $\mathcal{T}_\Sigma$ whose $j$-th strap $\sigma_j$ has (strictly) positive length if and only if $j\in\mathcal{A}$, for all $j=1,\dotsc,n$.
\end{definition}

\begin{remark}
\label{rem:teich-fix-hol-subset}
   Observe that $\mathcal{T}_\Sigma(\vec{b};\vec{c}) \subset \mathcal{T}_\Sigma^\mathcal{A}$ when $\mathcal{A}$ is the set of indices for which $c_j$ is (strictly) positive. Since it is an open subset, it is also a real-analytic manifold (with corners).
\end{remark}

\begin{lemma}[clipping is bijective]
\label{lem:flip-clip}
    Given a subset $\mathcal{A}\subset\{1,\dotsc,n\}$, flipper-clipping defines a bijection
\[
\clip : \mathcal{T}_\Sigma^\mathcal{A} \to \mathcal{T}_{\dot{\Sigma}(\mathcal{A})}^\varnothing.
\]
\end{lemma}

\begin{proof}
The reverse process of attaching half-planes to specified arches is canonical and gives the inverse map to $\clip$.    
\end{proof}
In fact, it is intuitively clear that $\clip$ is a real-analytic homeomorphism (see \cref{cor:flip-clip-homeo} for a proof). Together with the construction in \cref{subsec:shears}, we obtain a parametrization of $\mathcal{T}_\Sigma^\mathcal{A}$ by shearing coordinates in \cref{prop:shearingcoords}.

\begin{remark}
    Although we have introduced flipper-clipping as a function $\clip : \mathcal{T}_\Sigma^\mathcal{A} \to \mathcal{T}_{\dot{\Sigma}(\mathcal{A})}^\varnothing$, we will abuse notation and use it to also denote maps between
    \begin{itemize}
        \item Teichm\"uller spaces with fixed holonomy $\clip:\mathcal{T}_\Sigma(\vec{b};\vec{c})\to\mathcal{T}_{\dot{\Sigma}}(\vec{b};\vec{0})$,
        \item Teichm\"uller spaces of fixed neck lengths $\clip:\mathcal{T}_\Sigma(\vec{b};\vec{c}\;|\vec{d})\to\mathcal{T}_{\dot{\Sigma}}(\vec{b};\vec{0}|\vec{d})$,
        \item moduli spaces $\clip:\mathcal{M}_\Sigma(\vec{b};\vec{c})\to\mathcal{M}_{\dot{\Sigma}}(\vec{b};\vec{0})$, and
        \item moduli spaces of fixed neck lengths $\clip:\mathcal{M}_\Sigma(\vec{b};\vec{c}\;|\vec{d})\to\mathcal{M}_{\dot{\Sigma}}(\vec{b};\vec{0}|\vec{d})$.
    \end{itemize}
\end{remark}

\subsection{Shearing coordinates on Teichm\"uller space of crowned surfaces}
\label{subsec:shears}

The contents of this subsection is well-known to experts.
See, for example, \cite[Section~7.4]{martelli2022introductiongeometrictopology} for a definition of the shearing parameter between two adjacent ideal hyperbolic triangles and the construction of shearing coordinates on Teichm\"uller space of surfaces with punctures.

\begin{notation}[shearing sign convention]
    We take \emph{left} shears to be positive.
\end{notation}

\begin{definition}[ideal arcs and ideal triangulation]
Given a surface $\Sigma=\Sigma_{g,m,\vec{a}}$, we refer to an embedded open arc in $\Sigma$ with ``endpoints'' placed at the interior and boundary punctures of $\Sigma$ as an \emph{ideal arc} if it does not bound an open disk in $\Sigma$. We call a collection $\triangle$ of pairwise disjoint ideal arcs in $\Sigma$ as an \emph{ideal triangulation} of $\Sigma$ if
\begin{itemize}
    \item no arc is peripheral in the sense of being isotopic to one of the arches,
    \item no two arcs in $\triangle$ are isotopic via isotopy fixing the endpoints, and
    \item $\triangle$ is maximal in cardinality among sets satisfying above conditions.
\end{itemize}
We identify ideal triangulations up to isotopies of $\Sigma$. 
\end{definition}

\begin{lemma}[{\cite[Lemma~2.13]{HT25}}]
Given a surface $\Sigma=\Sigma_{g,m,\vec{a}}$, 
with $m$ interior punctures and $l$ boundary curves with $n=\sum_{k=1}^l a_k$ boundary punctures so that there are $a_k$ punctures on the $k$-th boundary (see \cref{notn:sigmasurface}), every ideal triangulation $\triangle$ of $\Sigma$ has cardinality
$$
|\triangle|=\tfrac{1}{2}(3|\chi(D\Sigma)|-n),
$$  
where $D\Sigma$ is the surface obtained by doubling $\Sigma$ along its boundary. Furthermore, if the genus of $\Sigma$ is $g$, then
$$
|\triangle|=6g-6+3m+3l+n.
$$
\end{lemma}
Given a crowned hyperbolic surface $X \in \mathcal{T}_{\Sigma}^\varnothing$, there is a unique geodesic representative of $\triangle=\triangle(X)$ such that if some puncture $p_i$ is geometrised as a flare, then the arcs in $\triangle$ with $p_i$ as an endpoint geometrize to geodesics that spiral towards the $i$-th cuff  in the following way: from the perspective of a point on the cuff and looking away from the funnel, the spiralling geodesics pass to the left.

\begin{proposition}[shearing coordinates, see e.g. {\cite[Cor.~2.22]{huang_thesis}}]
\label{prop:shearingcoords}
Let  $\triangle$ be an ideal triangulation  of $\Sigma$, and let $s_i$ denote the shearing parameter associated to the $i$-th arc of $\triangle$. Then the map 
$$
s_{\triangle} : \mathcal{T}_{\Sigma}^\varnothing \to\RR^{\triangle} 
$$
is a real-analytic embedding into the convex polyhedral cone in $\RR^{\triangle} $ given by the following conditions for $i=1,\dotsc,m$: the collection of $\{s_j\}$ corresponding to arcs incident to $\beta_i$ sum (with multiplicity) to a non-negative number.
\end{proposition}

\subsubsection{Change of shearing coordinates under flips}
Let $\alpha$ be an ideal arc in an  ideal triangulation $\triangle$ of $\Sigma$. We say that an ideal triangulation $\triangle'$ is obtained from $\triangle$ by a \textit{flip} of $\alpha$ if $\triangle'\neq \triangle$, and there is an arc $\alpha'$ of $\triangle'$ such that $\triangle\setminus\{\alpha\}=\triangle'\setminus\{\alpha'\}$.

\begin{proposition}[\cite{pennerbook}]
\label{prop:flips}
Let $\triangle$ be an ideal triangulation of $\Sigma$. Suppose that $e\in \triangle$ separates two ideal triangles complementary to $\triangle$, where these two triangles have frontier arcs $a,b,e\in\triangle$ and $c,d,e \in \triangle$ in counterclockwise order. Perform a flip of $e$ to produce the ideal triangulation $\triangle'$. Provided $a,b,c,d$ are all pairwise distinct and in $\triangle$, then $s_{e'}=-s_{e}$ and 
\begin{equation}
\begin{split}
&s_{a'} = s_a+\log(1+e^{s_e}), 
\quad s_{b'} = s_b-\log(1+e^{-s_e}),
\\&
s_{c'} = s_c+\log(1+e^{s_e}), \quad
s_{d'} = s_d-\log(1+e^{-s_e}).
\end{split}
\end{equation}
\end{proposition}

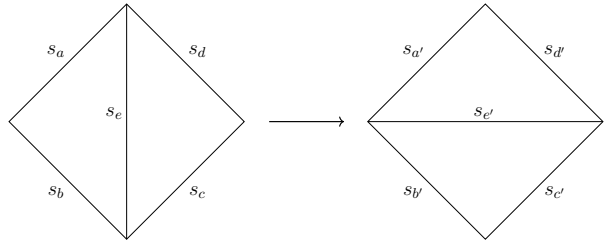
\begin{figure}[H]
    \centering
    \begin{tikzpicture}[scale=2]

\pgfmathsetmacro{\r}{sqrt(2)}

    \draw[rotate=45] (-1,-1) rectangle (1,1);
    
    \draw (0,-\r) -- (0,\r);

    \draw[->, thick] (\r+0.3,0) -- (\r+1.2,0);

    \begin{scope}[shift={(\r+2.9,0)}]
        \draw[rotate=45] (-1,-1) rectangle (1,1);
        \draw (-\r,0) -- (\r,0);
    \end{scope}


\node [font=\fontsize{14}{10}] at (-0.15,0.1) {$s_{e}$};

\node [font=\fontsize{14}{10}] at (-0.85,0.85) {$s_{a}$};

\node [font=\fontsize{14}{10}] at (0.85,0.85) {$s_{d}$};

\node [font=\fontsize{14}{10}] at (0.85,-0.85) {$s_{c}$};

\node [font=\fontsize{14}{10}] at (-0.85,-0.85) {$s_{b}$};

\node [font=\fontsize{14}{10}] at (4.3,0.1) {$s_{e'}$};

\node [font=\fontsize{14}{10}] at (4.3-0.85,0.85) {$s_{a'}$};

\node [font=\fontsize{14}{10}] at (4.3+0.85,0.85) {$s_{d'}$};

\node [font=\fontsize{14}{10}] at (4.3+0.85,-0.85) {$s_{c'}$};

\node [font=\fontsize{14}{10}] at (4.3-0.85,-0.85) {$s_{b'}$};
    
\end{tikzpicture}
    \caption{A flip of an arc.}
    \label{fig: flip}
\end{figure}

\begin{example}[Ideal pentagon]
\label{ex:pentagon}

Consider an ideal pentagon with the ideal triangulation as in \cref{fig: pentagon}, and the sequence of two flips as in \cref{fig: pentagon}. Then by \cref{prop:flips}:
\begin{equation}
\begin{split}
s'_2 = -s_2, \quad s_1' = \log\left(e^{s_1}+e^{s_1+s_2}\right),
\end{split}
\end{equation}
and 
\begin{equation}
\begin{split}
s''_1 = -\log\left(e^{s_1}+e^{s_1+s_2}\right), \quad s_2'' = \log\left(e^{s_1}+e^{-s_2}+e^{s_1-s_2}\right).
\end{split}
\end{equation}

\begin{figure}[H]
    \centering
    \begin{tikzpicture}[scale=1.2]

\pgfmathsetmacro{\r}{0}

\draw[thick]
    (270:2) --
    (342:2) --
    (54:2)  --
    (126:2) --
    (198:2) -- cycle;

\draw[thick] (54:2) -- (270:2);

\draw[thick] (126:2) -- (270:2);

\begin{scope}[shift={(6,0)}]
\draw[thick]
    (270:2) --
    (342:2) --
    (54:2)  --
    (126:2) --
    (198:2) -- cycle;

\draw[thick] (54:2) -- (270:2);

\draw[thick] (54:2) -- (198:2);
    
\end{scope}

\begin{scope}[shift={(12,0)}]
\draw[thick]
    (270:2) --
    (342:2) --
    (54:2)  --
    (126:2) --
    (198:2) -- cycle;

\draw[thick] (198:2) -- (342:2);

\draw[thick] (54:2) -- (198:2);
    
\end{scope}

\draw[->, thick] (2.3,0) -- (3.8,0);

\draw[->, thick] (2.3+6,0) -- (3.8+6,0);

\node [font=\fontsize{14}{10}] at (1,0) {$s_{1}$};

\node [font=\fontsize{14}{10}] at (-1,0) {$s_{2}$};

\node [font=\fontsize{14}{10}] at (0.3+6,0) {$s'_{1}$};

\node [font=\fontsize{14}{10}] at (-0.6+6,0.8) {$s'_{2}$};

\node [font=\fontsize{14}{10}] at (0+12,-0.2) {$s''_{1}$};

\node [font=\fontsize{14}{10}] at (-0.6+12,0.8) {$s''_{2}$};

\end{tikzpicture}
    \caption{Flips on ideal pentagon.}
    \label{fig: pentagon}
\end{figure}
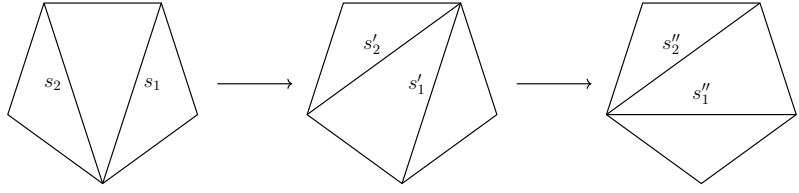
    
\end{example}

\subsection{Orthogeodesic coordinates}
\label{subsec:orthogeo-coord}

We introduce a class of \emph{global} coordinates on $\mathcal{T}_\Sigma$. These coordinates enjoy the property that Teichm\"uller subspaces with fixed holonomy $\mathcal{T}_\Sigma(\vec{b};\vec{c}) \subset \mathcal{T}_\Sigma$, are realised as (affine) coordinate subspaces. We refer to standard references like \cite[Section~7.3]{martelli2022introductiongeometrictopology} for the construction of Fenchel-Nielsen coordinates on Teichm\"uller space of surfaces with punctures.

\begin{definition}[essential arcs and maximal arc collection]
\label{def:gen-triang}
Given a surface $\Sigma=\Sigma_{g,m,\vec{a}}$, we refer to a properly embedded closed arc in $\Sigma$ with endpoints placed at the boundary $\partial \Sigma$ of $\Sigma$ as an \emph{essential arc} if 
\begin{itemize}
    \item no arc in $\square$, together with a subarc of the boundary, bound an open disk (see \cref{fig:forbidden-arcs}, left),
    \item no arc in $\square$,  together with two boundary subarcs (meeting at a puncture), bound an open disk (see \cref{fig:forbidden-arcs}, right).
\end{itemize}    
We call a collection $\square$ of disjoint essential arcs a \emph{maximal arc collection} on $\Sigma$ if
\begin{itemize}
    \item no two arcs in $\square$ are isotopic through properly embedded arcs, and
    \item $\square$ is maximal in cardinality among sets satisfying above conditions.
\end{itemize}
We identify maximal arc collections up to isotopies of $\Sigma$.
\end{definition}

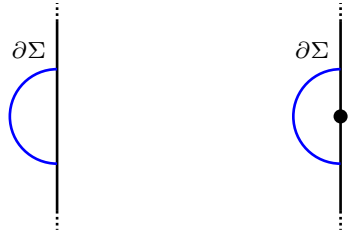
\begin{figure}[H]
    \centering
    \begin{tikzpicture}[
    thick,
    node/.style={circle, fill=black, inner sep=1.5pt},
    loop/.style={out=135, in=225, looseness=18}
]

\draw (0,0) -- (0,2);

\draw[densely dotted] (0,2) -- (0,2.2);

\draw[densely dotted] (0,0) -- (0,-0.2);




\node [font=\fontsize{7}{7}] at (-0.3,1.7) {$\partial\Sigma$};


\draw[blue] (0, 1.5) arc (90:270:0.5);

\draw (3,0) -- (3,2);

\draw[densely dotted] (3,2) -- (3,2.2);

\draw[densely dotted] (3,0) -- (3,-0.2);


\node[node] (mid) at (3,1) {};


\node [font=\fontsize{7}{7}] at (2.7,1.7) {$\partial\Sigma$};


\draw[blue] (3, 1.5) arc (90:270:0.5);

\end{tikzpicture}
    \caption{Two types of forbidden arcs (up to isotopy) in a maximal arc collection.}
    \label{fig:forbidden-arcs}
\end{figure}

\begin{remark}
    The condition that a properly embedded arc is essential is equivalent to requiring that the double of the arc on the double of $\Sigma$ (doubled along the boundary) is an essential (non-peripheral) closed curve.
\end{remark}

\begin{lemma}[complementary regions for maximal arc collections]
\label{lem:gen-triang-compl-regions}
Up to homeomorphism that sends arcs to arcs, there are five possible topological types of closures of complementary regions to a maximal arc collection $\square$ on $\Sigma$:
\begin{itemize}
    \item a pair of half-pants  (\cref{fig:gen-triang-compl-regions}, left-most), 
    \item a hexagon  (\cref{fig:gen-triang-compl-regions}, second starting from the left),
    \item a degenerate hexagon  (\cref{fig:gen-triang-compl-regions}, right-most three types).
\end{itemize}
\end{lemma}
\begin{proof}
By the maximality condition, the closure $\overline{C}$ of a complementary region $C$ to $\square$ in $\Sigma$ cannot have:
\begin{itemize}
    \item positive genus; 
    \item more than one boundary component (those can be connected by an arc); 
    \item or more than one interior puncture (otherwise, we can further separate a pair of half-pants from $\overline{C}$).
\end{itemize} Similarly, if $\overline{C}$ has exactly one interior puncture, then $\overline{C}$ is a pair of half-pants. 

Suppose that $\overline{C}$ contains no interior punctures and no boundary punctures. Note that the arcs and boundary arcs in $\overline{C}\setminus C$ alternate; by the first and third conditions in \cref{def:gen-triang}, $\overline{C}\setminus C$ contains least $3$ arcs; and if $\overline{C}\setminus C$ has at least $4$ arcs, the collection cannot be maximal. Hence $\overline{C}$ is a hexagon.

If $\overline{C}$ contains boundary punctures but not interior punctures, add a (non-essential) arc around each boundary puncture. Using a similar argument, they cut out a hexagon. Depending on the number of added arcs, we obtain the three types of degenerate hexagons. The last type appears only if it is the entire $\Sigma$.  
\end{proof}

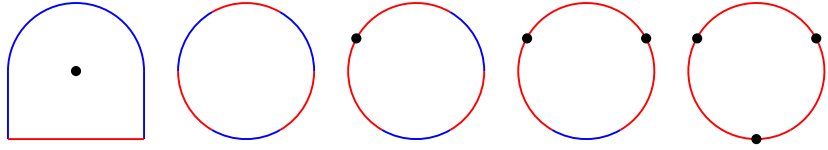
\begin{figure}[H]
    \centering
    \begin{tikzpicture}[
    thick,
    node/.style={circle, fill=black, inner sep=1.5pt},
    loop/.style={out=135, in=225, looseness=18}
]

\draw[red] (0,0) ++(60:1) arc[start angle=60, end angle=120, radius=1];

\draw[blue] (0,0) ++(120:1) arc[start angle=120, end angle=180, radius=1];

\draw[red] (0,0) ++(180:1) arc[start angle=180, end angle=240, radius=1];

\draw[blue] (0,0) ++(240:1) arc[start angle=240, end angle=300, radius=1];

\draw[red] (0,0) ++(300:1) arc[start angle=300, end angle=360, radius=1];

\draw[blue] (0,0) ++(360:1) arc[start angle=360, end angle=420, radius=1];

\draw[red] (2.5,0) ++(60:1) arc[start angle=60, end angle=180, radius=1];


\draw[red] (2.5,0) ++(180:1) arc[start angle=180, end angle=240, radius=1];

\draw[blue] (2.5,0) ++(240:1) arc[start angle=240, end angle=300, radius=1];

\draw[red] (2.5,0) ++(300:1) arc[start angle=300, end angle=360, radius=1];

\draw[blue] (2.5,0) ++(360:1) arc[start angle=360, end angle=420, radius=1];

\node[node] (mid) at (1.62,0.48) {};


\draw[red] (5,0) ++(60:1) arc[start angle=60, end angle=180, radius=1];


\draw[red] (5,0) ++(180:1) arc[start angle=180, end angle=240, radius=1];

\draw[blue] (5,0) ++(240:1) arc[start angle=240, end angle=300, radius=1];

\draw[red] (5,0) ++(300:1) arc[start angle=300, end angle=420, radius=1];


\node[node] (mid) at (4.13,0.48) {};

\node[node] (mid) at (5.88,0.48) {};


\draw[red] (7.5,0) ++(60:1) arc[start angle=60, end angle=180, radius=1];


\draw[red] (7.5,0) ++(180:1) arc[start angle=180, end angle=300, radius=1];


\draw[red] (7.5,0) ++(300:1) arc[start angle=300, end angle=420, radius=1];


\node[node] (mid) at (6.63,0.48) {};

\node[node] (mid) at (8.38,0.48) {};

\node[node] (mid) at (7.5,-1) {};


\draw[blue] (-2.5,0) ++(0:1) arc[start angle=0, end angle=180, radius=1];

\draw[blue] (-3.5,0) -- (-3.5,-1);

\draw[blue] (-1.5,0) -- (-1.5,-1);

\draw[red] (-3.5,-1) -- (-1.5,-1);

\node[node] (mid) at (-2.5,0) {};

\end{tikzpicture}
    \caption{Types of closures of complementary regions in a maximal arc collection on $\Sigma$. In \textcolor{blue}{blue}: arcs. In \textcolor{red}{red}: the boundary $\partial\Sigma$ of $\Sigma$.}
    \label{fig:gen-triang-compl-regions}
\end{figure}

\begin{lemma}
\label{lem:gen-triang-cardinality}
Given a surface $\Sigma=\Sigma_{g,m,\vec{a}}$, 
with genus $g$, $m$ interior punctures and $l$ boundary curves with $n=\sum_{k=1}^l a_k$ boundary punctures so that there are $a_k$ punctures on the $k$-th boundary (see \cref{notn:sigmasurface}), every maximal arc collection $\square$ on $\Sigma$ has cardinality
\[
|\square|=\tfrac{1}{2}
(3|\chi(D\Sigma)|-2m-n).
\]  
In particular, if the genus of $\Sigma$ is $g$, then
\[
|\square|=6g-6+2m+3l+n.
\]
\end{lemma}

\begin{proof}
The maximal arc collections  $\square$ and $\square'$ (on the opposite-oriented copy $\Sigma'$) combine to a pair of pants decomposition of $D\Sigma$ by \cref{lem:gen-triang-compl-regions}. Since the genus of $D\Sigma$ is $2g+l-1$ and the number of interior punctures is $n+2m$, every pants decomposition of $D\Sigma$ has
$$
3(2g+l-1)-3+(n+2m) = 6g-6+2m+3l+n
$$
simple closed curves. Therefore, 
\[
|\square| 
= 6g-6+2m+3l+n 
= \tfrac{1}{2}
(3|\chi(D\Sigma)|-2m-n).
\]
\end{proof}

Given a flippered hyperbolic surface $X \in \mathcal{T}_{\Sigma}$, there is a unique geodesic representative of $\square=\square(X)$ comprised (only) of essential orthogeodesics. 
As noted in the proof of \cref{lem:gen-triang-cardinality}, maximal arc collections double to pants decompositions of $D\Sigma$. We take advantage of the relationship between the orthogeodesic coordinates parametrization of the Teichm\"uller space of flippered surfaces $\mathcal{T}_{\Sigma}$ and the Fenchel--Nielsen parametrization of $\mathcal{T}_{D\Sigma}$ to show:

\begin{proposition}
\label{prop:orthogeod-coord}
Let $\square$ be a maximal arc collection on $\Sigma$. Let $\ell_i$ denote the length of the orthogeodesic representative of the $i$-th entry in $\square$. Then the map
\[
\mathcal{O}_\square :=(\vec{\ell},\vec{b},\vec{c}): \mathcal{T}_\Sigma \to \RR^{|\square|}_{>0} \times \RR^{n+m}_{\geqslant 0}
\]
is a real-analytic homeomorphism. 
\end{proposition}

\begin{proof}
Let $\{L_i,\tau_i\}_{i=1}^{|\square|},\{L_j^\partial\}_{j=1}^{n+2m}$ be the Fenchel-Nielsen coordinates associated with the pants decomposition given by $\square \cup \square'$ and the dual arc system given by $\partial  \Sigma = \{\alpha_1,\ldots,\alpha_n\}$. The Teichm\"uller space $\mathcal{T}_\Sigma$ sits as a subset of $\mathcal{T}_{D\Sigma}$ consisting of marked surfaces with a reflection automorphism taking $\square$ to $\square'$ whilst fixing $\alpha_1,\ldots,\alpha_n$. Explicitly, the subspace is given by constraining the twist parameters $\tau_i$ to be $0$, with the length parameters $L_i$ satisfying $L_i=2\ell_i$ such that 
\[
L_j^\partial=2c_j \,(j=1,\ldots,n), L_{n+j}^\partial=b_j\,(j=1,\ldots,m),L_{n+m+j}^\partial=b_j\,(j=1,\ldots,m).
\]
The claim that $\mathcal{O}_\square$ is a real-analytic diffeomorphism is immediate.
\end{proof}

\begin{corollary}
\label{cor:topology-teich-space}
The Teichm\"uller space for a flippered hyperbolic surface $\mathcal{T}_\Sigma$ is naturally stratified into $2^{n+m}$ strata according to the geometric type of its punctures, described in \cref{defn:anatomy}.
Each stratum is a topological ball, and $\mathcal{T}_\Sigma$ has the structure of a $\frac{1}{2}(3|\chi(D\Sigma)|+n)$-dimensional manifold with corners.
\end{corollary}
\begin{proof}
The partition of the closed ray $[0,\infty)$ as 
\[
[0,\infty)=\{0\}\sqcup(0,\infty)
\] 
induces the product partition of the non-negative orthant $[0,\infty)^{k}$ into $2^k$ balls of dimensions $0$ to $k$. Subsequently, we induce a partition of $\mathcal{T}_\Sigma$ into $2^{m+n}$ subsets, according to the vector $(\vec{b},\vec{c})\in [0,\infty)^{m}\times[0,\infty)^{n}$ of the cuff and strap lengths. Then by \cref{prop:orthogeod-coord}, each stratum is a topological ball, and $\mathcal{T}_\Sigma$ has the structure of a manifold with corners, and by \cref{lem:gen-triang-cardinality} it has dimension
$$
\dim_{\RR}\mathcal{T}_\Sigma = \tfrac{1}{2}
(3|\chi(D\Sigma)|-2m-n)+(n+m)=\tfrac{1}{2}(3|\chi(D\Sigma)|+n).
$$
\end{proof}

\begin{corollary}
\label{cor:top-teich-space-fix-hol}
The Teichm\"uller space with fixed holonomy $\mathcal{T}_\Sigma(\vec{b};\vec{c})$ embeds as a half-dimensional subspace in $\mathcal{T}_{D\Sigma}(b_1,\ldots,b_m,b_1,\ldots,b_m,2c_1,\ldots,2c_n)$,  and hence has dimension  $\frac{1}{2}(3|\chi(D\Sigma)|-2m-n)$. The fixed neck length subset $\mathcal{T}_\Sigma(\vec{b};\vec{c}\;|\vec{d})$ embeds as an $l$ co-dimensional subset of $\mathcal{T}(\vec{b};\vec{c})$ and hence is $\frac{1}{2}(3|\chi(D\Sigma)|-l-2m-n)$-dimensional.    
\end{corollary}

\begin{corollary}[clipping is real-analytic]
\label{cor:flip-clip-homeo}
Let $\mathcal{A}\subset\{1,\ldots,n\}$ be a set of indices for which $c_j$, $j\in\mathcal{A}$, is strictly positive. The map 
\[
\clip : \mathcal{T}_\Sigma^\mathcal{A} \to \mathcal{T}_{\dot{\Sigma}(\mathcal{A})}^\varnothing
\]    
is a real-analytic diffeomorphism.    
\end{corollary}
\begin{proof}
$\clip$ is the identity map with respect to the orthogeodesic coordinates for the spaces $\mathcal{T}_\Sigma^\mathcal{A}$ and $\mathcal{T}_{\dot{\Sigma}(\mathcal{A})}^\varnothing$.  
\end{proof}

\subsection{Converting between shearing and orthogeodesic parameters}

Generally speaking, it is taxing to convert from shearing parameters to orthogeodesic lengths, and vice versa. The simplest case when this is doable makes use of the following lemma:

\begin{lemma}
\label{lem:shear-width-quadrilateral}
Given an ideal hyperbolic quadrilateral labeled as per \figref{figure: shear and width of the quadrilateral}, the shearing parameter $s$ and the length $\eta$ of the depicted orthogeodesic are related by:
\[
e^s = \sinh^2 \left(\frac{\eta}{2}\right).
\]
\end{lemma}

\begin{proof}
Position the quadrilateral in the upper half-plane model $\HH$ so that its vertices are at $\infty, 0,1, e^s+1 \in \partial \HH$, and the diagonal is the blue vertical geodesic $\overline{1\infty}$ as in \cref{figure: shear and width of the quadrilateral}. Indeed, such quadrilateral has the correct shearing $s$. Project the vertices $1,e^s+1 \in \partial \HH$ onto the vertical geodesic $\overline{0\infty}$. The distance between the projections is $\ell = \ln (e^s+1)$. Standard calculations show: 
\[
e^s = e^{\ell}-1 = \cosh^2 \Bigl(\frac{\eta}{2}\Bigr) - 1 =  \sinh^2 \Bigl(\frac{\eta}{2}\Bigr).
\]
\end{proof}

\begin{figure}[H]
  \centering
\resizebox{1\textwidth}{!}{%
\begin{circuitikz}
\tikzstyle{every node}=[font=\Huge]
\draw [line width=3pt]  (0,15.75) circle (15cm);
\draw [line width=3pt,short] (-11,26) .. controls (-4.75,20.25) and (4.75,20.25) .. (11,26);
\draw [line width=3pt,short] (-11,26) .. controls (-6.25,20.5) and (-6.25,10.75) .. (-11,5.5);
\draw [line width=3pt,short] (11,26) .. controls (6.25,20.5) and (6.25,10.75) .. (11,5.5);
\draw [line width=3pt] (-11,5.5) .. controls (-4.75,11) and (4.75,11) .. (11,5.5);
\draw [line width=3pt,short] (0,9.65) -- (0,21.65);
\node [font=\fontsize{50}{50}] at (4,13.25) {$s$};
\node [font=\fontsize{50}{50}] at (1.5,17.75) {$\eta$};
\draw [line width=3pt,short] (-3.03+3,21.5-0.3) -- (-2.45+3,21.5-0.3);
\draw [line width=3pt,short] (-2.5+3,21.5-0.3) -- (-2.5+3,22-0.3);
\draw [line width=3pt,short] (-3.03+3,21.5-0.3-11) -- (-2.45+3,21.5-0.3-11);
\draw [line width=3pt,short] (-2.5+3,21.5-0.3-11.5) -- (-2.5+3,22-0.3-11.5);

\draw [blue,line width=3pt,short]  (-11,26) -- (11,5.5);

\end{circuitikz}
}%

   \hspace*{.5in}
  \begin{tikzpicture}[scale=4]


  \draw[thick] (-0.2,0) -- (1.75,0);

  \draw[very thick, black] (0,0) -- (0,1.6);
  \node[left] at (0,1.6) {$\infty$};
  \node[below] at (0,0) {$0$};

  \def\x{1.5} 


    \draw[very thick] (\x,0) arc (0:180:{(\x-1)/2});

    \draw[very thick] (1,0) arc (0:180:1/2);

    \draw[very thick, blue] (1,0) -- (1,1.6);

    \draw[very thick, black] (\x,0) -- (\x,1.6);

    \draw[thick, green] (0,1) arc (90:0:1);

    \draw[thick, green] (0,\x) arc (90:0:\x);

    \draw[very thick] (\x, 0.866) arc (90:150:0.866);

    
  \filldraw[black] (1,0) circle (0.02) node[below] {$1$};
  \filldraw[black] (\x,0) circle (0.02) node[below] {$e^s+1$};
  \filldraw[black] (0,0) circle (0.02) node[below] {$0$};


    \draw [decorate,decoration={brace,amplitude=10pt}]
    (1,0) -- (1.5,0) node [midway,yshift=15pt] {$e^s$};

    \draw [decorate,decoration={brace,amplitude=10pt}]
    (0,0) -- (1,0) node [midway,yshift=15pt] {$1$};

    \draw [decorate,decoration={brace,amplitude=10pt}]
    (0,1) -- (0,1.5) node [midway,xshift=-15pt] {$\ell$};

    \node[below] at (1.2,0.8) {$\eta$};

\end{tikzpicture}
   \caption{Left: Ideal quadrilateral with a diagonal and an orthogeodesic $\eta$. Right: a depiction of the same quadrilateral in the upper half-plane model; in \textcolor{green}{green}: the projection of one of the sides onto the opposite side.}
    \label{figure: shear and width of the quadrilateral}
\end{figure}
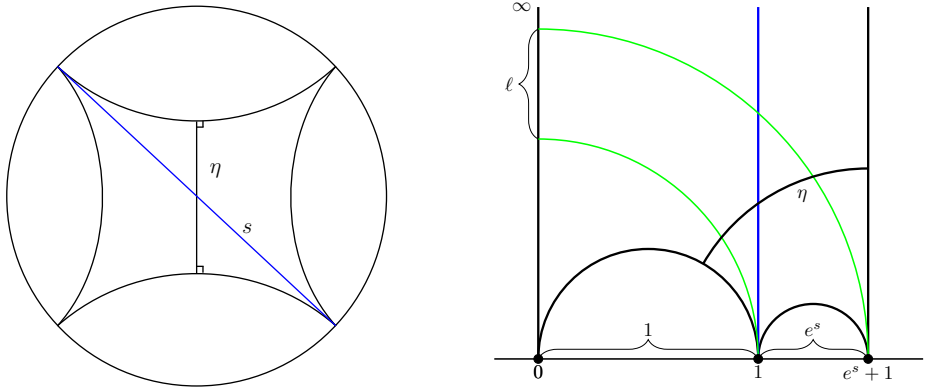

\section{Generalised Chekhov's action and its recursive structure}
\label{sec:gen-chekhov-action}

In \cite{chekhov2024}, Chekhov defines the following action functional on the decorated moduli space of crowns $S:\mathcal{M}^{\mathrm{dec}}_{\mathbb{A}_{n}}(\vec{0}|\vec{d})\rightarrow\mathbb{R}$,
\[
S = \kappa \log \left(\prod_{i=1}^n \lambda_{i}h_i \right),
\]
where $h_i$ is the length of the horocyclic arc at $i$-th tine, $\lambda_{i}$ is the lambda length of the $i$-th arch, and $\kappa>0$ is a ``coupling'' constant. Chekhov shows a posteriori that $S$ is decoration independent and hence descends to an action on the moduli space. In addition, Chekhov argues that letting $\kappa=1$ is the most natural choice \cite[Remark~2.5]{chekhov2024}. The goals of this section are as follows: we
\begin{itemize}
    \item give a geometric interpretation of the action functional (\cref{prop: essential intervals});
    \item use this interpretation as basis for generalizing Chekhov's action functional to flippered surfaces (\cref{defn: action for flippered surface});
    \item give first examples of these generalized Chekhov's actions by computing them for flippered triangles (\cref{subsec:action-triang})
    \item introduce topological recursion formulae for the generalized Chekhov's action (\cref{sec:chekhovactionrecursion}). 
\end{itemize}

\begin{remark}
\label{rmk:additive-action}
Chekhov implicitly asserts in \cite[Lemma~2.6]{chekhov2024} that the action for a crowned hyperbolic surface is equal to the sum of actions for all of its crowns. We will adopt this convention when we generalise the action.
\end{remark}

\subsection{Re-interpreting and generalising Chekhov's action}

\begin{definition}[orthojectories]\label{defn:orthojectory}
    An \emph{orthojectory} on a flippered hyperbolic surface $X$ is an \emph{oriented} geodesic on $X$ which orthogonally emanates from a point on $\partial X$. We further require an orthojectory to be \emph{inextensible} in the sense that it is not a subsegment of a longer geodesic on $X$. We refer to the starting point of an orthojectory as its \emph{base point}.
\end{definition}

\begin{definition}[peripheral and essential orthojectories]\label{def:orthojectories-types}
    An orthojectory $\rho$ on a flippered hyperbolic surface $X$ is called \emph{peripheral} if its entire image can be homotoped, via proper homotopies, arbitrarily deeply into a flipper/tine. Otherwise, it is called \emph{essential}.
\end{definition}

All orthojectories launched near a tine are \emph{extremely} short geodesic segments which join two adjacent arches and hence are peripheral. All orthojectories launched sufficiently deep within a flipper disappear off into the flipper and hence are also peripheral. The upshot is that non-peripheral orthojectories launch from a set of finite measure on $\partial X$.

\begin{definition}[essential interval]\label{defn:essentialint}
Let $\mathbb{I}_i\subset\alpha_i$ denote the smallest closed interval on $\alpha_i$ to contain all of the base points for essential orthojectories on the arch $\alpha_i$. 
\end{definition}

\begin{lemma}[essential interval endpoints]
\label{lem:ess-int-endpts}
Let $\alpha_i$ be the arch of a flippered surface $X$, and let $\tilde{\alpha}_i$ denote one of its lifts. Denote the adjacent arches of $\tilde{\alpha}_i$ on the universal cover $\tilde{X}\subset \HH$ by $\tilde{\alpha}'$ and $\tilde{\alpha}''$, and denote the ideal ends of $\tilde{\alpha}'$ and $\tilde{\alpha}''$ which are not in a shared flipper/tine with $\tilde{\alpha}_i$ by $\tilde{p}'$ and $\tilde{p}''$. Then, the orthogonal projectives of $\tilde{p}'$ and $\tilde{p}''$ onto $\tilde{\alpha}_i$ are the endpoints of the lift $\tilde{\mathbb{I}}_i$ of the essential interval to $\alpha_i$. Moreover, the essential interval $\mathbb{I}_i$ is precisely the set of all base points for essential orthojectories emanating from $\alpha_i$.
\end{lemma}

\begin{proof}
Note that an orthojectory properly homotopes into a flipper/tine if and only if its lift on the universal cover properly homotopes into the corresponding flipper/tine. Thus, $\tilde{\mathbb{I}}_i$ is equal to the essential interval of $\tilde{\alpha}_i$. Consider the orthojectories $\xi'$ and $\xi''$ respectively going from $\tilde{\alpha}_i$ and to $\tilde{p}'$ and $\tilde{p}''$ (see \cref{fig:gen-chekhov's-act}). The orthojectory $\xi'$ is essential, but all orthojectories  
on the side of $\xi'$ which is deeper in the flipper/tine are peripheral.
Thus, any such orthojectory is forced to go into the flipper or hits $\tilde{\alpha}'$, and hence is peripheral. Therefore, the basepoint for $\xi'$ gives one of the endpoints of $\tilde{\mathbb{I}}_i$. The same claim holds for $\xi''$, and the first half of the result follows.\medskip

To finish off, we show that $\tilde{\mathbb{I}}_i$ (and hence $\mathbb{I}_i$) precisely consists of the basepoints for all essential orthojectories via an argument by contradiction. Assume that some peripheral orthojectory $\hat{\xi}$ has base point in $\tilde{\mathbb{I}}_i$. Then, the proper homotopy pushing $\hat{\xi}$ into the relevant flipper/tine will push either $\xi'$ or $\xi''$ arbitrarily deeply into the flipper/tine, which is impossible.
\end{proof}

\begin{figure}[H]
    \centering
    \begin{tikzpicture}[scale=1.2]

  \draw[very thick]
    (0,2) ++(-30:1) arc[start angle=-30, end angle=-150, radius=1];

  \draw[very thick]
    (-2,0) ++(-80:1.2) arc[start angle=-80, end angle=70, radius=1.2];

  \draw[very thick]
    (2,0) ++(120:0.9) arc[start angle=120, end angle=250, radius=0.9];

  \draw[very thick]
    (-1.4,-2.5) ++(-10:1) arc[start angle=-10, end angle=110, radius=1];

\draw[thick, blue] (-1.79,-1.19) to [out=10,in=-100] (-0.1,1);

\draw (0,1.2) node {$\tilde{\alpha}_i$};

\draw (-1,1.2) node {$\tilde{p}'$};

\draw (1.2,1.2) node {$\tilde{p}''$};


\draw (0,-0.2) node {$\xi'$};


\draw[thick, red] [rotate around={-5:(-0.1,1)}] (-0.1,0.95) rectangle (-0.05,1);

\end{tikzpicture}
    \caption{An orthojectory $\xi'$ bounding one end of the essential interval on $\tilde{\alpha}_i$.}
    \label{fig:gen-chekhov's-act}
\end{figure}
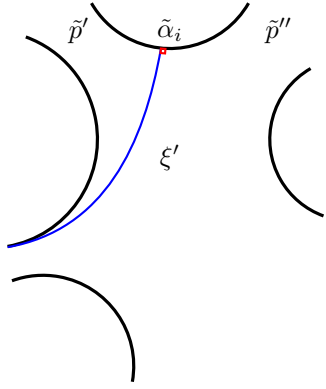

\begin{proposition}[Chekhov's action via essential intervals]
\label{prop: essential intervals}
The action functional $S$ for a crowned surface is equal to the half-sum of the lengths of all essential intervals:
\[
S = \frac{1}{2} \sum_{i=1}^n \ell (\mathbb{I}_i).
\]
\end{proposition}

\begin{proof}
Since Chekhov's action function for a crowned hyperbolic surface $X$ can be defined using any decoration, we set all decorating horocycles to have unit length, that is, $h_i = 1$ for $i=1,\dotsc,n$. In this setting, we have
\[S 
= \log \left( \prod_{i=1}^n \lambda_{i}h_i\right)
= \log \left( \prod_{i=1}^n \lambda_{i}\right).
\]
Since $\lambda_i$ is the exponentiation of the halved (signed) distance between decorating horocycles, the desired claim would follow immediately from showing that essential interval $\mathbb{I}_i$ for each arch $\alpha_i$ is precisely the geodesic segment adjoining the length $1$ decorating horocycles.\medskip

Consider a lift $\tilde{\alpha}_i$ of $\alpha_i$ to the universal cover $\tilde{X}\subset\HH$ of $X$, and denote the adjacent boundary arches by $\tilde{\alpha}'$ and $\tilde{\alpha}''$. We study the essential interval $\mathbb{I}_i\subset\alpha_i$ via its lift $\tilde{\mathbb{I}}_i\subset\tilde{\alpha}_i$ instead. \cref{lem:ess-int-endpts} tells us that $\tilde{\mathbb{I}}_i$ is bookended by the base points of the essential orthojectories $\xi'$ and $\xi''$ emanating from $\tilde{\alpha}_i$ and tending to the ideal endpoints of $\tilde{\alpha}'$ and $\tilde{\alpha}''$ which are not in a shared tine with $\tilde{\alpha}_i$.\medskip

At this point, we make the key observation that $\xi'$ and $\xi'$ each isolate half of an ideal triangle on $\tilde{X}$ (i.e. a triangle with angles $0,0,\pi/2$). In such triangles, a horocycle based at an ideal vertex and hitting the non-ideal vertex necessarily has length $1$ (see \cref{fig:5-crown} for a sample depiction with $n=i=5$). Therefore, the essential interval $\tilde{\mathbb{I}}_i$ is precisely the interval running between two length $1$ horocycles.
\end{proof}

\begin{remark}
    The cases when the crown containing $\alpha_i$ has $2$ or $1$ tines look different on the crowned surface. If there are $2$ tines, then the two orthojectories that bookend $\mathbb{I}_i$ intersect as per \cref{fig:2-tine-crow-ess-int}. If the crown containing $\alpha_i$ has exactly $1$ tine, the orthojectories themselves also fail to be \textit{embedded} (see \cref{fig:1-tine-crow-ess-int}). However, this does not affect the argument as we work on the universal cover. 
\end{remark}

\begin{figure}[H]
    \centering
    \begin{tikzpicture}

\draw (-2,-0.5) to [out=30,in=-60] (-1.1,2) to [out=-60,in=-120] (1.3,1.9) to [out=-120,in=165] (2.5,-0.3) to [out=165,in=90] (0.4,-2) to [out=90,in=30] (-2,-0.5);

\draw[densely dotted] (0.2,0.25) circle (0.5);

\draw [blue] (-0.9,1.15) arc (270:332:0.55);

\draw[red] (-2,-0.5) to [out=30,in=-115] (-0.4,1.43);

\draw[red] [rotate=-20] (-0.87,1.2) rectangle (-0.82,1.15);

\draw [blue] (1.15,1.15) arc (-90:-155:0.48);

\draw[red] (2.5,-0.3) to [out=165,in=-70] (0.7,1.42); 

\draw[red] [rotate=15] (1.04,1.18) rectangle (0.99,1.13);

\draw [blue] (1.63,-0.33) arc (-160:-220:0.48);

\draw [blue] (0,-1) arc (140:40:0.48);

\draw [blue] (-1,-0.39) arc (-20:63:0.48);

\draw (0.2,1.55) node {\scriptsize $\lambda_{5}$};

\draw (1.7,0.6) node {\scriptsize $\lambda_{4}$};

\draw (1.3,-0.9) node {\scriptsize $\lambda_{3}$};

\draw (-0.6,-0.8) node {\scriptsize $\lambda_{2}$};

\draw (-1.3,0.75) node {\scriptsize $\lambda_{1}$};

\draw (-0.7,1.4) node {\tiny $h_1$};

\draw (1,1.35) node {\tiny $h_5$};

\draw (1.8,-0.15) node {\tiny $h_4$};

\draw (0.35,-1.1) node {\tiny $h_3$};

\draw (-1.2,-0.15) node {\tiny $h_2$};

\end{tikzpicture}
    \caption{An essential interval on a crown with $\geq3$ tines, measured along a side of lambda length $\lambda_{5}$. It is bookended by orthojectories marked in {\color{red}red}.}
    \label{fig:5-crown}
\end{figure}
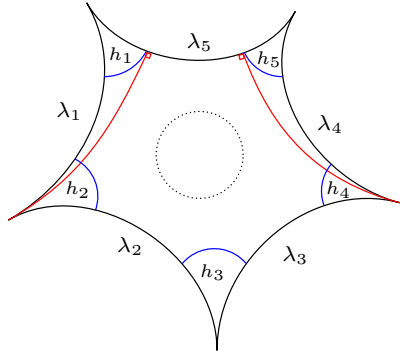

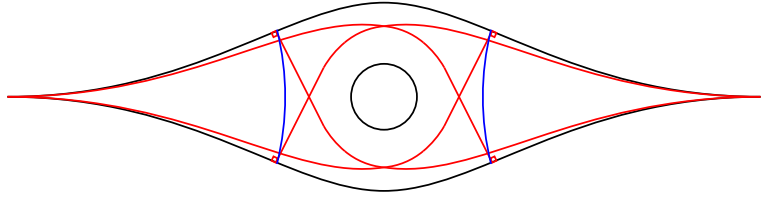
\begin{figure}[H]
    \centering
    \begin{tikzpicture}[scale=1.5]

\draw[thick] (-4,0) .. controls (-2,0) and (-1,1) .. (0,1);

\draw[thick] (4,0) .. controls (2,0) and (1,1) .. (0,1);

\draw[thick] (-4,0) .. controls (-2,0) and (-1,-1) .. (0,-1);

\draw[thick] (4,0) .. controls (2,0) and (1,-1) .. (0,-1);

\draw[thick] (0,0) circle (0.35);

\draw[thick, red] (-4,0) .. controls (-2,0) and (-1,0.9) .. (0,0.75);

\draw[thick, red] (0,0.75) arc(85:50:0.75) coordinate (A);

\draw[thick, red] (A) arc(50:30:1) coordinate (B);

\draw[thick, red] (B) -- (1.15,-0.7);

\draw[thick, red] (4,0) .. controls (2,0) and (1,0.9) .. (0,0.75);

\draw[thick, red] (0,0.75) arc(95:130:0.75) coordinate (C);

\draw[thick, red] (C) arc(130:150:1) coordinate (D);

\draw[thick, red] (D) -- (-1.15,-0.7);

\draw[thick, red] (4,0) .. controls (2,0) and (1,-0.9) .. (0,-0.75);

\draw[thick, red] (0,-0.75) arc(-95:-130:0.75) coordinate (E);

\draw[thick, red] (E) arc(-130:-150:1) coordinate (F);

\draw[thick, red] (F) -- (-1.15,0.7);

\draw[thick, red] (-4,0) .. controls (-2,0) and (-1,-0.9) .. (0,-0.75);

\draw[thick, red] (0,-0.75) arc(-85:-50:0.75) coordinate (G);

\draw[thick, red] (G) arc(-50:-30:1) coordinate (H);

\draw[thick, red] (H) -- (1.15,0.7);

\draw[thick, red] [rotate around={25:(1.15,-0.7)}] (1.15,-0.7) rectangle (1.2,-0.65);

\draw[thick, red] [rotate around={-25:(-1.15,-0.7)}] (-1.15,-0.7) rectangle (-1.2,-0.65);

\draw[thick, red] [rotate around={-25:(1.15,0.7)}] (1.15,0.7) rectangle (1.2,0.65);

\draw[thick, red] [rotate around={25:(-1.15,0.7)}] (-1.15,0.7) rectangle (-1.2,0.65);

\draw[thick, blue] (-4,0) ++(-14:2.95) arc[start angle=-14, end angle=14, radius=2.95];

\draw[thick, blue] (4,0) ++(166:2.95) arc[start angle=166, end angle=194, radius=2.95];

\end{tikzpicture}
    \caption{Essential intervals on a 2-tined crown, bookended by orthojectories marked in {\color{red}red}.}
    \label{fig:2-tine-crow-ess-int}
\end{figure}

\begin{figure}[H]
    \centering
    \begin{tikzpicture}[scale=1.5]

\draw[thick] (0,0) arc (-90:0:1) ;  

\draw[thick] (1,1) .. controls (1,2) and (0,3) .. (0,5);

\draw[thick] (0,0) arc (-90:-180:1) ;  

\draw[thick] (-1,1) .. controls (-1,2) and (0,3) .. (0,5);

\draw[thick] (0,1) circle (0.3);

\draw[thick, red] (0,5) .. controls (0,3) and (0.9,2) .. (0.8,1);

\draw[thick, red] (0.8,1) arc (-5:-95:0.75) coordinate (A);

\draw[thick, red] (A) arc (-95:-245:0.6) coordinate (B);

\draw[thick, red] (B) -- (0.8,1.9); 

\draw[thick, red] (0,5) .. controls (0,3) and (-0.9,2) .. (-0.8,1);

\draw[thick, red] (-0.8,1) arc (185:275:0.75) coordinate (C);

\draw[thick, red] (C) arc (275:425:0.6) coordinate (D);

\draw[thick, red] (D) -- (-0.8,1.9); 

\draw[thick, red] [rotate around={25:(0.755,1.91)}] (0.73,1.885) rectangle (0.78,1.935);

\draw[thick, red] [rotate around={-25:(-0.755,1.91)}] (-0.73,1.885) rectangle (-0.78,1.935);

\draw[thick, blue] (0,5) ++(-76:3.2) arc[start angle=-76, end angle=-104, radius=3.2];
\end{tikzpicture}
    \caption{The essential interval on a 1-tined crown, bookended by orthojectories marked in {\color{red}red}.}
    \label{fig:1-tine-crow-ess-int}
\end{figure}
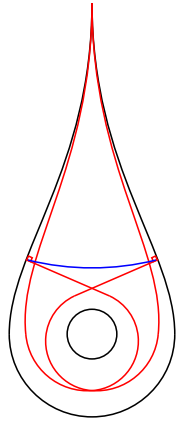

\subsubsection{Generalising Chekhov's action}

\begin{definition}[generalized action]
\label{defn: action for flippered surface}
Given \cref{prop: essential intervals}, one a geometrically natural generalisation of Chekhov's action is 
\begin{align}
S_0(X) := \frac{1}{2} \sum_{i=1}^n \ell (\mathbb{I}_i).
\end{align}
However, the following renormalisation 
\begin{align}
S(X) = \frac{1}{2} \sum_{i=1}^n \ell (\mathbb{I}_i)+\sum_{i=1}^n \log \cosh (c_i/2)
\end{align}
yields cleaner volume recursion formulae, so we define $S:\mathcal{M}_\Sigma\to\mathbb{R}$ to be the \textit{generalized Chekhov's action} for a flippered surface $X$. 
\end{definition}

\begin{remark}
Note that if $X$ is a crowned surface, then all $c_i=0$ and thus $S_0(X)=S(X)$. More generally, 
over the moduli space $\mathcal{M}_\Sigma(\vec{b};\vec{c}\;|\vec{d})$ (or $\mathcal{M}_\Sigma(\vec{b};\vec{c}))$
the difference $S_0(X)-S(X)$ is constant, therefore the generalized Mirzakhani volumes (\cref{def:gen-mirz-vol}) taken with respect to $S_0$ and $S$ differ by a multiplicative factor of $\prod_{i=1}^n \frac{1}{\cosh(c_i/2)}$.
\end{remark}

\begin{example}
We give explicit expressions for $S(X)$ with the simplest topologies, namely, $\mathbb{D}_3$ and $\mathbb{A}_1$. For $\mathbb{D}_3$ with straps of length $(c_1,c_2,c_3)\in[0,\infty)^3$, we show in \cref{subsec:action-triang}
\begin{itemize}
    \item if $c_1=c_2=c_3=0$, then $e^S=1$;
    \item if $c_2=c_3=0$, then $e^S=\cosh^2 (c_1/2)$;
    \item if $c_3=0$, then $e^S=\frac{1}{2}(\cosh c_1+\cosh c_2)$;
    \item the general formula is $e^S= \frac{1}{4}(\cosh c_1+\cosh c_2+\cosh c_3-1+\sqrt{D})$, where $D = \cosh^2 c_1 + \cosh^2 c_2 + \cosh^2 c_3 +2\cosh c_1 \cosh c_2 \cosh c_3 -1$.
\end{itemize}
For $\mathbb{A}_1$ with strap $c$ and neck length $d$, we show in \cref{sec:1crown}
\begin{itemize}
    \item if $c=0$, then $e^S = 2\cosh(d/2)$;
    \item the general formula is $e^S = \sqrt{\frac{\cosh d +\cosh c}{2}}+\cosh(d/2)$.
\end{itemize}
\end{example}

These examples serve as building blocks for expressing actions for general surfaces via ``topologically recursive'' formulae that we next establish.

\subsection{Topological recursion for Chekhov's action}
\label{sec:chekhovactionrecursion}

We use the term topological recursion in this context to mean that there is a method of expressing Chekhov's action functional $S(X)$ for a flippered surface using the action of topologically simpler surfaces which come from geometrically decomposing $X$. There are various manifestations of topological recursion, each arising from a different way of topologically decomposing a surface. Indeed, \cite[Lemma~2.6]{chekhov2024} already asserts the neck recursion stated in \cref{prop:neck} for moduli spaces of crowned hyperbolic surfaces. To state these results precisely, we introduce the notion of Nielsen cutting.

\begin{definition}[Nielsen cut]
\label{def:cut-orth}
Given an essential orthogeodesic $\eta$ on a flippered hyperbolic surface $X \in \mathcal{T}_\Sigma$, we write $X\|\eta$ to denote the (possibly disconnected) flippered hyperbolic surface obtained from $X\setminus\eta$ by attaching a flipper to each side of the cut. Since this is akin to cutting and followed by Nielsen extension \cite{bers1976nielsen}, we refer to the entire operation as a \emph{Nielsen cut}. Nielsen cuts with respect to disjoint essential orthogeodesics $\eta_1,\dotsc,\eta_k$ commute, and we adopt the notation $X\|\{\eta_1,\dotsc,\eta_k\}$ to refer to the (set of) surface(s) obtained from Nielsen cutting with respect to $\eta_1,\dotsc,\eta_k$.
\end{definition}

\subsubsection{Arch-to-arch Nielsen cutting}

\begin{theorem}[arch-to-arch cut]
\label{prop:action-arch-arch}
Let $\eta$ be an essential orthogeodesic on $X$ with endpoints on $\partial X$.  Then Chekhov's action $S(X)$ satisfies:
\begin{align}
e^{-S(X)} &= e^{-S(X\|\eta)}\cdot\sinh^2(\ell_\eta/2).
\end{align}
\end{theorem}

\begin{proof}
Suppose that the endpoints of $\eta$ are on distinct arches $\alpha_i$ and $\alpha_j$ of $X$. Denote by $\alpha_i',\alpha_i''$ the arches of $X\|\eta$ resulting from extending the two geodesic rays $\alpha_i\setminus\eta$ during the process of Nielsen cutting $X$ along $\eta$, and denote by $\mathbb{I}_i', \mathbb{I}_i''$ the respective essential intervals of $\alpha_i',\alpha_i''\subset X\|\eta$. 

In the universal cover $\tilde{X}\subset \HH$, let $\tilde{\eta}$ be a lift of $\eta$, and let $\tilde{\alpha}_i$ and $\tilde{\alpha}_j$ denote the respective lifts of $\alpha_i$ and $\alpha_j$ orthogonal to $\tilde{\eta}$. Let $\pi_{\tilde{\alpha}_i}(\tilde{\alpha}_j)$ denote the nearest point projection of $\tilde{\alpha}_j$ onto $\tilde{\alpha}_i$.  
Observe that $\pi_{\tilde{\alpha}_i}(\tilde{\alpha}_j) \subset \tilde{\mathbb{I}}_i$, and each of the endpoints of  $\pi_{\tilde{\alpha}_i}(\tilde{\alpha}_j)$ is one of the endpoints of either $\tilde{\mathbb{I}}'_i$ or $\tilde{\mathbb{I}}''_i$  (see \cref{fig:topol-rec-action}), and
\begin{align*}
\ell(\mathbb{I}_i) 
= 
\ell(\mathbb{I}_i') + \ell(\pi_{\tilde{\alpha}_i}(\tilde{\alpha}_j))
+\ell(\mathbb{I}_i'').  
\end{align*}
By the tri-rectangle identity (see, e.g.: \cite[Theorem~2.3.1]{Buser}), the length of $\pi_{\tilde{\alpha}_i}(\tilde{\alpha}_j)$ only depends on the length of $\eta$:
\begin{equation*}
\ell(\pi_{\tilde{\alpha}_i}(\tilde{\alpha}_j))
= 
2\log \coth \left(\ell_\eta/2\right).  
\end{equation*}
By symmetry, we also have $\ell(\pi_{\tilde{\alpha}_j}(\tilde{\alpha}_i)) = 2\log \coth \left(\ell_\eta/2\right)$. Observe that by \cref{lem:ess-int-endpts} the essential intervals of all other arches of $X$ are unaffected from cutting along $\eta$. Combined with \cref{prop: essential intervals} and \cref{defn: action for flippered surface}:
\begin{align*}
e^{-S(X)} =&  \frac{e^{-\frac{1}{2} \sum_{k=1}^n \ell (\mathbb{I}_k)} }{\prod_{i=1}^n\cosh(c_i/2)}
\\
=& \frac{e^{-\frac{1}{2} \sum_{k=1,k\neq i,j}^n \ell (\mathbb{I}_k)} \cdot e^{-\frac{1}{2}( \ell(\mathbb{I}_i')+\ell(\mathbb{I}_i'')+\ell(\mathbb{I}_j')+\ell(\mathbb{I}_j''))} \cdot e^{-\frac{1}{2}(\ell(\pi_{\tilde{\alpha}_i}(\tilde{\alpha}_j))+\ell(\pi_{{\tilde{\alpha}_j}}(\tilde{\alpha}_i)))} }{\prod_{i=1}^n \cosh(c_i/2)}  
\\
=& \frac{e^{-\frac{1}{2} \sum_{k=1,k\neq i,j}^n \ell (\mathbb{I}_k)} \cdot e^{-\frac{1}{2}( \ell(\mathbb{I}_i')+\ell(\mathbb{I}_i'')+\ell(\mathbb{I}_j')+\ell(\mathbb{I}_j''))} }{\prod_{i=1}^n \cosh(c_i/2)\cdot \cosh^2(\ell_\eta/2)} 
\\
&\cdot e^{-\frac{1}{2}(\ell(\pi_{\tilde{\alpha}_i}(\tilde{\alpha}_j))+\ell(\pi_{{\tilde{\alpha}_j}}(\tilde{\alpha}_i)))} \cdot \cosh^2(\ell_\eta/2)
\\
=&e^{-S(X\|\eta)}\cdot\sinh^2(\ell_\eta/2).
\end{align*}
Next, if the endpoints of $\eta$ are on the same arch $\alpha_i$, we similarly get
\begin{align*}
\ell(\mathbb{I}_i) = \ell(\mathbb{I}_i') + 2\ell(\pi_{\tilde{\alpha}_i}(\tilde{\alpha}_i^*))+\ell(\mathbb{I}_i''),    
\end{align*}
where $\tilde{\alpha}_i$ and $\tilde{\alpha}_i^*$ are both lifts of $\alpha_i$ connected by some lift $\tilde{\eta}$ of $\eta$. Supplanting the r\^{o}le of $\tilde{\alpha}_j$ with $\tilde{\alpha}_i^*$, the result follows by the same calculations.
\end{proof}

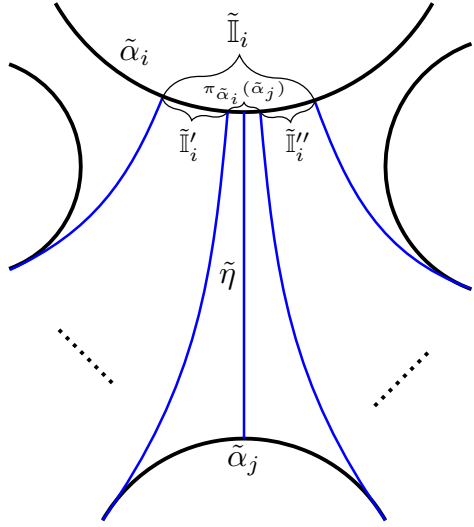
\begin{figure}[H]
    \centering
    \begin{tikzpicture}[scale=1.2]
\usetikzlibrary{decorations.pathreplacing}

  \draw[very thick]
    (0,3) ++(-30:2) arc[start angle=-30, end angle=-150, radius=2];

   \draw[very thick]
    (0,-3.5) ++(30:1.5) arc[start angle=30, end angle=150, radius=1.5]; 

  \draw[very thick]
    (-2.5,0.5) ++(-70:1) arc[start angle=-70, end angle=70, radius=1];

  \draw[very thick]
    (2.5,0.5) ++(110:1.2) arc[start angle=110, end angle=250, radius=1.2];

    (-1.4,-2.5) ++(-10:1) arc[start angle=-10, end angle=110, radius=1];

\draw[thick, blue] (-2.15,-0.45) to [out=25,in=-112] (-0.75,1.13);

\draw[thick, blue] (2.08,-0.62) to [out=160,in=-70] (0.65,1.1);

\draw[thick, blue] (-1.3,-2.75) to [out=55,in=-95] (-0.15,1);

\draw[thick, blue] (1.3,-2.75) to [out=125,in=-85] (0.15,1);

\draw (-1,1.55) node {$\tilde{\alpha}_i$};

\draw (0,-2.2) node {$\tilde{\alpha}_j$};




\draw (-0.15,-0.5) node {$\tilde{\eta}$};



\draw[thick, blue] (0,1) -- (0,-2);


\draw[very thick, dotted] (-1.7,-1) -- (-1.2,-1.5);

\draw[very thick, dotted] (1.2,-1.7) -- (1.7,-1.2);

   \draw [decorate,decoration={brace,amplitude=14pt}]
    (-0.75,1.13) -- (0.65,1.1) node [midway,yshift=22pt] {$\tilde{\mathbb{I}}_i$};

    \draw [decorate,decoration={brace,amplitude=3pt}]
    (-0.15,1) -- (0.15,1);

    \draw [decorate,decoration={brace,amplitude=5pt}]
    (0.65,1.1) -- (0.15,1);

    \draw [decorate,decoration={brace,amplitude=5pt}]
    (-0.15,1) -- (-0.75,1.13);

\node [font=\fontsize{5}{5}] at (0,1.18) {$\pi_{\tilde{\alpha}_i}(\tilde{\alpha}_j)$}; 

\node [font=\fontsize{8}{8}] at (-0.5,0.7) {$\tilde{\mathbb{I}}'_i$}; 

\node [font=\fontsize{8}{8}] at (0.5,0.7) {$\tilde{\mathbb{I}}''_i$};

\end{tikzpicture}
    \caption{Splitting of the essential interval $\mathbb{I}_i$.}
    \label{fig:topol-rec-action}
\end{figure}

\begin{corollary}
Let $\eta_1,\dotsc,\eta_k$ be disjoint orthogeodesics that constitute a maximal arc collection on $X$. Then Chekhov's action $S(X)$ satisfies:
\[
e^{-S(X)} = e^{-S(X\|\{\eta_1,\dotsc,\eta_k\})} \cdot \prod_{i=1}^k\sinh^2(\ell_{\eta_i}/2).
\]
\end{corollary}

\begin{corollary}
\label{cor:action-clip}
Let $X$ denote a flippered hyperbolic surface, and let $\mathrm{Clip}(X)$ be the flipper-clipping of $X$, then
\[
e^{-S(X)} = e^{-S(\text{Clip}(X))} \cdot \prod_{i=1}^n \coth^2\left(c_i/2\right).
\]    
\end{corollary}

\begin{proof}
Let $\sigma_i$ (of respective lengths $c_i$) denote the straps of $X$. The Nielsen cut $\mathrm{Clip}(X)\|\{\sigma_1,\ldots,\sigma_k\}$ consists of a copy of $X$ union $k$ copies of $\mathbb{D}_3$ with strap lengths $(c_i,0,0)$. From \cref{cor:action-flip-triang}, the contribution of the $\mathbb{D}_i$ term to $e^S(\mathrm{Clip}(X))$ is given by $\cosh^2(c_i/2)$, and thus
\[
e^{S(\text{Clip}(X))}\cdot \prod_{i=1}^n \sinh^2\left(c_i/2\right)
= 
e^{S(X)}\cdot \prod_{i=1}^n \cosh^2\left(c_i/2\right).
\]
The result follows from rearrangement.
\end{proof}

\subsubsection{Arch-to-cuff Nielsen cutting}
Let $\eta$ be a simple orthogeodesic on a flippered hyperbolic surface $X$ with one endpoint on the arch $\alpha_j$ and the other endpoint on the cuff $\beta_i$. Denote by $X_{\beta_i}^{\text{nf}}$ the closure of the non-flare component of $X\setminus \beta_i$. Further, denote by $\beta_i^\eta$ the arch of $X_{\beta_i}^{\text{nf}}\|\eta$ obtained from extending the geodesic $\beta_i\setminus\eta$ during the process of Nielsen cutting $X_{\beta_i}^{\text{nf}}$ along $\eta$, and denote by $\mathbb{I}_{\beta_i}^\eta$ the respective essential interval. We prove:

\begin{theorem}[arch-to-cuff]
\label{prop:action-arch-cuff}
Let $\eta$ be a simple orthogeodesic on $X$ with one endpoint on the arch $\alpha_j$ and the other endpoint on the cuff $\beta_i$. Then Chekhov's action $S(X)$ satisfies:
\[
e^{-S(X)} = \frac{1}{2}e^{-S(X_{\beta_i}^{\text{nf}}\|\eta)} \cdot e^{\ell (\mathbb{I}_{\beta_i}^\eta)/2} \cdot \sinh\ell_\eta.
\]
\end{theorem}

\begin{proof}
Note that $\eta$ is non-separating, since one of its endpoints is on a loop. As in the proof of \cref{prop:action-arch-arch}, denote by $\alpha_j',\alpha_j''$ the arches of $X_i\|\eta$ resulting from extending the two geodesic rays $\alpha_i\setminus\eta$ during the process of Nielsen cutting $X_i$ along $\eta$, and denote by $\mathbb{I}_j', \mathbb{I}_j''$ the respective essential intervals of $\alpha_j',\alpha_j''\subset X_i\|\eta$. 

By the same argument as in the proof of \cref{prop:action-arch-arch}, 
\begin{align*}
\ell(\mathbb{I}_j) 
= 
\ell(\mathbb{I}_j') + \ell(\pi_{\tilde{\alpha}_j}(\tilde{\beta}_i))
+\ell(\mathbb{I}_j'').  
\end{align*}
Note that $\beta_i$ is not an arch of $X$ and hence does not carry an essential interval, while $\beta_i^\eta$ is an arch of $X_{\beta_i}^{\text{nf}}\|\eta$ and carries the essential interval $\mathbb{I}_{\beta_i}^\eta$. Further, by \cref{lem:ess-int-endpts} the essential intervals of all the other arches of $X$ are unaffected by cutting along $\eta$. Combining these observations with \cref{prop: essential intervals} and \cref{defn: action for flippered surface}, we obtain:
\begin{align*}
e^{-S(X)}  = &  \frac{e^{-\frac{1}{2} \sum_{k=1}^n \ell (\mathbb{I}_k)}}{\prod_{i=1}^n\cosh(c_i/2)} 
\\=& \frac{e^{-\frac{1}{2} \sum_{k=1,k\neq j}^n \ell (\mathbb{I}_k)} \cdot e^{-\frac{1}{2}( \ell(\mathbb{I}_j')+\ell(\mathbb{I}_j''))} }{\prod_{i=1}^n \cosh(c_i/2)\cdot \cosh^2(\ell_\eta/2)}
\cdot e^{-\frac{1}{2}(\ell(\pi_{\tilde{\alpha}_j}(\tilde{\beta}_i)))}\cdot \cosh^2(\ell_\eta/2)
\\=&
  \frac{1}{2}e^{-S(X_{\beta_i}^{\text{nf}}\|\eta)} \cdot e^{\ell (\mathbb{I}_{\beta_i}^\eta)/2} \cdot \sinh\ell_\eta.
\end{align*}
\end{proof}

\section{Weil--Petersson volume form and its recursive structure}
\label{sec:WP-volume-form}

The \emph{Weil--Petersson form} \cite{weil} is a mapping class group invariant symplectic form on the Teichm\"uller space $\mathcal{T}_{g,m}(\vec{b})$, and induces a symplectic form on $\mathcal{M}_{g,m}(\vec{b})$. Wolpert famously showed \cite{wolpertwpform} that Fenchel--Nielsen coordinates yield \emph{global} Darboux coordinates. Mirzakhani showed that the volume of the moduli space with respect to The top exterior product of the Weil--Petersson form is a rational polynomial in $\mathbb{Q}_{>0}[\pi^2,b_1^2,\ldots,b_m^2]$ and encodes intersection numbers of Chern classes of certain natural line bundles over moduli space.\medskip

In contrast, Teichm\"uller spaces of flippered surfaces (either with or without fixed boundary holonomy) may be odd-dimensional and often fail to be symplectic. Nevertheless, we will use Thurston's shearing coordinates to define a generalization of the top exterior product (i.e.: the \emph{Weil--Petersson volume form}) on Teichm\"uller spaces of  flippered surfaces $\mathcal{T}_\Sigma(\vec{b};\vec{c})$ and $\mathcal{T}_\Sigma(\vec{b};\vec{c}\;|\vec{d})$ which
\begin{itemize}
    \item are mapping class group invariant (\cref{eq:vol-form-teich-fixed-hol});
    \item generalize the Weil--Petersson volume form for crowned hyperbolic surfaces introduced in \cite{HT25};
    \item satisfy topologically recursive formulae for the volume forms (\cref{prop:top-rec-vol-form-orthogeod}) analogous to those employed by Mirzakhani in \cite{mirz_simp}.
\end{itemize}

\subsection{Defining the WP volume form for flippered surfaces}
\label{subsec:flippervolform}

Our primary goal is to define Weil--Petersson volume forms (as opposed to symplectic forms) for $\mathcal{T}_\Sigma(\vec{b};\vec{c})$ as well as for $\mathcal{T}_\Sigma(\vec{b};\vec{c}\;|\vec{d})$. We do this via the following steps:
\begin{description}
    \item[Step~1] we partition $\mathcal{T}_\Sigma$ into subsets which naturally identify (via the flipper clipping map) with Teichm\"uller spaces of \emph{crowned} hyperbolic surfaces.
    \item[Step~2] we pullback the Weil--Petersson volume form for Teichm\"uller spaces of crowned hyperbolic surfaces to the various strata in $\mathcal{T}_\Sigma$.
    \item[Step~3] $\mathcal{T}_\Sigma(\vec{b};\vec{c})$ and $\mathcal{T}_\Sigma(\vec{b};\vec{c}\;|\vec{d})$ are subsets in these strata, and we divide the Weil--Petersson volume form by particular choices of differentials to yield volume forms on these Teichm\"uller (sub)spaces.
\end{description}

\subsubsection{WP volume form on the moduli space of crowned surfaces}

Let us first review how one is able to define a mapping class group-invariant volume form on Teichm\"uller spaces of crowned hyperbolic surfaces (c.f. \cite[Lemma~2.8]{goncharov-sun}). 

\begin{definition}[Weil--Petersson volume for spaces of crowned surfaces]\label{defn:wpform}
Given an ideal triangulation $\triangle$ of $\Sigma$,
let $\{s_i\}_{i=1}^{|\triangle|}$ denote the associated shearing coordinates on $\mathcal{T}_\Sigma^\varnothing$ (see \cref{subsec:shears}), define the volume form $\Omega^{\mathrm{WP},\varnothing}_{\Sigma}$ on $\mathcal{T}_\Sigma^\varnothing$ as:
\begin{align}
\label{eq:vol-form-crowned-surf}  \Omega^{\mathrm{WP},\varnothing}_{\Sigma} = \pm \bigwedge_{i=1}^{|\triangle|}\mathrm{d}s_i,
\end{align}
where the sign is taken to output positive volumes (for a chosen top-dimensional chain).
\end{definition}

\begin{remark}[sign of the Weil--Petersson volume form]
The Weil--Petersson volume as defined above is non-canonical as there two possible choices of orientation for any top dimensional set. However, once an initial choice is made (say, for some fundamental domain of the mapping class group action), this then defines a legitimate volume form. A more explicit approach (see \cite{goncharov-sun}) might be to ``twist'' the volume form by a $\mathbb{Z}/2\mathbb{Z}$-torsor which (effectively) counts the parity of the number of flips and transpositions (on interior edges) taken to change between triangulations.
\end{remark}

\begin{remark}[mapping class group invariance]
The proof that $\Omega^{\mathrm{WP},\varnothing}_{\Sigma}$ is independent of the choice of the ideal triangulation (and hence mapping class group invariant) is essentially due to Penner's work: one first verifies that it is preserved under flips (\cite[Theorem~4.7]{pennerbook}), and then show that you can deform from any ideal triangulation to another. To see this, note that the underlying graph in the dual ``fatgraph'' of $\Sigma$ is a fatgraph in the classical sense, but with some half-edges attached to edges on cycles corresponding to boundary components. Observe that we can slide the half-edges to any edge on a boundary cycle (whilst preserving the fatgraph structure) via Whitehead moves. Then, we can deform between arbitrary fatgraphs (of the same topological surface) by first ``ignoring'' the half-edges and invoking \cite[Proposition~7.1]{pennervolume} to perform Whitehead moves to deform between the two fatgraphs (without the half-edges). This is possible even with the half-edges present, as we can always slide a half-edge away from the particular edge that we want to apply a Whitehead move to. After deforming from the initial fatgraph to the final modulo half-edges, we can then slide all half-edges into the desired position to complete the process.
\end{remark}

\subsubsection{Weil--Petersson volume forms on the flipper-based strata}
Given $\mathcal{A}\subseteq\{1,\ldots,n\}$, recall the following:
\begin{itemize}
    \item $\mathcal{T}_\Sigma^\mathcal{A} \subset \mathcal{T}_\Sigma$ (see \cref{def:T_Sigma(A)}) is the stratum of flippered surfaces for which $c_j$, $j\in\mathcal{A}$ are strictly positive (and all other $c_j$ are equal to $0$);
    \item the flipper-clipping map $\mathrm{Clip}:\mathcal{T}^{\mathcal{A}}_\Sigma\to\mathcal{T}^\varnothing_{\dot{\Sigma}(\mathcal{A})}$ is a real-analytic diffeomorphism (\cref{cor:flip-clip-homeo}).
\end{itemize}

\begin{definition}[Weil--Petersson volume form for strata]

We define the \emph{Weil--Petersson volume form} for the stratum $\mathcal{T}_\Sigma^\mathcal{A}$ as the following pullback:
\begin{equation}
\label{eq:vol-form-strata}
\Omega^{\mathrm{WP},\mathcal{A}}_{\Sigma} := \clip^*\left(\Omega^{\mathrm{WP},\varnothing}_{\dot{\Sigma}(\mathcal{A})}\right),
\end{equation}
where $\Omega^{\mathrm{WP},\varnothing}_{\dot{\Sigma}(\mathcal{A})}$ is the Weil--Petersson volume form on $\mathcal{T}_{\dot{\Sigma}(\mathcal{A})}^\varnothing$ (see \cref{eq:vol-form-crowned-surf}).
    
\end{definition}

\subsubsection{Weil--Petersson volume form for $\mathcal{T}_\Sigma(\vec{b};\vec{c})$ and $\mathcal{T}_\Sigma(\vec{b};\vec{c}\;|\vec{d})$}

Observe that $\mathcal{T}_\Sigma(\vec{b};\vec{c})$ is a subset of $\mathcal{T}_\Sigma^\mathcal{A}$ when $\mathcal{A}$ is the set of indices for which $c_j$ is (strictly) positive (\cref{rem:teich-fix-hol-subset}). Specifically, set $\mathcal{A} = \{j_1,\dotsc,j_k\}\subseteq \{1,\dotsc,n\}$, and define the map 
\begin{align}
f_\mathcal{A}:\mathcal{T}_\Sigma^\mathcal{A} &\to \RR^{|\mathcal{A}|}  \times \RR^m_{\geqslant 0}\notag \\
\label{eq:map f_A}
X &\mapsto \left(2\log\sinh\left(\tfrac{\ell({\sigma_{j_1})}}{2}\right),\dotsc,2\log\sinh\left(\tfrac{\ell({\sigma_{j_k})}}{2}\right);\ell(\beta_1),\dotsc,\ell(\beta_m)\right).
\end{align}
Then, $f_{\mathcal{A}}$ is a bundle map where the fibers are precisely given by 
\[
\mathcal{T}_\Sigma(\vec{b};\vec{c}) = f^{-1}_{\mathcal{A}}\left(\log \sinh^2\left(\frac{c_{j_1}}{2}\right),\dotsc,\log \sinh^2\left(\frac{c_{j_k}}{2}\right); b_1,\dotsc, b_m \right).
\]

\begin{notation}[division by differential forms]
We use ``$/\mu$'', where $\mu$ is a $k$-form, to denote division by differential forms in the sense of \cite[6.17.21]{MR350769} (see also \cref{subsec:submersions} for a local reference).
\end{notation}

In order to divide by a differential form, we have as input:
\begin{itemize}
    \item a surjective submersion $f:M\to N$, 

    \item a smooth top-dimensional form $\omega$ on $M$,

    \item and a smooth non-vanishing top-dimensional form $\mu$ on $N$.
\end{itemize}
For each level set $f^{-1}(y)$ of $f$, we obtain a top-dimensional form $\omega/\mu$ satisfying
\[
\omega = (f^*\mu)\wedge(\omega/\mu)\text{ on }f^{-1}(y).
\]

\begin{definition}[WP volume form for $\mathcal{T}_\Sigma(\vec{b};\vec{c})$]
\label{eq:vol-form-teich-fixed-hol}
We define the Weil--Petersson volume form $\Omega^{\mathrm{WP}}_{\Sigma}(\vec{b};\vec{c})$ on $\mathcal{T}_\Sigma(\vec{b};\vec{c})$ as
\begin{align*}
\Omega^{\mathrm{WP}}_{\Sigma}(\vec{b};\vec{c}) := \Omega^{\mathrm{WP},\mathcal{A}}_{\Sigma}/\mu_{\mathcal{A}}, 
\end{align*}
where 
\begin{itemize}
    \item $f=f_{\mathcal{A}}:\mathcal{T}^{\mathcal{A}}_\Sigma\to\RR^{|\mathcal{A}|}  \times \RR^m_{\geqslant 0}$ is the input submersion (see \cref{prop:orthogeod-coord}),
    \item $\omega = \Omega^{\mathrm{WP},\mathcal{A}}_{\Sigma}$ is the Weil--Petersson volume form on $\mathcal{T}^{\mathcal{A}}_\Sigma$,
    \item and $\mu=\mu_{\mathcal{A}}$ is the Euclidean volume element on $\RR^{|\mathcal{A}|}\times\RR^m_{\geqslant 0}$. 
\end{itemize}
\end{definition}

\begin{remark}[choice for $f_\mathcal{A}$]
Whilst the choices for $\Omega^{\mathrm{WP},\mathcal{A}}_{\Sigma}$ and $\mu_{\mathcal{A}}$ are natural, one might have naively chosen $f_\mathcal{A}$ to instead map via
\[
X\mapsto (\ell(\sigma_{j_1}),\ldots,\ell(\sigma_{j_k}),\ell(\beta_1),\ldots,\ell(\beta_m)).
\]
We first note that these two choices are more-or-less equivalent, as the induced volume forms agree up to multiplication by $\coth(c_{j_1}/2)\cdots\coth(c_{j_k}/2)$, which is a constant on $\mathcal{T}_\Sigma(\vec{b};\vec{c})$. One technical reason for using the present normalization for $f_{\mathcal{A}}$ in defining the Weil--Petersson form is that $\Omega_\Sigma^{\mathrm{WP}}(\vec{b};\vec{c})$ varies continuously as $c_j\to0$. In terms of conceptual geometric justification, \cref{lem:shear-width-quadrilateral} tells us that we are simply looking at level sets of relevant shearing parameters between pairs of adjacent arches in $X$ rather than their corresponding strap lengths. This choice for $f_{\mathcal{A}}$ also produces a very mildly cleaner statement when stating topologically recursive structure of the Weil--Petersson volume form $\Omega^{\mathrm{WP}}_{\Sigma}(\vec{b};\vec{c})$ in \cref{prop:top-rec-vol-form-orthogeod} and  \cref{prop:vol-form-teich-fixedhol-orth-coord}. 
\end{remark}

\begin{remark}
Both $\mathcal{T}_\Sigma^{\mathcal{A}}$ and $\RR^{|\mathcal{A}|}  \times \RR^m_{\geqslant 0}$ are manifolds with corners, and to apply \cref{subsec:submersions} to a fixed $\mathcal{T}_\Sigma(\vec{b};\vec{c})$ we restrict to the relevant (manifold) strata of $\mathcal{T}_\Sigma^{\mathcal{A}}$ and $\RR^{|\mathcal{A}|}  \times \RR^m_{\geqslant 0}$, where the stratification is according to $\vec{b} \in \RR^m_{\geqslant 0}$.
\end{remark}

The Teichm\"uller space $\mathcal{T}_\Sigma(\vec{b};\vec{c}\;|\vec{d})$ consisting of marked surfaces in $\mathcal{T}_\Sigma(\vec{b};\vec{c})$ with fixed neck lengths can be naturally assigned a volume by dividing by $\mathrm{d}d_j$:

\begin{definition}[WP volume form for $\mathcal{T}_\Sigma(\vec{b};\vec{c}\;|\vec{d})$]
\label{eq:vol-form-teich-fixed-hol-neck}
We define the Weil--Petersson volume form $\Omega^{\mathrm{WP}}_{\Sigma}(\vec{b};\vec{c}\;|\vec{d})$ on $\mathcal{T}_\Sigma(\vec{b};\vec{c}\;|\vec{d})$ as
\begin{align*}
\Omega^{\mathrm{WP}}_{\Sigma}(\vec{b};\vec{c}\;|\vec{d}) := \Omega^{\mathrm{WP}}_{\Sigma}(\vec{b};\vec{c})/\mu_l, 
\end{align*}
where 
\begin{itemize}
    \item $f:\mathcal{T}_\Sigma(\vec{b};\vec{c})\to\mathbb{R}^l$, $X\mapsto(\ell(\delta_1),\ldots,\ell(\delta_l))$ is the requisite submersion,
    \item $\omega = \Omega^{\mathrm{WP}}_{\Sigma}(\vec{b};\vec{c})$ is the Weil--Petersson volume form on $\mathcal{T}_\Sigma(\vec{b};\vec{c})$,
    \item and $\mu_l$ is the Euclidean volume element on $\RR^{l}$. 
\end{itemize}
\end{definition}

\subsection{Recursion formulae for Weil--Petersson volume forms}
A key point in Mirzakhani's integration scheme for moduli spaces of hyperbolic surfaces is that one should be able to reduce calculations on moduli spaces of surfaces of higher complexity to calculations on (covers of) moduli spaces of constituent subsurfaces. For Mirzakhani's work, this is complicated by the action of the mapping class group being non-trivial on (all) subsurface decompositions (and in particular, those which cut out a pair of pants), and she uses the McShane identity \cite{mcshane_allcusps,mirz_simp} to deal with this. For the present paper, we consider only surface decompositions on which the mapping class group acts trivially, and this produces significantly simpler expressions. In a planned sequel work, we will give novel McShane-type identities for crowned hyperbolic surfaces (i.e.: not just those in \cite{huang_thesis}, which were used in \cite{goncharov-sun}) and obtain other recursion formulae. In any case, we use the following notation when describing surface decompositions:

\begin{notation}[surface decomposition]
Consider a disconnected (topological) surface $\Sigma=\Sigma_1 \sqcup \cdots\sqcup\Sigma_\kappa$ obtained from cutting along essential simple closed curves and essential arcs, where each connected component $\Sigma_1, \cdots,\Sigma_\kappa$ is a (connected) finite type surfaces admitting hyperbolic metrics (see \cref{notn:sigmasurface}). We write
\begin{itemize}
    \item $\mathcal{T}_\Sigma$, 
    $\mathcal{T}_\Sigma(\vec{b};\vec{c})$, and $\mathcal{T}_\Sigma(\vec{b};\vec{c}\;|\vec{d})$ to respectively mean $\mathcal{T}_{\Sigma_1}\times\cdots\times\mathcal{T}_{\Sigma_\kappa}$, $\mathcal{T}_{\Sigma_1} (\vec{b}_1;\vec{c}_1)\times\cdots\times\mathcal{T}_{\Sigma_\kappa}(\vec{b}_\kappa;\vec{c}_\kappa)$, and $\mathcal{T}_{\Sigma_1}(\vec{b}_1;\vec{c}_1|\vec{d}_1)\times\cdots\times\mathcal{T}_{\Sigma_\kappa}
    (\vec{b}_\kappa;\vec{c}_\kappa|\vec{d}_\kappa)$, where the length vectors $\vec{b}_1,\ldots,\vec{b}_\kappa$, $\vec{c}_1,\ldots,\vec{c}_\kappa$, and $\vec{d}_1,\ldots,\vec{d}_\kappa$ are selectively ordered and labeled in accordance to the puncture ordering on $\Sigma_1 \sqcup \cdots\sqcup\Sigma_\kappa$;
    \item $\Omega^{\mathrm{WP}}_{\Sigma}
    (\vec{b};\vec{c})$ and $\Omega^{\mathrm{WP}}_{\Sigma}
    (\vec{b};\vec{c}\;|\vec{d})$ to respectively mean $\Omega^{\mathrm{WP}}_{\Sigma_1}(\vec{b};\vec{c}_1)\wedge\cdots\wedge\Omega^{\mathrm{WP}}_{\Sigma_\kappa}(\vec{b}_\kappa;\vec{c}_\kappa)$, and $\Omega^{\mathrm{WP}}_{\Sigma_1}
    (\vec{b};\vec{c}_1|\vec{d}_1)
    \wedge\cdots\wedge
    \Omega^{\mathrm{WP}}_{\Sigma_\kappa}(\vec{b}_\kappa;\vec{c}_\kappa|\vec{d}_\kappa)$.
\end{itemize}
We may omit specifying notation such as $\vec{b}$, $\vec{c}$, $\vec{d}$ and/or $\Sigma$ if the context is clear.
\end{notation}

We now examine what happens to the Weil--Petersson volume form when one cuts an essential simple closed curve, and then for cutting an essential arc. First, a little notation:

\begin{notation}[puncture labeling after cuts]
\label{not:punc-label-cuts}
Let $\gamma$ be an essential simple closed curve on $\Sigma$. Then $\Sigma\setminus\gamma$ has $2$ extra interior punctures compared to $\Sigma$. We label the extra punctures $p_{m+1},p_{m+2}$, and leave the existing labels unchanged. Similarly, if $\eta$ is an essential orthogeodesic on $X$, the underlying topological surface $\Sigma\setminus\eta$ has $2$ extra boundary punctures compared to $\Sigma$. We label the extra punctures $q_{n+1},q_{n+2}$, and leave the existing labels unchanged. 
\end{notation}

\subsubsection{Cutting along essential simple closed curves}
For Teichm\"uller spaces of surfaces with fixed boundary lengths (with \emph{no} tines or flippers), the Fenchel--Nielsen coordinate based expression for the Weil--Petersson symplectic form tells us that for any essential simple closed curve $\gamma$,
\[
\Omega^\mathrm{WP}_\Sigma(\vec{b})
=
\Omega
\wedge
\mathrm{d}\ell_\gamma\wedge\mathrm{d}\tau_\gamma,
\]
where $\Omega$ is a differential whose restriction to
\[
\{X\in\calT_\Sigma(\vec{b})\mid \ell_\gamma(X)=x\}
\cong
\calT_{\Sigma\setminus\gamma}(\vec{b},x,x)
\]
is equal to $\Omega^\mathrm{WP}_{\Sigma\setminus\gamma}(\vec{b},x,x)$. For crowned surfaces, one naturally expects an analogous statement. Chekhov implicitly asserts this in the paragraph (prior to and) justifying \cite[Lemma~2.6]{chekhov2024} in \S~2.3; and for \emph{decorated} crowned surfaces, Goncharov--Sun have a related claim in the form of \cite[Proposition~2.24]{goncharov-sun} with a conjectural constant balancing the equality, which is a key ingredient in establishing the neck recursion in \cite[Theorem~1.7]{goncharov-sun}. To establish the neck recursion for flippered surfaces, we explicitly convert between shearing and Fenchel--Nielsen parameters in a subsurface containing some chosen simple closed curve $\gamma$, and show:

\begin{theorem}
\label{prop:WP-vol-form-scc}
Let $\gamma$ be an arbitrary essential (non-peripheral) simple closed curve on $\Sigma$, then
\begin{align}
\label{eq:vol-form-neck-curve}
\Omega^{\mathrm{WP}}_{\Sigma}(\vec{b};\vec{c})
= 
\Omega\wedge \mathrm{d}\ell_\gamma\wedge \mathrm{d}\tau_\gamma,
\end{align}
where $\Omega$ is a differential whose restriction to 
\[
\{X\in\calT_\Sigma(\vec{b};\vec{c})\mid \ell_\gamma(X)=x,\ \tau_\gamma(X)=y\}
\cong
\calT_{\Sigma\setminus\gamma}(\vec{b},x,x;\vec{c})
\]
is equal to $\Omega^\mathrm{WP}_{\Sigma\setminus\gamma}(\vec{b},x,x;\vec{c})$. 

Further, if $\gamma = \nu_i$ is a neck curve, then
\begin{align}
\label{eq:vol-form-neck-curve-fix-neck}
\Omega^{\mathrm{WP}}_{\Sigma}(\vec{b};\vec{c}\;|\vec{d})
&= 
\Omega'
\wedge \mathrm{d}\tau_\gamma,
\end{align}
where $\Omega'$ is a differential whose restriction to 
\[
\{X\in\calT_\Sigma(\vec{b};\vec{c}\;|\vec{d})\mid \ell_{\nu_i}(X)=d_i,\ \tau_{\nu_i}(X)=y\}
\cong
\calT_{\Sigma\setminus\gamma}(\vec{b},d_i;\vec{c}\;|\vec{d})
\]
is equal to $\Omega^{\mathrm{WP}}_{\Sigma\setminus\gamma}(\vec{b},d_i;\vec{c}\;|\vec{d})$. 

If $\gamma$ is not a neck curve, then
\begin{align}
\Omega^{\mathrm{WP}}_{\Sigma}(\vec{b};\vec{c}\;|\vec{d}) 
= \Omega'' \wedge \mathrm{d}\ell_\gamma\wedge \mathrm{d}\tau_\gamma,
\end{align}
where $\Omega''$ is a differential whose restriction to 
\[
\{X\in\calT_\Sigma(\vec{b};\vec{c}\;|\vec{d})\mid \ell_\gamma(X)=x,\ \tau_\gamma(X)=y\}
\cong
\calT_{\Sigma\setminus\gamma}(\vec{b},x,x;\vec{c}\;|\vec{d})
\]
is equal to $\Omega^{\mathrm{WP}}_{\Sigma\setminus\gamma}(\vec{b},x,x;\vec{c}\;|\vec{d})$. 
\end{theorem}

\begin{proof}
Let $\mathcal{A}$ denote the set of indices for which $c_i$ is strictly positive. The Weil--Petersson form $\Omega^{\mathrm{WP}}_\Sigma(\vec{b};\vec{c})$ on $\mathcal{T}_\Sigma(\vec{b},\vec{c})$ is given by
\[
\mathrm{Clip}^*
(\Omega^{\mathrm{WP},\varnothing}_{\dot{\Sigma}(\mathcal{A})})
/\omega,
\]
where
\[
\omega:=\mathrm{d}b_1\wedge\ldots\wedge\mathrm{d}b_m\wedge\mathrm{d}\log(\sinh^2(c_1/2))\wedge\cdots\wedge\mathrm{d}\log(\sinh^2(c_l/2)).
\]
As a small abuse of notation, let $\gamma$ on $\Sigma$ denote the corresponding curve on $\dot{\Sigma}(\mathcal{A})$ (\cref{not:doubling}). By \cref{app:conversion}), we have
\[
\Omega^{\mathrm{WP},\varnothing}_{\dot{\Sigma}(\mathcal{A})}
=
\Omega\wedge\mathrm{d}\ell_\gamma\wedge\mathrm{d}\tau_\gamma,
\]
where $\Omega$ restricts to $\Omega^{\mathrm{WP}}_{\dot{\Sigma}(\mathcal{A})\setminus\gamma}/(\mathrm{d}\ell_{\gamma'}\wedge\mathrm{d}\ell_{\gamma''})$ on 
\[
\{X\in\calT_{\dot{\Sigma}(\mathcal{A})\setminus\gamma}\mid \ell_{\gamma'}(X)=\ell_{\gamma''}(X)=x\}.
\]
The result then follows because the clipping map preserves the length and twist of $\gamma$, and 
\[
\Omega^{\mathrm{WP}}_{\dot{\Sigma}(\mathcal{A})}(\vec{b},x,x;\vec{c})
=
\Omega^{\mathrm{WP}}_{\dot{\Sigma}(\mathcal{A})\setminus\gamma}/(\mathrm{d}\ell_{\gamma'}\wedge\mathrm{d}\ell_{\gamma''}\wedge\omega).
\]

To obtain the desired formula for when we fix neck lengths to being length $\vec{d}$, we first consider the case when $\gamma=\nu_i$ and divide by $\mathrm{d}\ell_\gamma=\mathrm{d}\ell_{\nu_i}$, use the fact that $\mathrm{d}\ell_\gamma\wedge\mathrm{d}\tau_\gamma
/\mathrm{d}\ell_\gamma=\mathrm{d}\tau_\gamma$ and replace $x$ with $d_i$. The statement when $\gamma$ is not a neck curve follows immediately from \cref{eq:vol-form-neck-curve}.
\end{proof}

\subsubsection{Cutting along essential arcs}

\begin{theorem}[WP volume under Nielsen cuts]
\label{prop:top-rec-vol-form-orthogeod}
Let $\eta$ be an essential arc on $\Sigma$. Then, for $\vec{c}_\eta = \left(c_1,\dotsc, c_n,\ell_\eta,\ell_\eta\right)$ 
\begin{align}
\label{eq:nielsen-cut-vol-form}
\Omega^{\mathrm{WP}}_{\Sigma}(\vec{b};\vec{c}) 
&= 
(\|\eta)^{*}\left(\Omega^{\mathrm{WP}}_{\Sigma\setminus\eta}(\vec{b};\vec{c}_\eta) \right)\wedge \coth(\ell_\eta/2)\;\mathrm{d}\ell_\eta,\text{ on $\mathcal{T}_\Sigma(\vec{b};\vec{c})$, and}
\\
\Omega^{\mathrm{WP}}_{\Sigma}(\vec{b};\vec{c}\;|\vec{d}) 
&=
(\|\eta)^{*}\left(\Omega^{\mathrm{WP}}_{\Sigma\setminus\eta}(\vec{b};\vec{c}_\eta|\vec{d}) \right)\wedge \coth(\ell_\eta/2)\;\mathrm{d}\ell_\eta,\text{ on $\mathcal{T}_\Sigma(\vec{b};\vec{c}\;|\vec{d})$.}
\end{align}
\end{theorem}

\begin{proof}
We divide the proof into two cases, depending on whether the arc $\eta$ joins distinct arches or has the endpoints on the same arch.\medskip

\noindent\textsc{Distinct arch case.} Let $\alpha_i\neq\alpha_j$ be the arches of $\Sigma$ that $\eta$ meets at its endpoints. Since $\eta$ is embedded and $\alpha_i,\alpha_j$ are disjoint (since they are distinct by assumption), there is a unique (up to isotopy) embedded ideal quadrilateral $Q \subseteq \Sigma$ which contains $\alpha_i,\alpha_j$ and has $\eta$ as an essential arc.

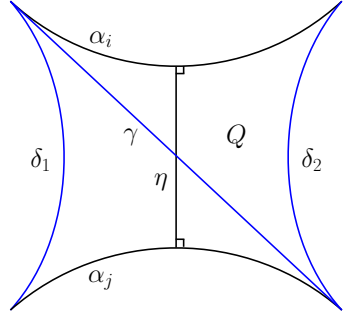
\begin{figure}[H]
    \centering
    \centering
\resizebox{1\textwidth}{!}{%
\begin{circuitikz}
\tikzstyle{every node}=[font=\Huge]

\draw[line width=3pt,short] (-11,26) .. controls (-4.75,20.25) and (4.75,20.25) .. (11,26);

\draw [blue][line width=3pt,short] (-11,26) .. controls (-6.25,20.5) and (-6.25,10.75) .. (-11,5.5);

\draw [blue][line width=3pt,short] (11,26) .. controls (6.25,20.5) and (6.25,10.75) .. (11,5.5);

\draw [line width=3pt] (-11,5.5) .. controls (-4.75,11) and (4.75,11) .. (11,5.5);

\draw [line width=3pt,short] (0,9.65) -- (0,21.65);

\draw [line width=3pt,short] (-3.03+3,21.5-0.3) -- (-2.45+3,21.5-0.3);

\draw [line width=3pt,short] (-2.5+3,21.5-0.3) -- (-2.5+3,22-0.3);

\draw [line width=3pt,short] (-3.03+3,21.5-0.3-11) -- (-2.45+3,21.5-0.3-11);

\draw [line width=3pt,short] (-2.5+3,21.5-0.3-11.5) -- (-2.5+3,22-0.3-11.5);

\node [font=\fontsize{50}{50}] at (-5,23.5) {$\alpha_i$};

\node [font=\fontsize{50}{50}] at (-5,7.5) {$\alpha_{j}$};

\node [font=\fontsize{50}{50}] at (-1,14) {$\eta$};

\node [font=\fontsize{50}{50}] at (-3,17) {$\gamma$};

\node [font=\fontsize{50}{50}] at (4,17) {$Q$};

\node [font=\fontsize{50}{50}] at (-9,15.5) {$\delta_1$};

\node [font=\fontsize{50}{50}] at (9,15.5) {$\delta_2$};

\draw[blue] [line width=3pt,short] (-11,26) -- (11,5.5);

\end{circuitikz}
}%
    \caption{Embedded ideal quadrilateral $Q$ in $\Sigma$.}
    \label{fig:vol-form-rec1}
\end{figure}

Choose $\mathcal{A}$ so that $\mathcal{T}_\Sigma(\vec{b};\vec{c}) \subset \mathcal{T}_\Sigma^\mathcal{A}$. Let $\gamma$ be one of the diagonals of $Q$ and let $\delta_1,\delta_2$ be the sides of $Q$ other than $\alpha_i,\alpha_j$, see \cref{fig:vol-form-rec1}. Complete $\{\gamma,\delta_1,\delta_2\}$ to an ideal triangulation of $\dot{\Sigma}(\mathcal{A})$, and let $\{s_1,\dotsc,s_k,s_{\delta_1},s_{\delta_2},s_\gamma\}$ be the associated shearing coordinates on $\mathcal{T}_{\dot{\Sigma}(\mathcal{A})}^{\varnothing}$. By \cref{eq:vol-form-teich-fixed-hol}, we have (up to sign):
\begin{equation}
\label{eq:vol-step1}
\begin{split}
\Omega^{\mathrm{WP}}_{\Sigma}(\vec{b};\vec{c}) 
&=
\pm\left(\left(\bigwedge_{i=1}^k \mathrm{d}s_i\right) \wedge \mathrm{d}s_{\delta_1} \wedge \mathrm{d}s_{\delta_2} \wedge \mathrm{d}s_{\gamma}\right)
\Big/\mu_{\mathcal{A}} 
\\&
= 
\pm\left[\left(\left(\bigwedge_{i=1}^k \mathrm{d}s_i\right) \wedge \mathrm{d}s_{\delta_1} \wedge \mathrm{d}s_{\delta_2} \right)
\Big/\mu_{\mathcal{A}} \right]\wedge \mathrm{d}s_{\gamma}.
\end{split}
\end{equation}
\cref{eq:vol-step1} holds because 
\[
f_{\mathcal{A}}^*\mu_\mathcal{A}\wedge \left[\left(\left(\bigwedge_{i=1}^k \mathrm{d}s_i\right) \wedge \mathrm{d}s_{\delta_1} \wedge \mathrm{d}s_{\delta_2} \right)
\Big/\mu_{\mathcal{A}} \right]\wedge \mathrm{d}s_{\gamma}
=
\left(\left(\bigwedge_{i=1}^k \mathrm{d}s_i\right) \wedge \mathrm{d}s_{\delta_1} \wedge \mathrm{d}s_{\delta_2} \wedge \mathrm{d}s_{\gamma}\right),
\]
but $\Omega^{\mathrm{WP}}_{\Sigma}(\vec{b};\vec{c})$ is the unique form on $\mathcal{T}_\Sigma(\vec{b};\vec{c})$ satisfying this property.

\begin{figure}[H]
    \centering
    \centering
\resizebox{1\textwidth}{!}{%
\begin{circuitikz}
\tikzstyle{every node}=[font=\Huge]

\draw[line width=3pt,short] (-11,26) .. controls (-4.75,20.25) and (4.75,20.25) .. (11,26);

\draw [blue][line width=3pt,short] (-11,26) .. controls (-6.25,20.5) and (-6.25,10.75) .. (-11,5.5);

\draw [line width=3pt,short] (11,26) .. controls (6.25,20.5) and (6.25,10.75) .. (11,5.5);

\draw [line width=3pt] (-11,5.5) .. controls (-4.75,11) and (4.75,11) .. (11,5.5);

\draw [line width=3pt,short] (0,9.65) -- (0,21.65);

\draw [line width=3pt,short] (-3.03+3,21.5-0.3) -- (-2.45+3,21.5-0.3);

\draw [line width=3pt,short] (-2.5+3,21.5-0.3) -- (-2.5+3,22-0.3);

\draw [line width=3pt,short] (-3.03+3,21.5-0.3-11) -- (-2.45+3,21.5-0.3-11);

\draw [line width=3pt,short] (-2.5+3,21.5-0.3-11.5) -- (-2.5+3,22-0.3-11.5);



\node [font=\fontsize{50}{50}] at (-2,14) {$\sigma_{n+1}$};

\node [font=\fontsize{50}{50}] at (-3,17) {$\gamma_1$};


\node [font=\fontsize{50}{50}] at (-9,15.5) {$\delta_1$};


\draw[blue] [line width=3pt,short] (-11,26) -- (11,5.5);


\draw[line width=3pt,short] (-11+25,26) .. controls (-4.75+25,20.25) and (4.75+25,20.25) .. (11+25,26);

\draw [line width=3pt,short] (-11+25,26) .. controls (-6.25+25,20.5) and (-6.25+25,10.75) .. (-11+25,5.5);

\draw [blue][line width=3pt,short] (11+25,26) .. controls (6.25+25,20.5) and (6.25+25,10.75) .. (11+25,5.5);

\draw [line width=3pt] (-11+25,5.5) .. controls (-4.75+25,11) and (4.75+25,11) .. (11+25,5.5);

\draw [line width=3pt,short] (0+25,9.65) -- (0+25,21.65);

\draw [line width=3pt,short] (-3.03+3+25,21.5-0.3) -- (-2.45+3+25,21.5-0.3);

\draw [line width=3pt,short] (-2.5+3+25,21.5-0.3) -- (-2.5+3+25,22-0.3);

\draw [line width=3pt,short] (-3.03+3+25,21.5-0.3-11) -- (-2.45+3+25,21.5-0.3-11);

\draw [line width=3pt,short] (-2.5+3+25,21.5-0.3-11.5) -- (-2.5+3+25,22-0.3-11.5);



\node [font=\fontsize{50}{50}] at (-1+25-1,14) {$\sigma_{n+2}$};

\node [font=\fontsize{50}{50}] at (-3+25,17) {$\gamma_2$};



\node [font=\fontsize{50}{50}] at (9+25,15.5) {$\delta_2$};

\draw[blue] [line width=3pt,short] (-11+25,26) -- (11+25,5.5);

\end{circuitikz}
}%
    \caption{Ideal quadrilateral $Q$ (as a part of $\dot{\Sigma\setminus\eta}$) after Nielsen cutting $X\|\eta$ and flipper clipping $\text{Clip}(X\|\eta)$.}
    \label{fig:vol-form-rec2}
\end{figure}
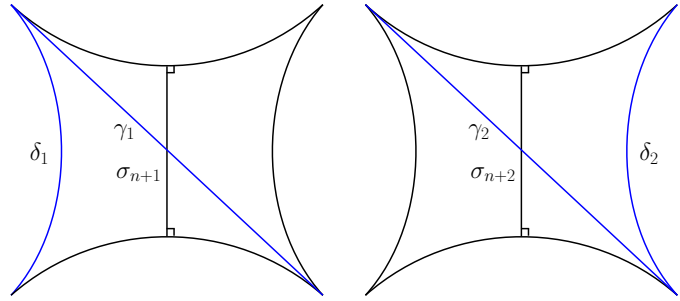

Next, let $\mathcal{A}_\eta = \mathcal{A}\sqcup\{n+1,n+2\}$ (see \cref{not:punc-label-cuts}). Then $\mathcal{T}_{\Sigma\setminus \eta}(\vec{b};\vec{c}_\eta)\subset \mathcal{T}_{\Sigma\setminus\eta}^{\mathcal{A}_\eta}$. Let $\gamma_1,\gamma_2$ be ideal arcs on $\dt{(\Sigma\setminus\eta)}(\mathcal{A_\eta})$ so that $\{s_1,\dotsc,s_k,s_{\delta_1},s_{\delta_2},s_{\gamma_1},s_{\gamma_2}\}$ are shearing coordinates on $\mathcal{T}_{\dt{(\Sigma\setminus\eta)}(\mathcal{A_\eta})}^{\varnothing}$, see \cref{fig:vol-form-rec2} (we slightly abuse the notation here by using the same notation for shears of the ``same'' ideal geodesics under $\|\eta$). By \cref{eq:vol-form-teich-fixed-hol},
\[
\Omega^{\mathrm{WP}}_{\Sigma\setminus\eta}(\vec{b};\vec{c}_\eta)
= \pm \left(\left(\bigwedge_{i=1}^k \mathrm{d}s_i\right) \wedge \mathrm{d}s_{\delta_1} \wedge \mathrm{d}s_{\delta_2} \wedge \mathrm{d}s_{\gamma_1}\wedge \mathrm{d}s_{\gamma_2}\right)\Big/\mu_{\mathcal{A_\eta}}.
\]
By \cref{lem:shear-width-quadrilateral},
\[
\mu_{\mathcal{A}_{\eta}} = \pm \mathrm{d}s_{\gamma_1}\wedge  \mathrm{d}s_{\gamma_2} \wedge  \mu_{\mathcal{A}},
\]
so dividing by $\mu_{\mathcal{A}}$ is tantamount to wedging by $\mathrm{d}s_{\gamma_1}\wedge \mathrm{d}s_{\gamma_2}$ and then dividing by $\mu_{\mathcal{A_\eta}}$, and hence
\[
\Omega^{\mathrm{WP}}_{\Sigma\setminus\eta}(\vec{b};\vec{c}_\eta)
= \pm \left(\left(\bigwedge_{i=1}^k \mathrm{d}s_i\right) \wedge \mathrm{d}s_{\delta_1} \wedge \mathrm{d}s_{\delta_2} \right)\Big/\mu_{\mathcal{A}}.
\]
Substituting the above identity into \cref{eq:vol-step1} and invoking \cref{lem:shear-width-quadrilateral} to obtain
\[
\mathrm{d}s_{\gamma} = \coth(\ell_\eta/2)\,\mathrm{d}\ell_\eta,
\]
we get $\Omega^{\mathrm{WP}}_{\Sigma}(\vec{b};\vec{c}) = \pm(\|\eta)^{*}(\Omega^{\mathrm{WP}}_{\Sigma\setminus\eta}(\vec{b};\vec{c}_\eta)) \wedge \coth(\ell_\eta/2)\;\mathrm{d}\ell_\eta$.
\vspace{1em}

\noindent\textsc{Same arch case.} Let $\alpha_i$ be the arch of $\Sigma$ that $\eta$ meets at its endpoints. Consider a small neighborhood $\mathcal{U}$ of $\eta\cup\alpha_i$, and let $\gamma, \delta \subset \mathcal{U} \subset \Sigma$ 
be ideal arcs as in \cref{fig:vol-form-rec3}. Choose $\mathcal{A}$ so that $\mathcal{T}_\Sigma(\vec{b};\vec{c}) \subset \mathcal{T}_\Sigma^\mathcal{A}$, and complete $\{\gamma
,\delta
\}$ to an ideal triangulation of $\dot{\Sigma}(\mathcal{A})$; let $\{s_1,\dotsc,s_k,s_{\delta},s_{\gamma}
\}$ be the associated shearing coordinates on  $\mathcal{T}_{\dot{\Sigma}(\mathcal{A})}^{\varnothing}$.

\begin{figure}[H]
    \centering
    \begin{tikzpicture}[scale=1.5]

\draw[rotate=90][thick] (0,0) arc (-90:0:1) ;  

\draw[rotate=90][thick] (1,1) .. controls (1,2) and (0,3) .. (1.5,3.5);

\draw[rotate=90][thick] (0,0) arc (-90:-180:1) ;  

\draw[rotate=90][thick] (-1,1) .. controls (-1,2) and (0,3) .. (-1.5,3.5);

\draw[rotate=90][thick,blue] (1.5,3.5) 
arc (-68:-112:4);

\draw[rotate=90][thick] (0.66,2.6) -- (-0.66,2.6);


\draw [rotate=90][thick,blue] 
(1.5,3.5) .. controls (1,3.33) .. (0,2.6);

\draw [rotate=90][thick,blue] 
(0,2.6) arc (180-55:180-0:2) coordinate (A);

\draw [rotate=90][thick,blue] 
(A) arc (180-0:180-(-50):0.8) coordinate (B);

\draw [rotate=90][thick,blue] 
(B) arc (180-(-50):5+180-(-175):0.9) coordinate (C);

\draw [rotate=90][thick,blue] 
(C) .. controls (0.9,1.5) and (0.62,2.3) .. (0.62,2.6);

\draw [rotate=90][thick,blue] 
(0.62,2.6) .. controls (0.66,3) and (1,3.4) .. (1.5,3.5);

\draw[rotate=90][thick, densely dotted] (0,1.2) circle (0.3);

\node [font=\fontsize{12}{12}] at (-1,1.2) {$\alpha_i$};

\node [font=\fontsize{12}{12}] at (-2.4,0.3) {$\eta$};

\node [font=\fontsize{12}{12}] at (-1,0.6) {$\mathcal{U}$};

\node [font=\fontsize{12}{12}] at (-3.4,0.05) {$\delta$};

\node [font=\fontsize{12}{12}] at (-0.3,0) {$\gamma$};

\draw[rotate=90](0.66,2.6) rectangle (0.66-0.05,2.6-0.05);

\draw[rotate=90] (-0.66,2.6) rectangle (-0.66+0.05,2.6-0.05);

\end{tikzpicture}
    \caption{A neighborhood $\mathcal{U}$ of $\alpha_i$ and $\eta$.}
    \label{fig:vol-form-rec3}
\end{figure}
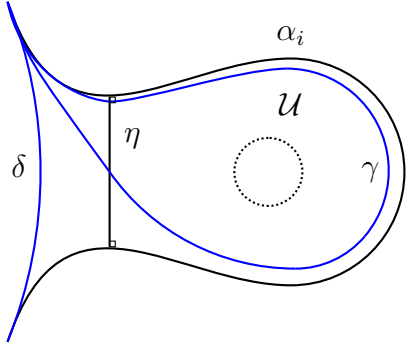
Then by \cref{eq:vol-form-teich-fixed-hol},
\begin{align}
\Omega^{\mathrm{WP}}_{\Sigma}(\vec{b};\vec{c}) 
= \pm 
\left(\left(\bigwedge_{i=1}^k \mathrm{d}s_i\right) \wedge \mathrm{d}s_{\delta}\wedge \mathrm{d}s_{\gamma}\right)
\Big/\mu_{\mathcal{A}}.\label{eq:volstep2}
\end{align}
Suppose that $f_\eta = g(s_1,\dotsc,s_k,s_{\gamma})$ for some $g\in C^\infty\left(\mathcal{T}_{\dot{\Sigma}(\mathcal{A})}^{\varnothing}\right)$.

Let $\mathcal{A}_\eta = \mathcal{A}\sqcup\{n+1,n+2\}$, then $\mathcal{T}_{\Sigma\setminus \eta}(\vec{b};\vec{c}_\eta)\subset \mathcal{T}_{\Sigma\setminus\eta}^{\mathcal{A}_\eta}$. Let $\gamma_1,\gamma_2$ be the ideal arcs on $\dt{(\Sigma\setminus\eta)}(\mathcal{A_\eta})$ as in \cref{fig:vol-form-rec4}, then $\{s_1,\dotsc,s_k,s_{\delta},s_{\gamma_1},s_{\gamma_2}\}$ are shearing coordinates on $\mathcal{T}_{\dt{(\Sigma\setminus\eta)}(\mathcal{A_\eta})}^{\varnothing}$. 

\begin{figure}[H]
    \centering
    \begin{tikzpicture}[scale=1.5]

\draw[rotate=90][thick] (0,0-3) arc (-90:0:1) ;  

\draw[rotate=90][thick] (1,1-3) .. controls (1,2-3) and (0,3-3) .. (1.5,3.5-3);

\draw[rotate=90][thick] (0,0-3) arc (-90:-180:1) ;  

\draw[rotate=90][thick] (-1,1-3) .. controls (-1,2-3) and (0,3-3) .. (-1.5,3.5-3);

\draw[rotate=90][thick] (1.5,3.5-3) 
arc (-68:-112:4);

\draw[rotate=90][thick] (0.66,2.6-3) -- (-0.66,2.6-3);


\draw [rotate=90][thick,blue] 
(1.5,3.5-3) .. controls (1,3.33-3) .. (0,2.6-3);

\draw [rotate=90][thick,blue] 
(0,2.6-3) arc (180-55:180-0:2) coordinate (A);

\draw [rotate=90][thick,blue] 
(A) arc (180-0:180-(-30):0.8) coordinate (B);

\draw [rotate=90][thick,blue] 
(B) arc (180-(-30):5+180-(-175):0.9) coordinate (C);

\draw [rotate=90][thick,blue] 
(C) .. controls (0.9,1.5-3) and (0.62,2.3-3) .. (0.62,2.6-3);

\draw [rotate=90][thick,blue] 
(0.62,2.6-3) .. controls (0.66,3-3) and (1,3.4-3) .. (1.5,3.5-3);

\draw[rotate=90][thick, densely dotted] (0,1.2-3) circle (0.3);


\node [font=\fontsize{7}{7}] at (-2.85+4+0.5-1,0.35) {$\sigma_{n+2}$};


\node [font=\fontsize{12}{12}] at (-0.3+4-1,0) {$\gamma_2$};

\draw[rotate=90](0.66,2.6-3) rectangle (0.66-0.05,2.6-0.05-3);

\draw[rotate=90] (-0.66,2.6-3) rectangle (-0.66+0.05,2.6-0.05-3);


\draw[rotate=90][thick] (-1.5,1.7) .. controls (-0.4,2.1) and (-0.4,3.1)   .. (-1.5,3.5);

\draw[rotate=90][thick,blue] (1.5,3.5) arc (-68:-112:4);

\draw[rotate=90][thick] (1.5,1.7) arc (68:112:4);

\draw[rotate=90][thick] (0.66,2.6) -- (-0.66,2.6);

\draw[rotate=90][thick,blue] (1.5,3.5) .. controls (1,3.3) .. (0,2.6);

\draw[rotate=90][thick,blue] (-1.5,1.7) .. controls (-1,1.9) .. (0,2.6);

\draw[rotate=90][thick] (1.5,1.7) .. controls (0.4,2.1) and (0.4,3.1)  .. (1.5,3.5);

\node [font=\fontsize{12}{12}] at (-3.3,0.05) {$\delta$};

\node [font=\fontsize{7}{7}] at (-2.85+0.5,0.35) {$\sigma_{n+1}$};

\node [font=\fontsize{10}{10}] at (-2.3,-0.1) {$\gamma_1$};

\draw[rotate=90](0.66,2.6) rectangle (0.66-0.05,2.6-0.05);

\draw[rotate=90] (-0.66,2.6) rectangle (-0.66+0.05,2.6-0.05);

\end{tikzpicture}
    \caption{What happens to $\mathcal{U}$ after applying Nielsen cutting and flipper clipping to get $\text{Clip}(X\|\eta)$.}
    \label{fig:vol-form-rec4}
\end{figure}
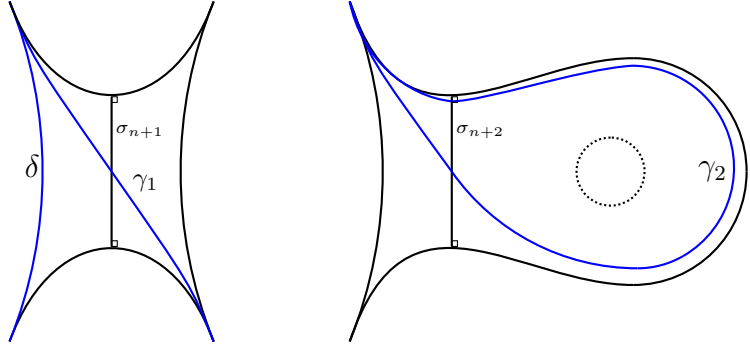

By \cref{eq:vol-form-teich-fixed-hol},
\[
\Omega^{\mathrm{WP}}_{\Sigma\setminus\eta}(\vec{b};\vec{c}_\eta)
= \pm \left(\left(\bigwedge_{i=1}^k \mathrm{d}s_i\right) \wedge \mathrm{d}s_{\delta}  \wedge \mathrm{d}s_{\gamma_1}\wedge \mathrm{d}s_{\gamma_2}\right)\Big/\mu_{\mathcal{A_\eta}}.
\]
Note that by \cref{lem:shear-width-quadrilateral}, $\mu_{\mathcal{A_\eta}}=\pm\mathrm{d}s_{\gamma_1}\wedge \mathrm{d} \log\sinh^2(\ell_{\sigma_{n+2}})\wedge \mu_{\mathcal{A}}$, and observe that $\log\sinh^2(\ell_{\sigma_{n+2}}) = g(s_1,\dotsc,s_k,s_{\gamma_2})$. Therefore  
\[
\Omega^{\mathrm{WP}}_{\Sigma\setminus\eta}(\vec{b};\vec{c}_\eta)
= \pm \left(\left(\bigwedge_{i=1}^k \mathrm{d}s_i\right) \wedge \mathrm{d}s_{\delta} \wedge \mathrm{d}s_{\gamma_2}\right)\Big/\left(\mathrm{d}g(s_1,\dotsc,s_k,s_{\gamma_2})\wedge \mu_{\mathcal{A}}\right).
\]
Since the pullback of $s_{\gamma_2}$ along $\|\eta$ is $s_{\gamma}$ and $\mathrm{d}g(s_1,\dotsc,s_k,s_\gamma)=\mathrm{d}f_\eta=\coth(\ell_\eta/2)\,\mathrm{d}\ell_\eta$, wedging by $\coth(\ell_\eta/2)\,\mathrm{d}\ell_\eta$ and comparing the result with \cref{eq:volstep2}, we conclude the proof of \cref{eq:nielsen-cut-vol-form}. The analogous result when neck lengths are fixed follows as a simple consequence.
\end{proof}

\begin{remark}
    The introduction of the function $g$ is simply to help us switch between expressions expressed in terms of (the conflated) shearing coordinates on $\mathcal{T}^\varnothing_{\dot{\Sigma}(\mathcal{A})}$ and $\mathcal{T}^\varnothing_{\dot{(\Sigma\setminus\eta)}(\mathcal{A})}$.
\end{remark}

\subsection{WP volume form in coordinates}

The recursion formulae \cref{prop:WP-vol-form-scc} and \cref{prop:top-rec-vol-form-orthogeod} allow us to express Weil--Petersson volume forms $\Omega^{\mathrm{WP}}_\Sigma(\vec{b};\vec{c})$ in terms of Weil--Petersson forms of surfaces constituting $\Sigma$. When carried through to the extreme, this results in expressions for the Weil--Petersson form in terms of $\Omega^{\mathrm{WP}}_{\mathbb{D}_3}$, $\Omega^{\mathrm{WP}}_{\mathbb{A}_1}$ as well as differentials of lengths of arcs, lengths of simple closed curves and twists of simple closed curves. As an example, iterated applications of \cref{prop:top-rec-vol-form-orthogeod} yields the following:

\begin{proposition}
\label{prop:vol-form-teich-fixedhol-orth-coord}
Let $\square$ be a maximal arc collection on $\Sigma$. By maximality, $\square$ decomposes $\Sigma$ into 
\begin{itemize}
    \item $M$ number of flippered triangles $\mathbb{D}_3$ with respective strap lengths (which are orthogeodesic coordinates) $\vec{c}_{(i)}:=(c_{(i),1},c_{(i),2},c_{(i),3})$, $i=1,\ldots,M$; 
    \item and $m$ number of funneled half-pants $\mathbb{A}_1$ respectively containing the $j$-th funnel and a strap of length $c'_{(j)}$ for $j=1,\ldots,m$.
\end{itemize}
Then, the Weil--Petersson volume form $\Omega^{\mathrm{WP}}_{\Sigma}(\vec{b};\vec{c})
$ on $\mathcal{T}_\Sigma(\vec{b};\vec{c})$ satisfies 
\begin{align}
\Omega^{\mathrm{WP}}_{\Sigma}(\vec{b};\vec{c}) 
= \pm\Biggl(
\prod_{i=1}^M
\Omega^{\mathrm{WP}}_{\mathbb{D}_3}(\vec{c}_{(i)}) \cdot
\prod_{j=1}^{m}
\Omega^{\mathrm{WP}}_{\mathbb{A}_1}(b_j;c'_{(j)})
\Biggr)
\bigwedge_{\eta\in\square}\coth(\ell_\eta/2)\mathrm{d}\ell_\eta.
\label{eqn:volformbigproduct}
\end{align}
\end{proposition}

\begin{remark}
The product notation used in \cref{eqn:volformbigproduct} implicitly assert that $\Omega^{\mathrm{WP}}_{\mathbb{D}_3}$ and $\Omega^{\mathrm{WP}}_{\mathbb{A}_1}$ are volume $0$-forms. These functions are explicitly determined in \cref{subsec:vol-form-flip-triang} and \cref{subsec:1-crown-vol-form}.    
\end{remark}

\begin{remark}
\cref{prop:vol-form-teich-fixedhol-orth-coord} and the smoothness of $\Omega^{\mathrm{WP}}_{\mathbb{D}_3}$ and $\Omega^{\mathrm{WP}}_{\mathbb{A}_1}$ hints that there is may be a precise sense in which these forms on different dimensional strata might ``glue'' to a global volume form (for a manifold with corners) on $\mathcal{T}_\Sigma$. For our purposes, however, we only work with submanifolds of $\mathcal{T}_\Sigma$ fixed $\vec{c}$ values and hence only ever work within a single stratum at any given moment.
\end{remark}

One can give a more general statement by allowing for ``mixed coordinates'' comprising both arc lengths and Fenchel--Nielsen coordinates of simple closed curves. 

\begin{theorem}
Consider a collection of pairwise disjoint arcs $\square=\{\eta_i\}$ and simple closed curves $\Gamma=\{\gamma_j\}$ which decompose $\Sigma$ into
\begin{itemize}
    \item $M$ flippered triangles $\mathbb{D}_3$ with straps of length $\vec{c}_{(i)}:=(c_{(i),1},c_{(i),2},c_{(i),3})$, $i=1,\ldots,N$, whose lengths are all orthogeodesic coordinates for arcs in $\square$;
    \item $N$ $\mathbb{A}_1$ with parameters $\ell_{(j)}$ and $c'_{(j)}$, for $j=1,\ldots,N$ (which come from lengths of corresponding arcs in $\square$ and curves in $\Gamma\cup\{\beta_1,\ldots,\beta_m\}$).
\end{itemize}
Then, the Weil--Petersson volume form $\Omega^{\mathrm{WP}}_{\Sigma}(\vec{b};\vec{c})
$ on $\mathcal{T}_\Sigma(\vec{b};\vec{c})$ is equal to
\begin{align*}
\pm
\Biggl(
\prod_{i=1}^M
\Omega^{\mathrm{WP}}_{\mathbb{D}_3}(\vec{c}_{(i)}) 
\cdot
\prod_{j=1}^{N}
\Omega^{\mathrm{WP}}_{\mathbb{A}_1}(\ell_{(j)};c'_{(j)})
\Biggr)
\Biggl(
\bigwedge_{\eta\in\square}
\coth(\ell_\eta/2)\mathrm{d}\ell_\eta
\Biggr)
\wedge
\Biggl(
\bigwedge_{\gamma\in\Gamma}
\mathrm{d}\ell_\gamma\wedge\mathrm{d}\tau_\gamma
\Biggr).
\end{align*}
\end{theorem}

\begin{remark}
    If $\Gamma=\Gamma'\sqcup\{\nu_1,\ldots,\nu_l\}$ contains \emph{all} of the necks of $\Sigma$, then we can also give a comparatively clean expression for $\Omega^{\mathrm{WP}}_\Sigma(\vec{b};\vec{c}\;|\vec{d})$ by replacing the $\bigwedge_{\gamma\in\Gamma}
\mathrm{d}\ell_\gamma\wedge\mathrm{d}\tau_\gamma$ term in the above result by
\[
\Biggl(\bigwedge_{\gamma\in\Gamma'}
\mathrm{d}\ell_\gamma\wedge\mathrm{d}\tau_\gamma
\Biggr)
\wedge
\Biggl(\bigwedge_{k=1}^{l}
\mathrm{d}\tau_{\nu_k}
\Biggr).
\]

\end{remark}

\subsubsection{Weil-Petersson volume forms on spaces of flippered triangles}
\label{subsec:vol-form-flip-triang}

\begin{proposition}
\label{prop:vol-form-triang}
The volume form of the Teichm\"uller space $\mathcal{T}_{\DD_3}(\vec{c})$ of the flippered triangle with strap lengths $\vec{c}=(c_1, c_2, c_3)$ is given by the function
\begin{equation*}
\Omega^{\mathrm{WP}}_{\DD_3}(\vec{c})= \frac{\cosh c_1+\cosh c_2+\cosh c_3-1+\sqrt{D}}{2\sqrt{D}},
\end{equation*}
where $D = \cosh^2 c_1 + \cosh^2 c_2+\cosh^2 c_3 + 2\cosh c_1 \cosh c_2 \cosh c_3-1$. Moreover, $\Omega^{\mathrm{WP}}_{\DD_3}(\vec{c})\leq1$ and equals 1 precisely when at least one of $c_1,c_2,c_3$ is $0$. 
\end{proposition}

\begin{remark}
    The mapping class group action is trivial for $\mathbb{D}_3$, so this result also holds for the moduli space $\mathcal{M}_{\DD_3}(\vec{c})$.
\end{remark}

\begin{proof}

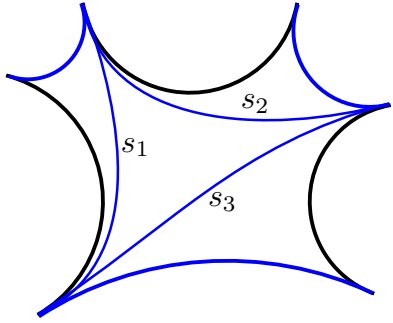
\begin{figure}[H]
    \centering
    \begin{tikzpicture}[scale=1.2]
\clip (-2.5,-1.5) rectangle (2.5,2.5);
\usetikzlibrary{decorations.pathreplacing}

  \draw[very thick]
    (0,2) ++(-10:1) arc[start angle=-10, end angle=-170, radius=1];

  \draw[very thick]
    (-2,0) ++(-60:1.2) arc[start angle=-60, end angle=75, radius=1.2];

  \draw[very thick]
    (2,0) ++(100:0.9) arc[start angle=100, end angle=250, radius=0.9];


\draw[very thick, blue]
    (-1.47,1.65) ++(20:0.5) arc[start angle=20, end angle=-110, radius=0.52];


\draw[very thick, blue]
    (1.65,1.57) ++(-75:0.70) arc[start angle=-75, end angle=-200, radius=0.70];


\draw[very thick, blue]
    (0.41,-3.56) ++(65:3) arc[start angle=65, end angle=123, radius=3.17];

\draw[thick, blue] (-1.4,-1.04) to [out=25,in=-70] (-0.98,1.8);

\draw[thick, blue] (-0.98,1.8) to [out=-80,in=-165] (1.83,0.9);

\draw[thick, blue] (1.83,0.9) to [out=-165,in=33] (-1.4,-1.04);

\draw (-0.5,0.5) node {$s_1$};

\draw (0.6,0.9) node {$s_2$};

\draw (0.3,0) node {$s_3$};



\end{tikzpicture}
    \caption{An ideal triangulation of $\dot{\DD}_3(1,2,3)$.}
    \label{fig:flip-triang-flip-clip}
\end{figure}

Consider the ideal triangulation of the puncture-double surface $\dot{\DD}_3(1,2,3)$ as in \cref{fig:flip-triang-flip-clip} (see \cref{not:doubling}), and let $\{s_1,s_2,s_3\}$ be the associated shears. Independently performing a flip of each of the arcs in the ideal triangulation and applying \cref{prop:flips} together with \cref{lem:shear-width-quadrilateral}, we obtain 
\begin{equation}
\label{eq:shears-flip-triang}
\begin{cases}
s_1 + \log(1+e^{s_2}) = \log\sinh^2(c_1/2), \\
s_2 + \log(1+e^{s_3}) = \log\sinh^2(c_2/2), \\
s_3 + \log(1+e^{s_1}) = \log\sinh^2(c_3/2).
\end{cases}
\end{equation}
Let $X_i=e^{s_i}$ for $1\leqslant i \leqslant 3$ be the associated exponentiated shears, and let $C_i=\cosh c_i$ for $1\leqslant i \leqslant 3$. Then upon exponentiation, \cref{eq:shears-flip-triang} can be written as 
\begin{equation}
\label{eq:exp-shears-flip-triang}
\begin{cases}
X_1(1+X_2) = \frac{C_1-1}{2}, \\
X_2(1+X_3) = \frac{C_2-1}{2}, \\
X_3(1+X_1) = \frac{C_3-1}{2}.
\end{cases}
\end{equation}
Thus by \cref{eq:vol-form-teich-fixed-hol}, the volume form on $\mathcal{M}_{\DD_3}(c_1,c_2,c_3)$ is equal to
\begin{equation}
\label{eq:vol-form-flip-triang}
\begin{split}
\Omega^{\mathrm{WP}}_{\DD_3}(\vec{c})&=\pm\bigwedge_{i=1}^3\frac{\mathrm{d}X_i}{X_i} \Bigg/ \bigwedge_{i=1}^3\frac{\mathrm{d}\bigl(X_i(1+X_{i+1})\bigr)}{X_i(1+X_{i+1})} \\&= \pm\frac{(1+X_1)(1+X_2)(1+X_3)}{(1+X_1)(1+X_2)(1+X_3)+X_1X_2X_3}.  
\end{split}
\end{equation}
We now express \cref{eq:vol-form-flip-triang} in $C_1,C_2,C_3$. From \cref{eq:exp-shears-flip-triang}, we compute
\[
D = C_1^2+C_2^2+C_3^2+2C_1C_2C_3-1 = 4\bigl((1+X_1)(1+X_2)(1+X_3)+X_1X_2X_3\bigr)^2.
\]
It is then straightforward to verify the equality
\begin{equation*}
\begin{split}
\frac{C_1+C_2+C_3-1+\sqrt{D}}{2\sqrt{D}} 
= \frac{(1+X_1)(1+X_2)(1+X_3)}{(1+X_1)(1+X_2)(1+X_3)+X_1X_2X_3} = \Omega^{\mathrm{WP}}_{\DD_3}(\vec{c}).
\end{split}
\end{equation*}
Now, the lower line in \cref{eq:vol-form-flip-triang} tells us that $\Omega^{\mathrm{WP}}_{\DD_3}(\vec{c})\leq1$. To get equality, at least one of the $X_i\to 0$, which is a degenerate situation where $s_i\to-\infty$ and the corresponding flipper collapses to a tine. We can see this algebraically via \cref{eq:exp-shears-flip-triang}, which asserts $C_i=1$ and hence $c_i=0$.
\end{proof}

\subsubsection{Weil-Petersson volume forms on moduli spaces of flippered 1-crowns}
\label{subsec:1-crown-vol-form}

\begin{proposition}
\label{prop:1-crown-vol-form}
The volume form on the Teichm\"uller space $\mathcal{T}_{\AA_1}(c|d)$ of the 1-crown with neck length $d$ and the flipper with strap length $c$ is given by the function
\begin{equation}
\Omega^{\mathrm{WP}}_{\AA_1}(c|d)=\frac{\sqrt{\frac{\cosh c+\cosh d}{2}}+\cosh(d/2)}{\sqrt{2\cosh c+2\cosh d}}
\end{equation}
Moreover, $\Omega^{\mathrm{WP}}_{\mathbb{A}_1}(c|d)\leq1$ and equals 1 precisely when $c=0$. 
\end{proposition}

\begin{remark}
    The mapping class group action is trivial for $\mathbb{A}_1$, so this result also holds for the moduli space $\mathcal{M}_{\mathbb{A}_1}(c|d)$.
\end{remark}

\begin{proof}
Consider the ideal triangulation of the puncture-double surface $\dot{\AA}_1(1)$ (\cref{not:doubling}) by two ideal arcs emanating from the interior puncture and let $\{s_1,s_2\}$ be the associated shearing coordinates, listed in the counterclockwise order (from the perspective of the interior puncture). Let $X_i=e^{s_i}$ for $i=1,2$ be the associated exponentiated shears. By \cref{eq:vol-form-teich-fixed-hol}, the volume form on $\mathcal{M}_{\AA_1}(c|d)$ is equal to
\begin{equation}
\Omega^{\mathrm{WP}}_{\AA_1}(c|d) = \pm \frac{\mathrm{d}X_1}{X_1}\wedge \frac{\mathrm{d}X_2}{X_2}\Big/\mathrm{d}\log \sinh^2(c/2)\wedge \mathrm{d}d.
\end{equation}
Combining \cref{ex:pentagon} with \cref{lem:shear-width-quadrilateral}, we have 
\begin{equation}
\sinh^2(c/2)=X_1+X_2^{-1}+X_1X_2^{-1},
\end{equation}
which is convenient (for future reference) to equivalently express as 
\begin{equation}
\label{eq:1-flip-crown-system}
\cosh^2(c/2) =  (1+X_1)(1+X_2^{-1}).
\end{equation}
We compute
\begin{equation}
\begin{split}
\mathrm{d}\log \sinh^2(c/2) = \frac{(X_2+1)\mathrm{d}X_1-(X_2^{-1}+X_1X_2^{-1})\mathrm{d}X_2}{X_1X_2+1+X_1}.
\end{split}
\end{equation}
Next, $d = s_1+s_2 = \log (X_1X_2).$ Then $\mathrm{d}d = \frac{\mathrm{d}X_1}{X_1}+\frac{\mathrm{d}X_2}{X_2}$, and we compute
\begin{equation}
\begin{split}
\mathrm{d}\log \sinh^2(c/2)\wedge \mathrm{d}d &= 
\frac{X_{1}^{-1}X_{2}^{-1}+2X_{2}^{-1}+1}{X_1X_2+X_1+1} \mathrm{d}X_1\wedge \mathrm{d}X_2
\end{split}
\end{equation}
Therefore, 
\begin{equation}
\label{eq:1-crown-vol-form-shears}
\begin{split}
\Omega^{\mathrm{WP}}_{\AA_1}(c|d) &= \pm \frac{\mathrm{d}X_1}{X_1}\wedge \frac{\mathrm{d}X_2}{X_2}\Big/\mathrm{d}\log \sinh^2(c/2)\wedge \mathrm{d}d
\\& = \frac{X_1X_2+X_1+1}{X_1X_2+2X_1+1}.
\end{split}
\end{equation}
We now express \cref{eq:1-crown-vol-form-shears} in $c,d$. Note that
\begin{equation}
\begin{split}
(X_1X_2-1)^2+4X_1X_2(1+X_1)(1+X_2^{-1})= (X_1X_2+2X_1+1)^2,
\end{split}
\end{equation}
hence $ (X_1X_2+2X_1+1)^2=(e^d-1)^2+4e^d\cosh^2(c/2).$
It is then straightforward to verify the equalities
\begin{equation}
\begin{split}
\Omega^{\mathrm{WP}}_{\AA_1}(c|d) &= \frac{X_1X_2+X_1+1}{X_1X_2+2X_1+1} 
\\& = \frac{\frac{1}{2}\left(\sqrt{(e^d-1)^2+4e^d\cosh^2(c/2)}+e^d+1\right)}{\sqrt{(e^d-1)^2+4e^d\cosh^2(c/2)}
}
\\& = \frac{\sqrt{(e^{d/2}-e^{-d/2})^2+4\cosh^2(c/2)}+e^{d/2}+e^{-d/2}}{2\sqrt{(e^{d/2}-e^{-d/2})^2+4\cosh^2(c/2)}}
\\& =  \frac{\sqrt{4\sinh^2(d/2)+4\cosh^2(c/2)}+2\cosh(d/2)}{2\sqrt{4\sinh^2(d/2)+4\cosh^2(c/2)}}
\\&
= \frac{\sqrt{\frac{\cosh c+\cosh d}{2}}+\cosh(d/2)}{\sqrt{2\cosh c+2\cosh d}}.
\end{split}
\end{equation}
From the first line in \cref{eq:1-crown-vol-form-shears}, we see that $\Omega^{\mathrm{WP}}_{\AA_1}(c|d)\leq1$. Equality to $1$ is attained precisely when $X_1\to0$ or if $X_2\to+\infty$. Indeed, since $X_1X_2=e^d$, these two conditions must both occur at the same time, and this is the singular situation when the flipper degenerates to a tine. Algebraically, we can see this from \cref{eq:1-flip-crown-system}, which asserts that $\cosh^2(c/2)\to1$, and hence $c\to0$.
\end{proof}

\section{Mirzakhani volumes of moduli spaces of flippered surfaces and their recursive structure}
\label{sec:volumes}

In this section, we derive volume recursion formulae for the Mirzakhani volume in the form of the neck, disk, and the crown recursions (respectively, \cref{prop:neck-recursion}, \cref{thm:disk-recursion} and \cref{thm:crown-rec}). The upshot is that one can reduce volume calculations for moduli spaces of more complicated surfaces to that of integrals involving volume functions of ``building blocks'' ---moduli spaces of topologically simpler surfaces. To facilitate this computational scheme, we
\begin{itemize}
    \item state the volumes of the moduli spaces of flippered triangles and of flippered $1$-crowns (\cref{thm:vol-fli-triang} and \cref{thm:vol-1-crown});
    \item determine single integral expressions for the volumes of flippered $n$-gons and of flippered $n$-crowns (\cref{thm:vol-n-gon-1-int}, \cref{thm:vol-n-crown-1-int});
    \item compute the volume of the moduli space of flippered $n$-crowns without neck constraints (\cref{cor:vol-n-crown-no-neck}).
\end{itemize}

These results implicitly assume the finiteness of the Mirzakhani volumes, which we prove in \cref{thm:finvolflipper}.

\subsection{Volumes of moduli spaces of flippered triangles and $1$-crowns}
\label{subsec:vol-flip-triang}

We combine \cref{subsec:action-triang} and \cref{subsec:vol-form-flip-triang} to express the volumes of the moduli spaces of flippered triangles.

\begin{theorem}
\label{thm:vol-fli-triang}
The volume of the moduli space $\mathcal{M}_{\DD_3}(c_1,c_2,c_3)$ of the flippered triangle with straps of length $c_1, c_2, c_3$ is equal to
\begin{equation}
V_{\DD_3}(c_1,c_2,c_3) = \frac{2
}{\sqrt{D}},   
\end{equation}
where $D = \cosh^2 c_1 + \cosh^2 c_2+\cosh^2 c_3 + 2\cosh c_1 \cosh c_2 \cosh c_3-1$.
\end{theorem}
\begin{proof}
Combine \cref{prop:action-triang} and \cref{prop:vol-form-triang}.
\end{proof}

\begin{theorem}
\label{thm:vol-1-crown}
The volume of the moduli space $\mathcal{M}_{\AA_1}(c|d)$ of the 1-crown with neck length $d$ and the flipper with strap length $c$ equals
\begin{equation*}
V_{\AA_1}(c|d)= \frac{1}{\sqrt{2\cosh c+2\cosh d}}.
\end{equation*}
\end{theorem}
\begin{proof}
Combine \cref{prop:1-crown-action} and \cref{prop:1-crown-vol-form}.
\end{proof}

\begin{remark}
\label{rem:vol-3-gon-1-crown}
From \cref{prop:vol-form-triang} it is immediate that $V_{\DD_3}(c_1,c_2,c_3)$ equals the generalized Chekhov action (given by \cref{cor:action-flip-triang}) precisely when $c_1c_2c_3=0$. Similarly, \cref{prop:1-crown-vol-form} tells us that $V_{\AA_1}(c|d)$ equals the generalized Chekhov action (given by \cref{cor:action-flip-triang}) precisely when $c=0$. Moreover, setting all strap lengths to $0$ agrees with volumes computed in \cite{HT25}.
\end{remark}

\subsection{Neck recursion} 

Cutting a flippered surface at any of its necks yields a hyperbolic surface with one fewer crown and one more funnel (after applying Nielsen extension), as well as a surface which is a crown. 

\begin{proposition}
\label{prop:neck-recursion}
Consider $\Sigma=\Sigma'\cup \mathbb{A}_m$ (see \cref{fig:neckrecursion}) with $\vec{c}:=(\vec{c}\,',c_1,\ldots,c_m)$ and $\vec{d}'=(\vec{d},\ell)$, then 
\begin{align}
V_\Sigma\bigl(\vec{b};\vec{c}\;|\vec{d}\bigr) 
&=
V_{\Sigma'}\bigl(\vec{b},\ell;\vec{c}\,'|\vec{d}\,'\bigr)
\cdot 
V_{\mathbb{A}_m}\bigl(c_1,\ldots,c_m|\ell\bigr) \cdot \ell 
\text{, and}\\
V_\Sigma\bigl(\vec{b};\vec{c}\bigr) 
&= 
\int_0^\infty 
V_{\Sigma'}\bigl(\vec{b},\ell;\vec{c}\bigr)
\cdot
V_{\mathbb{A}_m}\bigl(c_1,\dots,c_m|\ell\bigr) 
\,\cdot \ell\;\mathrm{d}\ell.
\end{align}
\end{proposition}

\begin{proof}
By combining the additivity property of the Chekhov action with respect to the crown components (see \cref{rmk:additive-action}) and the splitting property of the Weil--Petersson volume form along the neck curves (\cref{eq:vol-form-neck-curve-fix-neck} in \cref{prop:WP-vol-form-scc}), it follows that 
\begin{equation*}
\begin{split}
&V_\Sigma\bigl(\vec{b};\vec{c}\;|\vec{d}\bigr) 
\\\smash{\mathllap{=}}\,& \int_{\mathcal{M}_{\Sigma}(\vec{b};\vec{c}\;|\vec{d})}e^{-S(X)}\Omega^{\text{WP}}_{\Sigma}\bigl(\vec{b};\vec{c}\;|\vec{d}\bigr)
\\\smash{\mathllap{=}}\,& \int_{\mathcal{M}_{\Sigma\setminus\gamma}(\vec{b},\ell_\gamma;\vec{c}\;|\vec{d})\times (\RR/\ell_\gamma\ZZ)} e^{-S(X\setminus\gamma)}\Omega^{\mathrm{WP}}_{\Sigma\setminus\gamma}(\vec{b},\ell_\gamma;\vec{c}\;|\vec{d})\wedge \mathrm{d}\tau_\gamma 
\\\smash{\mathllap{=}}\,& \int_0^\ell \left(\int_{\mathcal{M}_{\Sigma'}\bigl(\vec{b},\ell;\vec{c}\,'|\vec{d}\,'\bigr)\times \mathcal{M}_{\AA_m}(c_1,\dotsc,c_m|\ell)} e^{-S(X')}\cdot e^{-S(Y)}
\right.\\&\quad\quad\quad\quad\quad\quad\quad\quad\quad\quad\quad\quad\quad\quad\quad\left.\cdot\Omega^{\text{WP}}_{\Sigma'}\bigl(\vec{b},\ell;\vec{c}\,'|\vec{d}\,'\bigr) \wedge \Omega^{\text{WP}}_{\AA_m}(c_1,\dotsc,c_m|\ell)\right) \mathrm{d}\tau
\\\smash{\mathllap{=}}\,&
\int_0^\ell \left(\int_{\mathcal{M}_{\Sigma'}\bigl(\vec{b},\ell;\vec{c}\,'|\vec{d}\,'\bigr)} e^{-S(X')}\Omega^{\text{WP}}_{\Sigma'}\bigl(\vec{b},\ell;\vec{c}\,'|\vec{d}\,'\bigr)\right.\\
&\quad\quad\quad\quad\quad\quad\quad\quad\quad\quad\quad\quad
\left.\times \int_{\mathcal{M}_{\AA_m}(c_1,\dotsc,c_m|\ell)} e^{-S(Y)}  \Omega^{\text{WP}}_{\AA_m}(c_1,\dotsc,c_m|\ell)    \right) \mathrm{d}\tau
\\\smash{\mathllap{=}}\,& 
V_{\Sigma'}\bigl(\vec{b},\ell;\vec{c}\,'|\vec{d}\,'\bigr)
\cdot 
V_{\mathbb{A}_m}\bigl(c_1,\ldots,c_m|\ell\bigr) \cdot \ell.
\end{split}
\end{equation*}
Similarly, using \cref{eq:vol-form-neck-curve} in \cref{prop:WP-vol-form-scc},
\begin{equation*}
\begin{split}
&V_\Sigma\bigl(\vec{b};\vec{c}\bigr)
\\\smash{\mathllap{=}}\,& 
\int_{\mathcal{M}_{\Sigma}(\vec{b};\vec{c})}e^{-S(X)}\Omega^{\text{WP}}_{\Sigma}(\vec{b};\vec{c})
\\\smash{\mathllap{=}}\,&  \int_{\mathcal{M}_{\Sigma\setminus\gamma}(\vec{b},\ell_\gamma;\vec{c}\;|\ell_\gamma)\times\RR_{+}\times(\RR/\ell_\gamma\ZZ)}e^{-S(X\setminus\gamma)}\Omega^{\mathrm{WP}}_{\Sigma\setminus\gamma}(\vec{b},\ell_\gamma;\vec{c}\;|\ell_\gamma) \wedge \mathrm{d}\ell_\gamma\wedge \mathrm{d}\tau_\gamma
\\\smash{\mathllap{=}}\,& 
\int_0^\infty \left(\int_0^\ell \left(\int_{\mathcal{M}_{\Sigma'}(\vec{b},\ell;\vec{c}\,')\times\mathcal{M}_{\AA_m}(c_1\dotsc,c_m|\ell)} e^{-S(X')}\cdot e^{-S(Y)} 
\right.\right.\\&\quad\quad\quad\quad\quad\quad\quad\quad\quad\quad\quad\quad\quad\quad\quad\quad\left.\left.\cdot\Omega^{\text{WP}}_{\Sigma'}(\vec{b},\ell;\vec{c}\,')\wedge\Omega^{\text{WP}}_{\AA_m}(c_1\dotsc,c_m|\ell)\right) \mathrm{d}\tau\right) \mathrm{d}\ell
\\\smash{\mathllap{=}}\,&
\int_0^\infty  \left( \int_0^\ell \left( \int_{\mathcal{M}_{\Sigma'}(\vec{b},\ell;\vec{c}\,')} e^{-S(X')}\Omega^{\text{WP}}_{\Sigma'}(\vec{b},\ell;\vec{c}\,') \right.\right.\\
&\quad\quad\quad\quad\quad\quad\quad\quad\quad\quad\quad\quad
\left.\left.  \cdot \int_{\mathcal{M}_{\AA_m}(c_1\dotsc,c_m|\ell)} e^{-S(Y)}\Omega^{\text{WP}}_{\AA_m}(c_1\dotsc,c_m|\ell) \right) \mathrm{d}\tau \right) \mathrm{d}\ell
\\\smash{\mathllap{=}}\,& 
\int_0^\infty 
V_{\Sigma'}\bigl(\vec{b},\ell;\vec{c}\,'\bigr)
\cdot
V_{\mathbb{A}_m}\bigl(c_1,\dots,c_m|\ell\bigr) 
\,\cdot \ell\;\mathrm{d}\ell.
\end{split}
\end{equation*}
\end{proof}

\subsection{Disk recursion}
Cutting out a disk from a flippered surface along an essential arc produces a flippered surface with the same number of necks, but fewer flippers/tines, as well as a disk. The following result indicates how their volumes can be related.

\begin{theorem}
\label{thm:disk-recursion}
Consider $\Sigma=\Sigma'\cup \mathbb{D}_n$ (see \cref{fig:diskrecursion}) with $\vec{c}:=(\vec{c}\,',c_1,\ldots,c_n)$, then
\begin{align}
V_{\Sigma}\bigl(\vec{b};\vec{c}\;|\vec{d}\bigr) 
&= \frac{1}{2} \int_0^\infty V_{\Sigma'}\bigl(\vec{b};\vec{c}\,',\ell|\vec{d}\bigr)\cdot
V_{\DD_{n+1}}\bigl(c_1,\ldots,c_n,\ell\bigr) \, \mathrm{d}\cosh\ell\text{, and }\\
V_{\Sigma}\bigl(\vec{b};\vec{c}\bigr) 
&= \frac{1}{2} \int_0^\infty V_{\Sigma'}\bigl(\vec{b};\vec{c}\,',\ell\bigr)\cdot
V_{\DD_{n+1}}\bigl(c_1,\ldots,c_n,\ell\bigr) \, \mathrm{d}\cosh\ell.
\end{align}
\end{theorem}

\begin{proof}
Let $\eta$ be the arc on $\Sigma$ such that $\Sigma\setminus \eta = \Sigma' \sqcup \DD_{n+1}$. Observe that $\eta$ is invariant
under the action of the mapping class group, since all boundary punctures are fixed. Then, we can apply both the recursion of the Chekhov action (\cref{prop:action-arch-arch}) and the recursion of the Weil--Petersson volume form (\cref{prop:top-rec-vol-form-orthogeod}) with respect to $\eta \subset \Sigma$, and write
\begin{align*}
&V_{\Sigma}\bigl(\vec{b};\vec{c}\;|\vec{d}\bigr)  \\\smash{\mathllap{=}}\,& \int_{\mathcal{M}_{\Sigma}(\vec{b};\vec{c}\;|\vec{d})}e^{-S(X)}\Omega^{\text{WP}}_{\Sigma}(\vec{b};\vec{c}\;|\vec{d})
\\\smash{\mathllap{=}}\,&  \int_{\mathcal{M}_{\Sigma\setminus \eta}(\vec{b};\vec{c}_\eta|\vec{d})\times\RR_+} e^{-S(X\|\eta)}\cdot \sinh^2{(\ell_\eta/2)}\cdot \Omega^{\text{WP}}_{\Sigma\setminus \eta}(\vec{b};\vec{c}_\eta|\vec{d}) \wedge \coth(\ell_\eta/2)\mathrm{d}\ell_\eta
\\\smash{\mathllap{=}}\,&  \frac{1}{2}  \int_{\mathcal{M}_{\Sigma\setminus \eta}(\vec{b};\vec{c}_\eta|\vec{d})\times\RR_+} e^{-S(X\|\eta)} \Omega^{\text{WP}}_{\Sigma\setminus \eta}(\vec{b};\vec{c}_\eta|\vec{d}) \wedge \mathrm{d}\cosh\ell_\eta
\\\smash{\mathllap{=}}\,& \frac{1}{2} \int_{0}^\infty \left( \int_{\mathcal{M}_{\Sigma'}(\vec{b};\vec{c}\,',\ell|\vec{d})\times\mathcal{M}_{\DD_{n+1}}(c_1,\dotsc,c_n,\ell)} e^{-S(X')}\cdot e^{-S(Y)}\right.\\
&\quad\quad\quad\quad\quad\quad\quad\quad\quad\quad\quad\quad\quad\quad\left.
\cdot \Omega^{\text{WP}}_{\Sigma'}(\vec{b};\vec{c}\,',\ell|\vec{c})\wedge \Omega^{\text{WP}}_{\DD_{n+1}}(c_1,\dotsc,c_n,\ell) \right) \mathrm{d}\cosh\ell
\\\smash{\mathllap{=}}\,&  \frac{1}{2}\int_{0}^\infty \left( \int_{\mathcal{M}_{\Sigma'}(\vec{b};\vec{c}\,',\ell|\vec{d})} e^{-S(X')}\Omega^{\text{WP}}_{\Sigma'}(\vec{b};\vec{c}\,',\ell|\vec{c})\right.\\
&\quad\quad\quad\quad\quad\quad\quad\quad\quad\quad
\left.\cdot \int_{\mathcal{M}_{\DD_{n+1}}(c_1,\dotsc,c_n,\ell)} e^{-S(Y)}\Omega^{\text{WP}}_{\DD_{n+1}}(c_1,\dotsc,c_n,\ell) \right) \mathrm{d}\cosh \ell
\\\smash{\mathllap{=}}\,&  \frac{1}{2} \int_0^\infty V_{\Sigma'}\bigl(\vec{b};\vec{c}\,',\ell|\vec{d}\bigr)\cdot
V_{\DD_{n+1}}\bigl(c_1,\ldots,c_n,\ell\bigr) \, \mathrm{d}\cosh\ell.
\end{align*}
The proof for $V_{\Sigma}\bigl(\vec{b};\vec{c}\bigr)$ is identical.
\end{proof}

\begin{proposition}
Let $\Sigma$ be a surface and $\dot{\Sigma}(1,\dotsc,n)$ be the puncture-double of $\Sigma$ at $q_1\dotsc,q_n$. Then the Mirzakhani volume of the moduli space of crowned surfaces $\mathcal{M}_{\dot{\Sigma}(1,\dotsc,n)}(\vec{b};\vec{0})$ equals
\[
V_{\dot{\Sigma}(1,\dotsc,n)}(\vec{b};\vec{0}) = \int_{(1,\infty)^n} \frac{V_{\Sigma}\bigl(\vec{b};\text{arccosh}\;\vec{t}\bigr)}{(1+t_1)\dotsc(1+t_n)} \, \mathrm{d} t_1\wedge\dotsc\wedge\mathrm{d}t_n.
\]
\end{proposition}
\begin{proof}
Consider the collection of $n$ arcs on $\dot{\Sigma}(1,\dotsc,n)$, where each arc bounds $q_i$ together with its double for $i=1,\dotsc,n$. Let $\ell_i$ denotes the hyperbolic length of the $i$-th arc. Then apply the disk recursion ($n$ times) to the arc collection together with the volume expression  $V_{\DD_3}(\ell,0,0) = \frac{2}{1+\cosh\ell}$ and apply the change of variables $t_i=\cosh \ell_i$ for $i=1,\dotsc,n$.
\end{proof}

\begin{example}
Combining the above result with volumes of crowned hyperbolic surfaces obtained in \cite{chekhov2024} and \cite{HT25}, we obtain non-trivial integral identities. For example, for $\Sigma = \DD_3$, by \cite[Example~6.2.2]{chekhov2024},
\begin{equation}
\int_{(1,\infty)^3} \frac{2\, \mathrm{d}t_1\wedge\mathrm{d}t_2\wedge\mathrm{d}t_3}{(1+t_1)(1+t_2)(1+t_3)\sqrt{t_1^2+t_2^2+t_3^2+2t_1t_2t_3-1}}  = \frac{\pi^2}{3}.
\end{equation}
\end{example}
\medskip

\subsection{Crown recursion}
\label{subsec:crown-rec}

\begin{definition}[semi-ideal arcs]
\label{def:semi-ideal-arc}
We refer to an embedded semi-open arc on $\Sigma$ as \textit{semi-ideal} if one endpoint lies on $\partial \Sigma$ while the other ``endpoint'' tends to an interior puncture of $\Sigma$. We identify semi-ideal arcs up to isotopy through semi-ideal arcs, allowing the boundary endpoint to slide along the arch of $\partial\Sigma$. In this paper, we only consider the case $\Sigma=\AA_n$.
\end{definition}

Consider two distinct semi-ideal arcs on $\Sigma = \AA_n$ that share an endpoint at the interior puncture of $\AA_n$.  
Suppose that the other endpoints of the arcs are placed at the arches $\alpha_i, \alpha_j \subset\partial\Sigma$. Cut $\Sigma$ along the semi-ideal arcs and then reglue the closure of each component along the added semi-open boundary intervals in the natural way. This produces two surfaces $\Sigma', \Sigma''$ such that $\Sigma' \simeq \AA_{|i-j|}$ and $\Sigma'' \simeq \AA_{n-|i-j|}$ (see \cref{fig:crownrecursion}). Denoting by $\Sigma'^\circ, \Sigma''^\circ$ the  surfaces before gluing, we obtain $\Sigma = \Sigma'^\circ \cup \Sigma''^\circ$.

Such topological decompositions lead to the following recursion for the Mirzakhani volumes of the moduli spaces of $n$-crowns:

\begin{theorem} 
\label{thm:crown-rec}
Consider $\Sigma  = \AA_{n} = \AA_{n-m}^\circ\cup \AA_{m}^\circ$ (see \cref{fig:crownrecursion})  with $\vec{c}:=(\vec{c}\,',c_1,\ldots,c_m)$, then

\begin{align}
V_\Sigma\bigl(\vec{c}\;|d\bigr) 
&=
\int_\RR V_{\AA_{n-m}}\bigl(\vec{c}\,'|\;|d-\ell|\bigr)
\cdot 
V_{\mathbb{A}_m}\bigl(c_1,\ldots,c_m|\;|\ell|\bigr) \;\mathrm{d}\ell
\text{, and}\\
V_\Sigma\bigl(\vec{c}\bigr) 
&= 
2\cdot V_{\AA_{n-m}}\bigl(\vec{c}\,'\bigr) \cdot V_{\AA_{m}}\bigl(c_1,\ldots,c_m\bigr).
\end{align}
\end{theorem}

\begin{proof}
Consider two semi-ideal arcs on $\Sigma$ with endpoints on $\alpha_1,\alpha_{m+1} \subset \partial \Sigma$. Given a flippered surface $X=(\Sigma,h)$, there is a unique representative of each arc which is perpendicular to $\partial X$ at one of its endpoints and, if $p_1$ is geometrised as a flare, spirals towards the cuff of $X$ (in the arbitrary chosen direction). Upon cutting and regluing, we obtain two hyperbolic metrics $X'=(\Sigma',h')$ and $X''=(\Sigma'',h'')$. These metrics are (generically) incomplete, and we consider  the Nielsen extensions of their metric completions, which we also denote by $X',X''$. 

It is possible for spiraling directions to change after cutting and regluing (see \cref{fig:crownrecursion}). To see this, consider a triangulation $\triangle$ of $\Sigma=\AA_n$ where every arc joins a tine/flipper to the unique puncture on $\Sigma$. A change in spiraling happens precisely when the sum of the shearing parameters for a subsequence of adjacent arcs have a different sign to the total. Given this nicety, it is more convenient to work with \emph{enhanced} \cite{liu2009quantum, bonahon2007representations} or \emph{holed} \cite{fock1997dual, fock2007dual} moduli spaces $\mathcal{M}_\Sigma^{\pm}$, which parametrize pairs consisting of a surface $X\in\mathcal{M}_\Sigma$ and a choice of orientation for all cuffs. In particular, if $d>0$, the moduli space $\mathcal{M}^{\pm}_\Sigma(\vec{c}\;|d)$ has two components, which we denote by $\mathcal{M}_\Sigma(\vec{c}\;|d)$ and $\mathcal{M}_\Sigma(\vec{c}\;|-d)$,  with the orientation in agreement with the one in \cref{subsec:shears}. 

The process described in the first paragraph induces a map 
\begin{equation}
\natural: \mathcal{M}^{\pm}_\Sigma(\vec{c}\;|d) \to \mathcal{M}^{\pm}_{\Sigma''}(c_1,\dotsc,c_m)\times \mathcal{M}^{\pm}_{\Sigma'}(\vec{c}\,'),
\end{equation}
and we now describe it terms of shearing coordinates. Suppose that all strap lengths in $\vec{c}$ are strictly positive (degenerate cases are similar). Consider the puncture-double surface $\dot{\Sigma}(1,\dotsc,n)$ (see \cref{not:doubling}) and its ideal triangulation by $2n$ ideal arcs emanating from the interior puncture, and let $\{s_1,\dotsc,s_{2n}\}$ be the associated shearing coordinates. 
By \cref{eq:1-flip-crown-system}, the shears satisfy: 
\begin{equation}
\label{eq:n-crown-strap-length}
(1+e^{s_{2i-1}})(1+e^{-s_{2i}}) = \cosh^2(c_i/2)
\end{equation}
for $i=1,\dotsc,n$. Furthermore, $\sum_{j=1}^{2n}s_j=\pm d$, depending on the component of $\mathcal{M}^{\pm}_\Sigma(\vec{c}\;|d)$.

Consider the ideal triangulations of $\dot{\Sigma}'(m+1,\dotsc,n)$ and $\dot{\Sigma}''(1,\dotsc,m)$ given by the restrictions of the chosen ideal triangulation of $\dot{\Sigma}(1,\dotsc,n)$.
Then in the associated shearing coordinates, the map $\natural$ is given by
\begin{equation}
\natural(s_1,\dotsc,s_{2n}) =  (( s_{1},\dotsc,  s_{2m}),( s_{2m+1},\dotsc, s_{2n})).
\end{equation}
It follows that $\natural$ is a real-analytic embedding, and its image is given by the union of all subsets $\mathcal{M}_{\Sigma''}(c_1,\dotsc,c_m|d'')\times\mathcal{M}_{\Sigma'}(\vec{c}\,'|d')\subset \mathcal{M}_{\Sigma''}(c_1,\dotsc,c_m)\times\mathcal{M}_{\Sigma'}(\vec{c}\,')$  such that $d'+d''=  d$.

It also follows that $\natural$ satisfies
\begin{equation}
\Omega^{\text{WP}}_\Sigma(\vec{c}\;|d)  =  (\natural)^*\left(\Omega^{\text{WP}}_{\Sigma'}(\vec{c}\,'|d') \wedge \Omega^{\text{WP}}_{\Sigma''}(c_1,\dotsc,c_m|d'')\right)\wedge \mathrm{d}d'', 
\end{equation}
where we let $ \Omega^{\text{WP}}_\Sigma(\vec{c}\;|-d) := \Omega^{\text{WP}}_\Sigma(\vec{c}\;|d)$. Equivalently, 
\begin{equation}
\Omega^{\text{WP}}_\Sigma(\vec{c}\;|d)  =  (\natural)^*\left(\Omega^{\text{WP}}_{\Sigma'}(\vec{c}\,'|\; |d-\ell|) \wedge \Omega^{\text{WP}}_{\Sigma''}(c_1,\dotsc,c_m|\;|\ell|)\right)\wedge \mathrm{d}\ell.
\end{equation}

Since the geometric operation underpinning $\natural$ always cuts $\AA_n$ and reassembles it in such a way that a severed boundary arch is reglued to have the same (in fact unique) complete geodesic extension, the extremal orthojectories defining essential intervals remain unchanged for each arch (although the process cuts two of the essential intervals are somewhere in the interior before reassembling them). Therefore, $S(X)= S(X')+S(X'')$. Hence,
\begin{align*}
&V_\Sigma\bigl(\vec{c}\;|d\bigr) 
\\\smash{\mathllap{=}}\,&\int_{\mathcal{M}_\Sigma\bigl(\vec{c}\;|d\bigr)} e^{-S(X)}\Omega^{\text{WP}}_\Sigma\bigl(\vec{c}\;|d\bigr)
\\\smash{\mathllap{=}}\,& \int_{\mathcal{M}_{\Sigma'}(\vec{c}\,'|\;|d-\ell|)\times\mathcal{M}_{\Sigma''}(c_1,\dotsc,c_m|\;|\ell|)\times\RR} e^{-S(X')}\cdot e^{-S(X'')}
\\
&\quad\quad\quad\quad\quad\quad\quad\quad\quad\quad\quad\quad\quad\quad\cdot\Omega^{\text{WP}}_{\Sigma'}(\vec{c}\,'|\;|d-\ell|) \wedge \Omega^{\text{WP}}_{\Sigma''}(c_1,\dotsc,c_m|\;|\ell|)\wedge \mathrm{d}\ell
\\\smash{\mathllap{=}}\,& \int_{\RR} \left( \int_{\mathcal{M}_{\Sigma'}(\vec{c}\,'|\;|d-\ell|)} e^{-S(X')} \Omega^{\text{WP}}_{\Sigma'}(\vec{c}\,'|\;|d-\ell|) \right.\\
&\quad\quad\quad\quad\quad\quad\quad\quad\quad\quad\quad
\left. \cdot \int_{\mathcal{M}_{\Sigma''}(c_1,\dotsc,c_m|\;|\ell|)} e^{-S(X'')}  \Omega^{\text{WP}}_{\Sigma''}(c_1,\dotsc,c_m|\;|\ell|)\right) \mathrm{d}\ell
\\\smash{\mathllap{=}}\,& \int_\RR V_{\AA_{n-m}}\bigl(\vec{c}\,'|\;|d-\ell\bigr|)
\cdot 
V_{\mathbb{A}_m}\bigl(c_1,\ldots,c_m|\;|\ell|\bigr) \;\mathrm{d}\ell.
\end{align*}
For unfixed neck length case, note that the map 
\[\natural: \mathcal{M}^{\pm}_\Sigma(\vec{c}) \to \mathcal{M}^{\pm}_{\Sigma''}(c_1,\dotsc,c_m)\times \mathcal{M}^{\pm}_{\Sigma'}(\vec{c}\,')
\]
is a real-analytic diffeomorphism. Let $\Omega^{\text{WP},\pm}_\Sigma(\vec{c})$ be the top degree form on $\mathcal{M}^{\pm}_\Sigma(\vec{c})$ obtained as the pullback under the branched cover $\mathcal{M}^{\pm}_\Sigma(\vec{c}) \to \mathcal{M}_\Sigma(\vec{c})$. Then
\[
\Omega^{\text{WP},\pm}_\Sigma(\vec{c}) = (\natural)^*\left(\Omega^{\text{WP},\pm}_{\Sigma'}(\vec{c}\,') \wedge \Omega^{\text{WP},\pm}_{\Sigma''}(c_1,\dotsc,c_m)\right)
\]
away from a subset of measure zero, and therefore
\begin{align*}
&V_{\Sigma}(\vec{c})
\\\smash{\mathllap{=}}\,&\int_{\mathcal{M}_\Sigma(\vec{c})} e^{-S(X)}\Omega^{\text{WP}}_\Sigma(\vec{c})
\\\smash{\mathllap{=}}\,& \frac{1}{2}\int_{\mathcal{M}^{\pm}_\Sigma(\vec{c})} e^{-S(X)}\Omega^{\text{WP},\pm}_\Sigma(\vec{c})
\\\smash{\mathllap{=}}\,& \frac{1}{2} \int_{\mathcal{M}^{\pm}_{\Sigma'}(\vec{c}\,')\times \mathcal{M}^{\pm}_{\Sigma''}(c_1,\dotsc,c_m)} e^{-S(X')}\cdot e^{-S(X'')} \Omega^{\text{WP},\pm}_{\Sigma'}(\vec{c}\,') \wedge \Omega^{\text{WP},\pm}_{\Sigma''}(c_1,\dotsc,c_m)
\\\smash{\mathllap{=}}\,&  \frac{1}{2} \int_{\mathcal{M}^{\pm}_{\Sigma'}(\vec{c}\,')} e^{-S(X')}\Omega^{\text{WP},\pm}_{\Sigma'}(\vec{c}\,')\cdot \int_{\mathcal{M}^{\pm}_{\Sigma''}(c_1,\dotsc,c_m)}e^{-S(X'')} \Omega^{\text{WP},\pm}_{\Sigma''}(c_1,\dotsc,c_m)
\\\smash{\mathllap{=}}\,&  2 \int_{\mathcal{M}_{\Sigma'}(\vec{c}\,')} e^{-S(X')}\Omega^{\text{WP}}_{\Sigma'}(\vec{c}\,')\cdot \int_{\mathcal{M}_{\Sigma''}(c_1,\dotsc,c_m)}e^{-S(X'')} \Omega^{\text{WP}}_{\Sigma''}(c_1,\dotsc,c_m)
\\\smash{\mathllap{=}}\,& 2\cdot V_{\AA_{n-m}}\bigl(\vec{c}\,'\bigr) \cdot V_{\AA_{m}}\bigl(c_1,\ldots,c_m\bigr).
\end{align*}
\end{proof}

\subsection{Flippered $n$-gons and the Mehler-Fock transform}
\label{subsec:vol-n-gon-1-int}

The goal of this subsection is to provide a simple
expression for the Mirzakhani volumes $V_{\DD_n}(\vec{c})$ for $n\geqslant 3$ (\cref{thm:vol-n-gon-1-int}). To do that, we first express $V_{\DD_n}(\vec{c})$ as an iterated integral (\cref{eq:vol-flip-n-gon-recurs}) via the following steps:
\begin{itemize}
\item consider a maximal arc collection on $\DD_n$ (\cref{def:gen-triang}), and
\item apply the disk recursion to the decomposition $\DD_n = \DD_3 \cup\dotsc\cup \DD_3$ (\cref{thm:disk-recursion}), and
\item insert the explicit value of $V_{\DD_3}(\vec{c})$ (\cref{thm:vol-fli-triang}).
\end{itemize}
After that, we follow the strategy utilized in the crowned case to compute $V_{\DD_n}(\vec{0})$ in \cite[Theorem~6.1]{HT25}, which is to diagonalize the iterated integration with the Mehler-Fock transform (see the comments in \cite{MO402985} for the relevant discussion). 

For the diagonalization step, we will also require a new ingredient --- a product formula for the conical (Mehler) functions $P_{-1/2+\sqrt{-1}\xi}(x)$ \cite[Section~7.3]{MR350075}. We were unable to find an exact statement of the formula in the literature, and we provide it here together with a proof:

\begin{lemma}[a product formula for conical functions]
\label{prop:prod-formula-con-func}
For any $x,y\geqslant 1$ and $\xi \in \RR$,
\begin{equation}
\label{eq:prod-form-con-func}
\int_1^\infty \frac{P_{-1/2+\sqrt{-1}\xi}(z)\mathrm{d}z}{\sqrt{x^2+y^2+z^2+2xyz-1}} = \frac{\pi}{\cosh \pi\xi}\cdot P_{-1/2+\sqrt{-1}\xi}(x) P_{-1/2+\sqrt{-1}\xi}(y).
\end{equation}
\end{lemma}

\begin{proof}
Let $\Sigma = \AA_2$. Consider the essential arc on $\Sigma$ with endpoints on the arch $\alpha_1$. It decomposes $\Sigma$ into $\AA_1$ and $\DD_3$. Then by the disk recursion (\cref{thm:disk-recursion}) and \cref{thm:vol-fli-triang}, \cref{thm:vol-1-crown}, 
\begin{equation}
\label{eq:vol-2-crown-disk-rec}
\begin{split}
&{V}_{\Sigma}(\vec{c}\;|d) \\
\smash{\mathllap{=}}\,& \frac{1}{2}\int_0^\infty  {V}_{\AA_1}(\ell|d)\cdot{V}_{\DD_3}(c_1,c_2,\ell) \,\mathrm{d}\cosh\ell 
\\
\smash{\mathllap{=}}\,& \frac{1}{2}\int_0^\infty  \frac{1}{\sqrt{2\,\text{ch}\,\ell+2\,\text{ch}\, d}} \cdot \frac{2}{\sqrt{\text{ch}^2c_1+\text{ch}^2c_2+\text{ch}^2\ell+2\,\text{ch}\, c_1 \text{ch}\,c_2 \text{ch}\, \ell-1}}\,\mathrm{d}\cosh\ell.
\end{split}
\end{equation} 
Now by the crown recursion (\cref{thm:crown-rec}) and  \cref{thm:vol-1-crown}:
\begin{equation}
\label{eq:vol-2-crown-crown-rec}
\begin{split}
V_{\Sigma}(\vec{c}\;|d) &= \int_\RR V_{\AA_1}(c_1|\;|d-\ell|)V_{\AA_1}(c_1|\;|\ell|)\mathrm{d}\ell
\\& = \int_\RR \frac{\mathrm{d}\ell}{\sqrt{(2\cosh c_1+2\cosh(d-\ell))(2\cosh c_2 +2\cosh\ell)}}
\\& = \left(\frac{1}{\sqrt{2\cosh x+2\cosh c_1}}*\frac{1}{\sqrt{2\cosh x+2\cosh c_2}}\right)(d).
\end{split}
\end{equation}
By \cite[Equation~7.4.6]{MR350075},
\begin{equation}
\begin{split}
\mathcal{F} \left[\frac{1}{\sqrt{2\cosh x +2\cosh c}}\right](\xi) &=  \int_{\RR} \frac{e^{-\sqrt{-1}\xi x}}{\sqrt{2\cosh x+2\cosh c}}\mathrm{d} x
\\& =2\int_{0}^\infty \frac{\cos \xi x}{\sqrt{2\cosh x+2\cosh c}}\mathrm{d} x
\\& =\pi\cdot \frac{P_{-1/2+\sqrt{-1}\xi}(\cosh c)}{\cosh \pi\xi},
\end{split}
\end{equation}  
where $\mathcal{F}$ is the Fourier transform: $\mathcal{F}[f](\xi) = \int_{\RR}f(x)e^{-\sqrt{-1}\xi x}\mathrm{d}x$.
Applying the Fourier transform with respect to $d$ to both \cref{eq:vol-2-crown-disk-rec} and \cref{eq:vol-2-crown-crown-rec}, we then obtain an integral identity:
\begin{equation}
\begin{split}
&\frac{\pi}{\cosh \pi\xi}\int_0^\infty\frac{P_{-1/2+\sqrt{-1}\xi}(\cosh\ell)\mathrm{d}\cosh\ell}{\sqrt{\cosh^2c_1+\cosh^2c_2+\cosh^2\ell+2\cosh c_1\cosh c_2\cosh\ell-1}} 
\\&= \left(\frac{\pi}{\cosh \pi\xi}\right)^2 P_{-1/2+\sqrt{-1}\xi}(\cosh c_1)\cdot P_{-1/2+}(\cosh c_2),
\end{split}
\end{equation}
where in the right-hand side we used the convolution theorem. Dividing both sides by $\frac{\pi}{\cosh \pi\xi}$ and letting $x=\cosh c_1,y=\cosh c_2, z=\cosh \ell$ finishes the proof.
\end{proof}

\begin{remark}
The formula reduces to the classical Mehler's formula \cite[Equation~44]{HT25} when $y=1$. For a different proof of \cref{prop:prod-formula-con-func} that uses the Mehler's formula and another known identity for conical functions, see \cref{app:another-proof}.
\end{remark}

\begin{remark}
\cref{prop:prod-formula-con-func} has the same structure as \cite[Equation~44]{goncharov-sun}, where the conical Mehler functions in \cref{prop:prod-formula-con-func} are (integral transforms of) volumes for $\mathbb{A}_1$ (a trouser leg in \cite{goncharov-sun}, also referred to as half-pants in other texts) and formally corresponds to the $\mathcal{B}_T$-term in \cite[Equation~44]{goncharov-sun}. Similarly, our $\DD_3$ volume corresponds to their $\mathcal{E}_\tau$-term.  It is natural to wonder how this correspondence might fit in the growing theory of ``duality'' \cite{fock1997dual,fock2006moduli,goncharov2015geometry} between the two moduli spaces central to these two papers.
\end{remark}

Armed with \cref{prop:prod-formula-con-func}, we 
prove:

\begin{theorem}
\label{thm:vol-n-gon-1-int}
For $n\geqslant 3$,
\begin{equation}
{V}_{\DD_n}(c_1,\dotsc,c_n) = \pi^{n-2} \int_0^\infty \xi \sinh{2\pi\xi} \cdot \prod_{i=1}^n \frac{P_{-1/2+\sqrt{-1}\xi}(\cosh c_i)}{\cosh{\pi\xi}} \,\mathrm{d}\xi.
\end{equation}
\end{theorem}
\begin{proof}
Consider a maximal arc collection on $\DD_n$ (\cref{def:gen-triang}). Any collection would work, but for concreteness consider the ``zig-zag'' collection, i.e. one that breaks down the boundary punctures into $\{q_1,q_2\}, \{q_3\},\dotsc,\{q_{n-2}\},\{q_{n-1},q_n\}$ as witnessed by the complementary regions (\cref{lem:gen-triang-compl-regions}).
Then by disk recursion (\cref{thm:disk-recursion}) applied to the chosen maximal arc collection and \cref{thm:vol-fli-triang},
\begin{equation}
\label{eq:vol-n-fon-big-int}
\begin{split}
&{V}_{\DD_n}(c_1,\dotsc,c_n) 
\\\smash{\mathllap{=}}\,& 2 \int_{(0,\infty)^{n-3}} \frac{\mathrm{d}\cosh\ell_1 \wedge \dotsc \wedge \mathrm{d}\cosh\ell_{n-3}}{\sqrt{D(c_1,c_2,\ell_1)D(\ell_1,\ell_2,c_3)\dotsc D(\ell_{n-4},\ell_{n-3},c_{n-2})D(\ell_{n-3},c_{n-1},c_{n})}}, 
\end{split}
\end{equation}
where $\ell_1,\dotsc,\ell_{n-3}$ are the hyperbolic lengths of the orthogeodesic representatives of the arcs.

For $a\geqslant 1$, consider the integral operator $T_a$ that acts on $f:[1,\infty)\to \RR$ by
\begin{equation}
    T_a[f](y) = \int_1^\infty \frac{f(x)}{\sqrt{x^2+y^2+a^2+2xya-1.}}\,\mathrm{d}x,
\end{equation}
provided that the integral exists for all $y\in[1,\infty).$ Then it follows from \cref{eq:vol-n-fon-big-int} and the Fubini–Tonelli theorem that
\begin{equation}
\label{eq:vol-flip-n-gon-recurs}
\begin{split}
&V_{\DD_n}(c_1,\dotsc,c_n) 
\\\smash{\mathllap{=}}\,& 2 \cdot T_{\text{ch}\, c_{n-1}}\circ \dotsc\circ T_{\text{ch} \, c_{3}}\left[\frac{1}{\sqrt{\text{ch}^2 c_1+\text{ch}^2 c_2+x^2+2\,\text{ch}\, c_1 \text{ch}\, c_2\cdot x-1}}\right](\text{ch}\,c_n).
\end{split}
\end{equation}
We follow the strategy as in the proof of \cite[Theorem~6.1]{HT25}, and use the Mehler-Fock transform \cite[Section~1.9]{MR2254107}: $\mathcal{MF}[f](\xi)=\int_1^\infty f(x)P_{-1/2+\sqrt{-1}\xi}(x)\mathrm{d}x$. If the Fubini–Tonelli theorem applies, then by \cref{prop:prod-formula-con-func},
\begin{equation}
\begin{split}
\mathcal{MF}[T_a f](\xi) &= \int_1^\infty \left(\int_1^\infty \frac{f(x)}{\sqrt{x^2+y^2+a^2+2xya-1}}\,\mathrm{d}x\right) \cdot P_{-1/2+\sqrt{-1}\xi} (y)\,\mathrm{d}y \\& = \int_1^\infty f(x) \left(\int_1^\infty \frac{P_{-1/2+\sqrt{-1}\xi }(y)\,\mathrm{d}y}{\sqrt{x^2+y^2+a^2+2xya-1}} \right) \mathrm{d}x  \\& = \frac{\pi}{\cosh{\pi\xi}}\cdot P_{-1/2+\sqrt{-1}\xi}(a) \mathcal{MF}[f](\xi).
\end{split}
\end{equation}
and more generally, for $a_1,\dotsc,a_k\geqslant 1$
\begin{equation}
\begin{split}
\mathcal{MF}[T_{a_k}\circ\dotsc \circ T_{a_1}f](\xi) = \frac{\pi^k}{\cosh^k \pi \xi} \left(\prod_{i=1}^k P_{-1/2+\sqrt{-1}\xi}(a_i) \right) \mathcal{MF}[f](\xi).
\end{split}
\end{equation}
We now justify that the Fubini--Tonelli theorem does apply to our setting as in \cref{eq:vol-flip-n-gon-recurs}. Note that 
\begin{equation}
\label{eq:bound-for-D}
\begin{split}
D(\alpha,\beta,\gamma) &= \cosh^2 \alpha + \cosh^2 \beta +\cosh^2 \gamma +2\cosh \alpha\cosh\beta\cosh\gamma-1
\\& \geqslant \cosh^2 \alpha + \cosh^2 \beta +1+2\cosh\alpha\cosh\beta-1
\\& = (\cosh \alpha+\cosh \beta)^2.
\end{split}
\end{equation}
Also, $D(\alpha,\beta,\gamma)\geqslant(1+\cosh\gamma)^2$. Then from \cref{eq:vol-n-fon-big-int},
\begin{equation}
\begin{split}
&{V}_{\DD_n}(c_1,\dotsc,c_n) 
\\\smash{\mathllap{\leqslant}}\,& 2 \int_{(0,\infty)^{n-3}} \frac{\mathrm{d}\,\text{ch}\,\ell_1 \wedge \dotsc \wedge \mathrm{d}\,\text{ch}\,\ell_{n-3}}{(1+\text{ch}\,\ell_1)(\text{ch}\,\ell_1+\text{ch}\,\ell_2)\dotsc (\text{ch}\,\ell_{n-4}+\text{ch}\,\ell_{n-3})(\text{ch}\,\ell_{n-3}+\text{ch}\,c_n)}. 
\end{split}
\end{equation}
Then the results of \cite[Appendix~B]{HT25} apply and the Fubini--Tonelli theorem holds, and therefore,
\begin{equation}
\begin{split}
&\mathcal{MF}\left(V_{\DD_n}(c_1,\dotsc,c_{n-1},\text{arccosh}\,y)\right)(\xi) 
\\\smash{\mathllap{=}}\,& 2\cdot\frac{\pi^{n-2}}{\cosh^{n-2}\pi\xi} \prod_{i=1}^{n-1} P_{-1/2+\sqrt{-1}\xi}(\cosh c_i).
\end{split}
\end{equation}
Then since inverse Mehler-Fock transform is given by \cite[Equation~1.9.11]{MR2254107}
\begin{equation*}
\mathcal{MF}^{-1}[F](x) = \int_0^\infty \xi\tanh\pi\xi\cdot F(\xi)P_{-1/2+\sqrt{-1}\xi}(x)\,\mathrm{d}\xi,
\end{equation*}
we have
\begin{equation}
\label{eq:vol-flip-n-gon-last}
\begin{split}
V_{\DD_n}(c_1,\dotsc,c_n) &= \mathcal{MF}^{-1}\left(2\cdot\frac{\pi^{n-2}}{\cosh^{n-2}\pi\xi} \prod_{i=1}^{n-1} P_{-1/2+\sqrt{-1}\xi}(\cosh c_i)\right)(\cosh c_n)
\\& = 2\pi^{n-2} \int_0^\infty \frac{\xi \sinh{\pi\xi}}{\cosh^{n-1}\pi\xi} \prod_{i=1}^n P_{-1/2+\sqrt{-1}\xi}(\cosh c_i)\, \mathrm{d}\xi
\\& = \pi^{n-2} \int_0^\infty \xi \sinh{2\pi\xi} \cdot \prod_{i=1}^n \frac{P_{-1/2+\sqrt{-1}\xi}(\cosh c_i)}{\cosh{\pi\xi}} \,\mathrm{d}\xi.
\end{split}
\end{equation}

\end{proof}

\begin{remark}
The proof of \cref{thm:vol-n-gon-1-int} gives a geometric meaning (in the sense of the steps outlined in the beginning of \cref{subsec:vol-n-gon-1-int}) to the integral obtained by a formal change of variables in \cite[Equation~40]{HT25}:
\begin{equation}
V_{\DD_n}(\vec{0}) = 2\int_{(1,\infty)^{n-3}} \frac{\mathrm{d}x_2\wedge\cdots\wedge \mathrm{d}x_{n-2}}{(1+x_2)(x_2+x_3)\dotsc(x_{n-3}+x_{n-2})(x_{n-2}+1)}. 
\end{equation}
It agrees with \cref{eq:vol-n-fon-big-int} if we let $x_i = \cosh \ell_{i-1}$ and $\vec{c}=\vec{0}$, and equals
\begin{equation}
\begin{split}
\frac{1}{2^{n-3}}\int_{(0,\infty)^{n-3}}& V_{\DD_3}(0,0,\ell_1)V_{\DD_3}(\ell_1,\ell_2,0)\cdots V_{\DD_3}(\ell_{n-4},\ell_{n-3},0)V_{\DD_3}(\ell_{n-3},0,0)
\\&\mathrm{d}\cosh\ell_1\wedge\cdots\wedge \mathrm{d}\cosh\ell_{n-3}.
\end{split}   
\end{equation}
\end{remark}

\begin{remark}
\label{rmk:vol-n-gon-1-int}
Compare \cref{eq:vol-flip-n-gon-last} with \cite[Equation~48]{HT25}: if all $c_i=0$ (i.e. the case of the crowned $n$-gons), then $P_{-1/2+\sqrt{-1}\xi}(\cosh c_i)=1$, and the expressions agree.
\end{remark}

\subsection{Flippered $n$-crowns and the Fourier transform}
\label{subsec:vol-n-crown-1-int}

We provide a non-recursive single integral expression for the Mirzakhani volume $V_{\AA_n}(\vec{c}\;|d)$. To do that, we
\begin{itemize}
    \item successively apply the crown recursion (\cref{thm:crown-rec}), and insert the
    \item explicit value of $V_{\AA_1}(c|d)$ (\cref{thm:vol-1-crown}),
\end{itemize}
to obtain a convolution expression for $V_{\AA_n}(\vec{c}\;|d)$ (\cref{eq:vol-n-crown-conv}). After that, we apply the Fourier transform, similar to the crowned case  $V_{\AA_n}(\vec{0}\;|d) $ in \cite[Theorem~4.5]{HT25}. 

\begin{remark}
The convolution step here also arises in the proof of \cite[Theorem~4.5]{HT25}, and is precisely the formula that one obtains after a consequence of crown recursions.
\end{remark}

\begin{theorem}
\label{thm:vol-n-crown-1-int}
For $n\geqslant 1$, 
\begin{equation}
V_{\AA_n} (c_1,\dotsc,c_n|d)  = \pi^{n-1}\int_0^\infty \cos(\omega d) \cdot\prod_{i=1}^n \frac{P_{-1/2+\sqrt{-1}\omega}(\cosh c_i)}{\cosh \pi\omega} \,\mathrm{d}\omega.
\end{equation}
\end{theorem}

\begin{proof}
For a fixed $\vec{c}$, formally extend the function $V_{\AA_n}(\vec{c}\;|d)$ from the domain $d\geqslant0$ to $d\in\RR$ as even function. Then \cref{thm:crown-rec} says that 
\begin{equation}
V_{\AA_n}\bigl(\vec{c}\;|d\bigr) 
=
\int_\RR V_{\AA_{n-m}}\bigl(\vec{c}\,'|d-\ell\bigr)
\cdot 
V_{\mathbb{A}_m}\bigl(c_1,\ldots,c_m|\ell\bigr) \;\mathrm{d}\ell,
\end{equation}
and more generally
\begin{equation}
V_{\AA_n}\bigl(\vec{c}\;|d\bigr) 
=
\int_{\RR^{n-1}} V_{\AA_{1}}\Bigl(c_1\Big|d-\sum_{i=2}^{n}\ell_{i}\Bigr)\cdot V_{\AA_1}(c_2|\ell_2)
\cdot \dotsc\cdot
V_{\mathbb{A}_1}(c_n|\ell_n) \;\mathrm{d}\ell_2\wedge\dotsc\wedge\mathrm{d}\ell_n.
\end{equation}
Therefore $V_{\AA_n}\bigl(\vec{c}\;|d\bigr)$ is a convolution, specifically:
\begin{equation}
V_{\AA_n}(c_1,\dotsc, c_n|d) = \left(V_{\AA_1}(c_1|\ell_1)*\dotsc*V_{\AA_1}(c_n|\ell_n)\right)(d)
\end{equation}
and by \cref{thm:vol-1-crown},
\begin{equation}
\label{eq:vol-n-crown-conv}
\begin{split}
V_{\AA_n}(c_1,\dotsc, c_n|d)
= \left(\frac{1}{\sqrt{2\cosh \ell_1+2\cosh c_1}}*\dotsc*\frac{1}{\sqrt{2\cosh \ell_n+2\cosh c_n}}\right)(d).
\end{split}
\end{equation}
Note that $\frac{1}{\sqrt{2\cosh \ell + 2\cosh c}} \in L^1(\RR)$. Then by the Young’s convolution inequality and the convolution theorem, 
\begin{equation}
\begin{split}
V_{\AA_n}(c_1,\dotsc, c_n|d) = \mathcal{F}^{-1} \left[\prod_{i=1}^n \mathcal{F} \left[\frac{1}{\sqrt{2\cosh\ell_i +2\cosh c_i}}\right](\omega)\right](d),
\end{split}
\end{equation}
where 
$\mathcal{F}[f](\omega) = \int_{\RR}f(x)e^{-\sqrt{-1}\omega x}\mathrm{d}x$.
Then, since $P_{-1/2+\sqrt{-1}\omega}(x)=P_{-1/2-\sqrt{-1}\omega}(x)$ \cite[Equation~7.4.6]{MR350075} and $\mathcal{F}^{-1}[f](x) = \frac{1}{2\pi} \int_\RR f(\omega)e^{\sqrt{-1}\omega x}\mathrm{d}\omega$, we conclude
\begin{equation}
\begin{split}
V_{\AA_n}(c_1,\dotsc, c_n|d)  &= \mathcal{F}^{-1} \left[\pi^n\prod_{i=1}^n \frac{P_{-1/2+\sqrt{-1}\omega}(\cosh c_i)}{\cosh \pi\omega} \right]
\\& = \frac{1}{2\pi}\cdot\pi^n\int_\RR  e^{\sqrt{-1}\omega d}\cdot\prod_{i=1}^n \frac{P_{-1/2+\sqrt{-1}\omega}(\cosh c_i)}{\cosh \pi\omega} \, \mathrm{d}\omega
\\& = \pi^{n-1}\int_0^\infty \cos(\omega d) \cdot\prod_{i=1}^n \frac{P_{-1/2+\sqrt{-1}\omega}(\cosh c_i)}{\cosh \pi\omega} \,\mathrm{d}\omega.
\end{split}
\end{equation}

\end{proof}

\begin{remark}
Compare the result of \cref{thm:vol-n-crown-1-int}
with the expression immediately below Equation 22 in \cite{HT25}: up to a scalar change of variables they differ by the insertion of  $\prod_{i=1}^n P_{-1/2+\sqrt{-1}\omega}(\cosh c_i) $ under the integral sign. This mirrors the observation for $n$-gons made in \cref{rmk:vol-n-gon-1-int}.
\end{remark}

\subsection{Flippered $n$-crowns without neck constraints}

We use the crown recursion (\cref{thm:crown-rec}) to prove:

\begin{corollary}
\label{cor:vol-n-crown-no-neck}
For $n\geqslant 1$,
\begin{equation*}
V_{\AA_n} (c_1,\dotsc,c_n)  = 2^{n-1} \prod_{i=1}^n \frac{K\left(\tanh(c_i/2)\right)}{\cosh(c_i/2)},
\end{equation*}
where the function $K(s):=\int_0^1\frac{\mathrm{d}t}{\sqrt{(1-t^2)(1-s^2t^2)}}$ here denotes the complete elliptic integral of the first kind.     
\end{corollary}
\begin{proof}
By \cref{thm:crown-rec},
\begin{equation*}
V_{\AA_n}\bigl(\vec{c}\bigr) 
= 
2\cdot V_{\AA_{n-m}}\bigl(\vec{c}\,'\bigr) \cdot V_{\AA_{m}}\bigl(c_1,\ldots,c_m\bigr),
\end{equation*}
hence 
\begin{equation*}
V_{\AA_n}(c_1,\dotsc,c_n) = 2^{n-1} \prod_{i=1}^n V_{\AA_1}(c_i),
\end{equation*}
and it suffices to compute $V_{\AA_1}(c)$. To this end, let $t=\tanh{d/2}$. Then $\mathrm{d}t=\frac{1}{2}\text{sech}^2(d/2)\mathrm{d}d$ and $\frac{\mathrm{d}t}{\sqrt{1-t^2}} = \frac{1}{2}\text{sech}(d/2)\mathrm{d}d$. Then 
\begin{align*}
V_{\AA_1}(c)=\int_0^\infty V_{\AA_1}(c|d)\,\mathrm{d}d&=\int_0^\infty \frac{\mathrm{d}d}{2\sqrt{\cosh^2(d/2)+\sinh^2{(c/2)}}}  \\&= \int_0^\infty \frac{\text{sech}(d/2) \,\mathrm{d}d}{2\sqrt{1+\sinh^2(c/2)\text{sech}^2(d/2)}}  \\& =
\int_0^1 \frac{\mathrm{d}t}{\sqrt{(1-t^2)\left(1+\sinh^2(c/2)(1-t^2)\right)}}  \\&= \frac{1}{\cosh(c/2)}\int_0^1 \frac{\mathrm{d}t}{\sqrt{(1-t^2)\left(1-\tanh^2(c/2)t^2\right)}}   \\&= \frac{K(\tanh(c/2))}{\cosh(c/2)}.
\end{align*}
\end{proof}

\begin{remark}
At present, the volumes $V_{\AA_n}(\vec{c})$ form the only infinite family of Mirzakhani volumes for which we have obtained closed-form expressions valid for arbitrary $\vec{c}$.
\end{remark}

\subsection{Finiteness of the Mirzakhani volumes}

Using the results of \cref{subsec:vol-n-gon-1-int} and \cref{subsec:vol-n-crown-1-int}, we show:

\begin{theorem}
\label{thm:finvolflipper}
The Mirzakhani volumes of the moduli spaces of flippered surfaces are finite. Moreover, if $c_i\geqslant c_i'$, then
\begin{equation}
\label{eq:finvolflipper-bound}
\begin{split}
V_\Sigma(\vec{b};c_1,\dotsc c_i,\dotsc,c_n|\vec{d}) 
& \leqslant 
V_\Sigma(\vec{b};c_1,\dotsc c_i',\dotsc,c_n|\vec{d}) \\
V_\Sigma(\vec{b};c_1,\dotsc c_i,\dotsc,c_n) 
& \leqslant 
V_\Sigma(\vec{b};c_1,\dotsc c_i',\dotsc,c_n).
\end{split}
\end{equation}
\end{theorem}

\begin{proof}
We first consider the case $\Sigma = \DD_n$, then $\Sigma = \AA_n$, and then prove the theorem for all other surfaces.

If $\Sigma = \DD_n$, the finiteness of the volumes is shown in \cref{thm:vol-n-gon-1-int}. To show the inequality \cref{eq:finvolflipper-bound}, it suffices to use a slight improvement of the lower bound for $D$ in \cref{eq:bound-for-D} using the monotonicity of $\cosh x$.

If $\Sigma = \AA_n$, the finiteness of the volumes is shown in \cref{thm:vol-n-crown-1-int} and \cref{cor:vol-n-crown-no-neck}. The inequalities \cref{eq:finvolflipper-bound} follow from the inequality 
\begin{equation}
\label{eq:ineq-vol-A_1}
\frac{1}{\sqrt{2\cosh u_i +2\cosh c_i}} \leqslant \frac{1}{\sqrt{2\cosh u_i +2\cosh c'_i}}.
\end{equation}
If $\Sigma \neq \DD_n, \AA_n$, then the finiteness of the volumes follows from the neck recursion (\cref{prop:neck-recursion}), the inequality \cref{eq:ineq-vol-A_1} for $c_i'=0$ where $i=1,\dotsc,n$ and the finiteness of the Mirzakhani volumes of all crowned surfaces \cite[Theorem~1.5]{HT25}. The inequalities \cref{eq:finvolflipper-bound} follow from the inequality \cref{eq:ineq-vol-A_1}.
\end{proof}

\section{Further structure of the Mirzakhani volumes}
\label{sec:structure}

In this section, we
\begin{itemize}
    \item prove that the Mirzakhani volumes satisfy an analogue  (\cref{thm:vol-2pi}) of the string and dilaton equations (as paraphrased via \cite[Lemma~1-2]{do2009weil}),
    \item determine the structure of the coefficients in the Mirzakhani volume of general surface (\cref{thm:vol-str}), and
    \item show that the Mirzakhani volumes are periods (\cref{thm:vol-periods}).
\end{itemize}

\subsection{String and dilaton-type equations for Mirzakhani volumes}

In \cref{subsec:vol-n-gon-1-int} and \cref{subsec:vol-n-crown-1-int}, we derived single integral expressions for $V_{\DD_n}(\vec{c})$ and $V_{\AA_n}(\vec{c}\;|d)$. The strategies to obtain these results were similar in spirit: roughly speaking, each is a combination of a volume recursion and an integral transform. However, the exact choices made to realize these strategies differed: one used the disk recursion and the Mehler-Fock transform, while the other one used the crown recursion and the Fourier transform. Nevertheless, it is easy to see that the final formulas in \cref{thm:vol-n-gon-1-int} and \cref{thm:vol-n-crown-1-int} resemble each other. 

This lead us to observe an equation that relates $V_{\DD_n}(\vec{c})$ and $V_{\AA_n}(\vec{c}\;|d)$, which resembles an open strings version of \cite[Theorem~2]{do2009weil}, which Do and Norbury show is equivalent to a generalization of the dilaton equation proven by Witten in \cite{witten_conjecture}. Combining their result with the neck recursion shows that \textit{all} Mirzakhani volumes satisfy these dilaton-type equations. We further show that the volumes of the moduli spaces of fixed neck lengths satisfy a string-type equation (\cref{thm:vol-2pi}). 
\medskip

We start with the following lemma:

\begin{lemma}
\label{lem:analytic-continuation}
For $n\geqslant 1$, the integral 
\begin{equation}
V_{\AA_n} (c_1,\dotsc,c_n|d)  =  \pi^{n-1}\int_0^\infty \cos(\omega d) \cdot\prod_{i=1}^n \frac{P_{-1/2+\sqrt{-1}\omega}(\cosh c_i)}{\cosh \pi\omega} \,\mathrm{d}\omega
\end{equation}
extends holomorphically from $d\geqslant0$ to the strip $|\text{Im}\, d|<n\pi.$
\end{lemma}
\begin{proof}
Since for $c_i=0$, $P_{-1/2+\sqrt{-1}\omega}(\cosh c_i)=1$ \cite[Equation~7.3.13]{MR350075} and for $c_i>0$, $|P_{-1/2+\sqrt{-1}\omega}(\cosh c_i)| \leqslant P_{-1/2}(\cosh c_i)$  \cite[Equation~7.4.1]{MR350075}, for a fixed $\vec{c}$,
\begin{equation*}
    \left|\prod_{i=1}^n \frac{P_{-1/2+\sqrt{-1}\omega}(\cosh c_i)}{\cosh \pi\omega} \right| = O(e^{-n\pi\omega}).
\end{equation*}
Since $|\cos(\omega d)|\leqslant \cosh(\text{Im}\,d\cdot\omega) \leqslant\ e^{|\text{Im}\, d|\omega}$, the integral defines a holomorphic function when $|\text{Im}\,d| <n\pi$.
\end{proof}

Now we prove \cref{thm:vol-2pi}, assuming that
\begin{itemize}
\item if $\Sigma = \AA_n$, then $V_{\Sigma}(\vec{c}\;|d)$ is defined for $c_j\geqslant0, (j=1,\dotsc,n)$ and $|\text{Im}\, d|<n\pi$ (\cref{lem:analytic-continuation}), and \item if  $\Sigma\neq\DD_n,\AA_n$, the volume $V_{\Sigma}(\vec{b};\vec{c})$ (and $V_{\Sigma}(\vec{b};\vec{c}\;|\vec{d})$) is defined for 
\[b_i \in \CC, i=1,\dotsc,m \,\, \text{and}\,\, c_j\geqslant0, j=1,\dotsc,n \,\]
$\text{and}\, d_k>0, k=1,\dotsc,l$
via the neck recursion (\cref{prop:neck-recursion}).
\end{itemize}

\begin{theorem}
\label{thm:vol-2pi}
For $\Sigma = \Sigma_{g,m+1,\vec{a}}$ with $2g-2+m+l>0$, and $\vec{b}=(b_1,\dotsc,b_m)$,
\begin{equation}
\label{eq:vol-plug}
\begin{split}
V_{\Sigma}(\vec{b},2\pi\sqrt{-1};\vec{c}\;|\vec{d}) = &\sum_{i=1}^m \int_{0}^{b_i} b_i V_{\Sigma\cup p_{m+1}}(\vec{b};\vec{c}\;|\vec{d})\,\mathrm{d}b_i 
\\&+\sum_{j=1}^l d_j V_{\AA_{a_j}}(\vec{c}_j|d_j) \int_0^{d_j} \frac{V_{\Sigma\cup p_{m+1}}(\vec{b};\vec{c}\;|\vec{d})}{V_{\AA_{a_j}}(\vec{c}_j|d_j)} \mathrm{d}d_j
\end{split}
\end{equation}
For $\Sigma = \AA_n$ with $n\geqslant 3$,
\begin{equation}
\label{eq:vol-crown-der}
\frac{\partial V_{\Sigma}}{\partial d}\left(\vec{c}\;|2\pi\sqrt{-1}\right) = -\pi\sqrt{-1}V_{\Sigma\cup p_1}(\vec{c}).
\end{equation}
For $\Sigma = \Sigma_{g,m+1,\vec{a}}\neq\AA_n$ and $\vec{b}=(b_1,\dotsc,b_m)$,
\begin{align}
\label{eq:vol-der}
\frac{\partial V_\Sigma}{\partial b_{m+1}} \bigl(\vec{b},2\pi\sqrt{-1};\vec{c}\;|\vec{d}\;\bigr)
&=
2\pi \sqrt{-1} (2g-2+m+l)V_{\Sigma\cup p_{m+1}}\bigl(\vec{b};\vec{c}\;|\vec{d}\;\bigr)
\text{, and}\\
\frac{\partial V_\Sigma}{\partial b_{m+1}} \bigl(\vec{b},2\pi\sqrt{-1};\vec{c}\bigr)
&= 
2\pi \sqrt{-1} (2g-2+m+l)V_{\Sigma\cup p_{m+1}}\bigl(\vec{b};\vec{c}\bigr).
\end{align}
\end{theorem}
\begin{proof}
We first prove \cref{eq:vol-plug}. If $l=0$, i.e. $\Sigma = \Sigma_{g,m+1}$, then $\vec{c}=\vec{d}=\varnothing$, and the result is due to Do-Norbury \cite[Theorem~2, Equation~2]{do2009weil}. If $l\geqslant 1$, then since $2g-2+m+l>0$, the surface $\Sigma_{g,m+l}$ admits a hyperbolic metric and by the neck recursion (\cref{prop:neck-recursion}) together with \cite[Theorem~2]{do2009weil}:
\begin{equation*}
\begin{split}
&V_{\Sigma}(\vec{b},2\pi\sqrt{-1};\vec{c}\;|\vec{d}) 
\\\smash{\mathllap{=}}\,& V_{g,m+1+l}(\vec{b},2\pi\sqrt{-1},\vec{d})\cdot \prod_{j=1}^l d_j\cdot V_{\AA_{a_j}}(\vec{c}_j|d_j)
\\\smash{\mathllap{=}}\,& \left(\sum_{i=1}^m \int_0^{b_i} b_i V_{g,m+l}(\vec{b},\vec{d})\,\mathrm{d}b_i +\sum_{j=1}^l \int_{0}^{d_j}d_jV_{g,m+l}(\vec{b},\vec{d})\,\mathrm{d}d_j \right)\cdot \prod_{k=1}^l d_k\cdot V_{\AA_{a_k}}(\vec{c}_k|d_k)
\\\smash{\mathllap{=}}\,& \sum_{i=1}^m \int_0^{b_i} b_i V_{g,m+l}(\vec{b},\vec{d})\cdot \prod_{k=1}^l d_k\cdot V_{\AA_{a_k}}(\vec{c}_k|d_k)\,\mathrm{d}b_i
\\& \phantom{\sum}+ \sum_{j=1}^l d_j V_{\AA_{a_j}}(\vec{c}_j|d_j) \int_0^{d_j} d_jV_{g,m+l}(\vec{b},\vec{d})\cdot \prod_{k=1, k\neq j}^l d_k\cdot V_{\AA_{a_k}}(\vec{c}_k|d_k)\,\mathrm{d}d_j
\\\smash{\mathllap{=}}\,& \sum_{i=1}^m \int_{0}^{b_i} b_i V_{\Sigma\cup p_{m+1}}(\vec{b};\vec{c}\;|\vec{d})\,\mathrm{d}b_i 
+\sum_{j=1}^l d_j V_{\AA_{a_j}}(\vec{c}_j|d_j) \int_0^{d_j} \frac{V_{\Sigma\cup p_{m+1}}(\vec{b};\vec{c}\;|\vec{d})}{V_{\AA_{a_j}}(\vec{c}_j|d_j)} \mathrm{d}d_j.
\end{split}
\end{equation*}
Next, we prove \cref{eq:vol-crown-der} and \cref{eq:vol-der}. If $l=0$, i.e. $\Sigma = \Sigma_{g,m+1}$, then $\vec{c}=\vec{d}=\varnothing$, and the result is again due to Do-Norbury \cite[Theorem~2, Equation~3]{do2009weil}. If $l=1$ and $g=m=0$, then $\Sigma = \AA_n$ and $\Sigma\cup p_1 = \DD_n$. By \cref{thm:vol-n-crown-1-int} and \cref{thm:vol-n-gon-1-int}, 
\begin{equation}
\label{eq:vol-crown-der-calc}
\begin{split}
\frac{\partial V_{\Sigma}}{\partial d}\left(\vec{c}\;|2\pi\sqrt{-1}\right) & = \frac{\partial}{\partial d}\Bigg|_{d=2\pi\sqrt{-1}} \pi^{n-1}\int_0^\infty \cos(\omega d) \cdot\prod_{i=1}^n \frac{P_{-1/2+\sqrt{-1}\omega}(\cosh c_i)}{\cosh \pi\omega} \,\mathrm{d}\omega
\\& = \pi^{n-1}\int_0^\infty \frac{\partial\cos(\omega d)}{\partial d}\Bigg|_{d=2\pi\sqrt{-1}}\cdot\prod_{i=1}^n \frac{P_{-1/2+\sqrt{-1}\omega}(\cosh c_i)}{\cosh \pi\omega} \,\mathrm{d}\omega
\\& = -\pi^{n-1} \int_0^\infty \omega \cdot \sin\left(2\pi\sqrt{-1}\omega\right) \cdot\prod_{i=1}^n \frac{P_{-1/2+\sqrt{-1}\omega}(\cosh c_i)}{\cosh \pi\omega} \,\mathrm{d}\omega
\\& = -\pi\sqrt{-1}\cdot V_{\DD_n}(\vec{c}).
\end{split}
\end{equation}
Let us justify taking the derivative under the integral sign in \cref{eq:vol-crown-der-calc}. Similarly to the bounds in the proof of \cref{lem:analytic-continuation}, the differentiated integrand in \cref{eq:vol-crown-der-calc} satisfies
\begin{equation}
\label{eq:diff-under-int-sign}
\left|\omega \cdot \sin\left(\omega d\right) \cdot\prod_{i=1}^n \frac{P_{-1/2+\sqrt{-1}\omega}(\cosh c_i)}{\cosh \pi\omega} \right| = O\left(\omega e^{-(n\pi-|\text{Im}\,d|)\omega}\right). 
\end{equation}
Since $n\geqslant 3$ and the derivative is evaluated at $d=2\pi\sqrt{-1}$, the function in \cref{eq:diff-under-int-sign}, in a small neighborhood of $2\pi\sqrt{-1}$, is dominated by an integrable function on $[0,\infty)$. Differentiation under the integral sign is therefore justified by the dominated convergence theorem, and this proves \cref{eq:vol-crown-der}. 

In the remaining cases, we have $l\geqslant 1$ and $2g-2+m+l\geqslant0$. We first consider the case $2g-2+m+l=0$, i.e. when $(g,m,l)=(0,1,1)$ or $(g,m,l)=(0,0,2)$. If $(g,m,l)=(0,1,1)$, then by the neck recursion (\cref{prop:neck-recursion}),
\begin{equation*}
\begin{split}
\frac{\partial V_\Sigma}{\partial b_{2}} \bigl(b_1,2\pi\sqrt{-1};\vec{c}\;|d\bigr) &= \frac{\partial}{\partial b_{2} }\Bigg|_{b_{2}=2\pi\sqrt{-1}} \, V_{0,3}(b_1,b_2,d)\cdot d\cdot V_{\AA_{a_1}}(\vec{c}\;|d)
\\& = \frac{\partial}{\partial b_{2} }\Bigg|_{b_{2}=2\pi\sqrt{-1}}  \, 1\cdot d\cdot V_{\AA_{a_1}}(\vec{c}\;|d)
\\& = 0,
\end{split}
\end{equation*}
which agrees with \cref{eq:vol-der}.
The case of unfixed neck lengths is identical.

If $(g,m,l)=(0,0,2)$, then by the neck recursion (\cref{prop:neck-recursion}),
\begin{equation*}
\begin{split}
\frac{\partial V_\Sigma}{\partial b_1}\left(2\pi\sqrt{-1};\vec{c}\;|d_1,d_2\right) &= \frac{\partial}{\partial b_{1} }\Bigg|_{b_{1}=2\pi\sqrt{-1}} \, V_{0,3}(b_1,d_1,d_2)\cdot \prod_{i=1}^2 d_i\cdot V_{\AA_{a_i}}(\vec{c}_i|d_i)
\\& = \frac{\partial}{\partial b_{1} }\Bigg|_{b_{1}=2\pi\sqrt{-1}} \, 1 \cdot \prod_{i=1}^2 d_i\cdot V_{\AA_{a_i}}(\vec{c}_i|d_i)
\\& = 0.
\end{split}
\end{equation*}
which agrees with \cref{eq:vol-der}.
The case of unfixed neck lengths is identical.

Suppose that $l\geqslant1$ and $2g-2+m+l\geqslant 1$. Then the surface $\Sigma_{g,m+l}$ admits a hyperbolic metric, and by the neck recursion (\cref{prop:neck-recursion}) together with \cite[Theorem~2]{do2009weil}:
\begin{equation}
\begin{split}
&\frac{\partial V_\Sigma}{\partial b_{m+1}} \bigl(\vec{b},2\pi\sqrt{-1};\vec{c}\;|\vec{d}\;\bigr)
\\\smash{\mathllap{=}}\,& \frac{\partial}{\partial b_{m+1} }\Bigg|_{b_{m+1}=2\pi\sqrt{-1}} \, V_{g,m+1+l}\bigl(\vec{b},b_{m+1},\vec{d}\;\bigr)\cdot \prod_{i=1}^l d_i\cdot V_{\AA_{a_i}}(\vec{c}_i|d_i)
\\\smash{\mathllap{=}}\,& 2\pi\sqrt{-1}(2g-2+m+l)V_{g,m+l}(\vec{b},\vec{d})
\cdot \prod_{i=1}^l d_i\cdot V_{\AA_{a_i}}(\vec{c}_i|d_i)
\\\smash{\mathllap{=}}\,& 2\pi \sqrt{-1} (2g-2+m+l)V_{\Sigma\cup p_{m+1}}\bigl(\vec{b};\vec{c}\;|\vec{d}\;\bigr).
\end{split}  
\end{equation}
Finally, for unfixed neck lengths, we use the fact that $V_{\AA_{a_i}}(\vec{c_i}|d_i)\leqslant V_{\AA_{a_i}}(\vec{0}|d_i)$ (\cref{thm:finvolflipper}) and that a polynomial divided by the hyperbolic sine/cosine is integrable over $[0,\infty)$ to justify the differentiation under the integral sign.
\end{proof}

\begin{remark}
We are presently unaware of and interested in a cohomology-based interpretation and proof of \cref{eq:vol-crown-der}.  
\end{remark}
\medskip

\subsection{Mirzakhani volumes of moduli spaces of general surfaces}

For $\Sigma = \Sigma_{g,m,\vec{a}}$, let $A_k \subset \{1,\dotsc,n\}$ be the set of boundary puncture indices on the $k$-th boundary component for $k=1,\dotsc,l$. In particular, $|A_k|=a_k$. Define the function $\Phi_{a_k}(\omega,\vec{c}_k)$ as
\begin{equation}
\Phi_{a_k}(\omega,\vec{c}_k) = \prod_{i\in A_k}\frac{P_{-1/2+ix}(\cosh c_i)}{\cosh \pi x}
\end{equation}
We prove the following composite theorem:

\begin{theorem}
\label{thm:vol-str}
(i) For any $\Sigma \neq \DD_n,\AA_n$ or $\AA_{a_1,a_2}$, the volume of $\mathcal{M}_\Sigma(\vec{b};\vec{c})$ satisfies
\[
V_{\Sigma}(\vec{b};\vec{c}) \in\QQ\left[\pi,b_1^2,\dotsc,b_m^2, \mathcal{M}\{\Phi_{a_k}\}(-2r-1)\right], \, 1\leqslant k\leqslant l,0\leqslant 2r \leqslant \dim_{\RR}\mathcal{M}_\Sigma(\vec{b};\vec{c}), 
\]
where $\mathcal{M}\{\cdot\}(s)$ is the Mellin transform.

(ii) For any $\Sigma$, the Taylor germ of $V_{\Sigma}(\vec{b};\vec{c})$ in $c$-variables has coefficients in 
\[
\QQ[b_1^2,\dotsc, b_m^2,\log 2, \zeta(r),\beta(r)],\quad 1\leqslant r \leqslant \dim_\RR\mathcal{M}_\Sigma(\vec{b};\vec{c}).
\]
\end{theorem}

As preparation for proving this theorem, for $k\geqslant0,n\geqslant1$, define 
\begin{equation}
\label{eq:mellin-vol-n-crown-def}
\mu_{k,n}(\vec{c}) = \int_0^\infty d^{2k+1}\cdot V_{\AA_n}(\vec{c}\;|d)\,\mathrm{d}d.
\end{equation}
We show:

\begin{lemma}
\label{lem:mellin-vol-n-crown}
\begin{equation}
\label{eq:mellin-vol-flip-n-crown}
\mu_{k,n}(\vec{c}) = (-1)^{k+1}(2k+1)!\pi^{n-1}\mathcal{M}\{\Phi_n(\omega,\vec{c})\}(-2k-1),
\end{equation}
where the Mellin transform here is defined by analytic continuation.
\end{lemma}

\begin{proof}
By \cite[Equation~17.43.3]{MR2360010}
\begin{equation}
\mathcal{M}\{\cos x\}(s) = \Gamma(s)  \cos\left(\frac{\pi s}{2}\right),
\end{equation}
where $0<\text{Re}\, s<1$. Therefore,  $\mathcal{M}\{\cos ax\}(s) = a^{-s}\Gamma(s)  \cos\left(\frac{\pi s}{2}\right)$ for $a>0$. By analytic continuation, $\mathcal{M}\{\cos ax\}(s)$  extends to a meromorphic function on $\CC$, which is holomorphic except at zero and negative even integers, where it has simple poles. By \cref{thm:vol-n-crown-1-int},
\begin{equation}
\label{eq:mellin-vol-n-crown}
\begin{split}
\mu_{k,n}(\vec{c}) &= \int_0^\infty d^{2k+1}\cdot V_{\AA_n}(\vec{c}\;|d)\,\mathrm{d}d
\\&= \pi^{n-1}  \int_0^\infty d^{2k+1} \left(\int_0^\infty \cos (\omega d) \Phi_n(\omega,\vec{c})\,\mathrm{d}\omega\right)\mathrm{d}d
\\& = \pi^{n-1} \int_0^\infty \Phi_n(\omega,\vec{c}) \left(\int_0^\infty d^{2k+1}\cos(\omega d) \, \mathrm{d}d \right) \mathrm{d}\omega
\\& = \pi^{n-1} \int_0^\infty\Phi_n(\omega,\vec{c}) \cdot \mathcal{M}\{\cos(\omega d)\}(2k+2)\, \mathrm{d}\omega
\\& = \pi^{n-1}  \Gamma(2k+2)\cos(\pi(k+1)) \int_0^\infty \omega^{-2k-2}\Phi_n(\omega,\vec{c})\,\mathrm{d}\omega
\\& = (-1)^{k+1}(2k+1)!\pi^{n-1}  \mathcal{M}\{\Phi_n(\omega,\vec{c})\}(-2k-1).
\end{split}
\end{equation}
The inner integral in the third line of \cref{eq:mellin-vol-n-crown} does not converge, and the formal calculation justified by Abel regularization, see \cref{app:abel}.
\end{proof}

\begin{proof}[Proof of \cref{thm:vol-str}]

(i) By the neck recursion (\cref{prop:neck-recursion}):

\begin{equation}
\begin{split}
V_{\Sigma}(\vec{b};\vec{c}) = \int_{(0,\infty)^l} V_{\Sigma_{g,m+l}}(\vec{b},\vec{d}) \cdot \prod_{i=1}^l d_i\cdot V_{\AA_{a_i}}(\vec{c}_i|d_i).
\end{split}
\end{equation}
Write the polynomial $V_{\Sigma_{g,m+l}}(\vec{b},\vec{d})$ \cite{mirz_simp} as a finite sum
\begin{equation}
V_{\Sigma_{g,m+l}}(\vec{b},\vec{d}) = \sum_{\vec{k}\in\ZZ^l_+} \left( p_{\vec{k}}(\pi^2,\vec{b}) \cdot \prod_{i=1}^l d_i^{2k_i} \right) .
\end{equation}
Then by \cref{eq:mellin-vol-n-crown-def},
\begin{equation}
\label{eq:vol-str}
V_{\Sigma}(\vec{b};\vec{c}) = \sum_{\vec{k}\in\ZZ^l_+}  \left(  p_{\vec{k}}(\pi^2,\vec{b}) \cdot \prod_{i=1}^l \mu_{k_i,a_i}(\vec{c}_i)   \right).
\end{equation}
Combining it with \cref {lem:mellin-vol-n-crown} together with the bound (\cite{mirz_simp}, \cref{lem:gen-triang-cardinality}, \cref{prop:orthogeod-coord})
\[
2k_i \leqslant 6g-6+2(m+l)\leqslant 6g-6+2m+3l+n = \dim_{\RR}\mathcal{M}_\Sigma(\vec{b};\vec{c})
\]
furnishes part (i).
\medskip

(ii) a) We first consider the case $\Sigma \neq \DD_n,\AA_n$ or $\AA_{a_1,a_2}$. Using the hypergeometric representation of the Legendre function of the first kind \cite[Equation~7.3.6]{MR350075}
\[
P_{\nu}(z) = {}_2F_1\left(-\nu,\nu+1;1;\frac{1-z}{2}\right)
\]
with $\nu = -1/2+\sqrt{-1}\omega$ and  $z=\cosh c$, and expanding ${}_2F_1$ into series and simplifying the Pochhammer symbol, we have:
\begin{equation}
\begin{split}
\label{eq:conical-func-expansion}
    P_{-1/2+\sqrt{-1}\omega}(\cosh c) &= \sum_{n=0}^\infty  \frac{(-\nu)_n(\nu+1)_n}{(1)_n} \frac{\left(\frac{1-z}{2}\right)^n}{n!}
    \\& = \sum_{n=0}^\infty \frac{(-1)^n}{(n!)^2}\sinh^{2n}(c/2) \cdot \prod_{k=0}^{n-1}\left(\omega^2+\left(k+1/2\right)^2\right).
\end{split}    
\end{equation}
Let $s = \sinh^2(c/2)$ and $A_n(x) = \frac{1}{(n!)^2} \prod_{k=0}^{n-1}\left(x+\left(k+1/2\right)^2\right)$. Then 
\begin{equation}
\label{eq:conical-func-expan-short}
P_{-1/2+\sqrt{-1}\omega}(\cosh c) = \sum_{n=0}^\infty (-1)^n A_{n}(\omega^2)s^{n}.
\end{equation}
For a multi-index $\boldsymbol{\alpha} = (\alpha_1,\dotsc,\alpha_n)$, write
\begin{equation}
A_{\boldsymbol{\alpha}}(x) = \prod_{i=1}^n A_{\alpha_i}(x) = \sum_{j=0}^{|\boldsymbol{\alpha}|} a_{\boldsymbol{\alpha},j}x^{j}.
\end{equation}
Then for $\boldsymbol{s}=(s_1,\dotsc,s_n)$,
\begin{equation}
\label{eq:Phi-expansion}
\Phi_n(\omega,\vec{c}) = \sum_{\boldsymbol{\alpha}} (-1)^{|\boldsymbol{\alpha}|}A_{\boldsymbol{\alpha}}(\omega^2)\boldsymbol{s}^{\boldsymbol{\alpha}}\cdot\text{sech}^n(\pi\omega),
\end{equation}
and by \cref{lem:mellin-vol-n-crown},
\begin{equation}
\label{eq:mellin-vol-n-crown-series}
\mu_{k,n}(\vec{c}) = \sum_{\boldsymbol{\alpha}} (-1)^{k+|\boldsymbol{\alpha}|+1}(2k+1)!\left(\sum_{j=0}^{|\boldsymbol{\alpha}|} a_{\boldsymbol{\alpha},j}\pi^{n-1}\mathcal{M}\{\text{sech}^n{\pi\omega}\}(2j-2k-1) \right)\boldsymbol{s}^{\boldsymbol{\alpha}}.
\end{equation}
Now from \cref{lem:mellin-sechn-values}, \cref{eq:mellin-vol-n-crown-series}, \cref{eq:vol-str}, we conclude that for $\Sigma \neq \DD_n,\AA_n$ or $\AA_{a_1,a_2}$, the Taylor germ of $V_{\Sigma}(\vec{b};\vec{c})$ in $s_1,\dotsc,s_n$ (hence in $c_1,\dotsc,c_n$) has coefficients in 
\[
\QQ[\pi^2,b_1^2,\dotsc, b_m^2,\zeta(2r+1),\beta(2r)], \quad r\geqslant1.
\] 
Next, by \cref{lem:mellin-sechn-values} with $p=j-k-1$, the largest beta/zeta value that appears in \cref{eq:mellin-vol-n-crown-series} is at most
$\beta(2k + n + 1)/\zeta(2k + n + 1)$.
Together with \cite{mirz_simp}, \cref{lem:gen-triang-cardinality} and \cref{prop:orthogeod-coord}, it implies that
\begin{equation}
\begin{split}
2r+1 & \leqslant 2\max_{\vec{k}}\{\max_{i=1}^l\{k_i\}\}+\max_{i=1}^l\{a_i\} + 1 
\\&\leqslant 6g-6+2(m+l)+\max_{i=1}^l\{a_i\} + 1 
\\&\leqslant 6g-6+2m+3l+n
\\&= \dim_{\RR}\mathcal{M}_{\Sigma}(\vec{b};\vec{c}).
\end{split}
\end{equation}
\medskip

(ii) b) Let $\Sigma = \DD_n$. Using \cref{thm:vol-n-gon-1-int} and the expansion \cref{eq:Phi-expansion}, we obtain
\[
V_{\DD_n}(\vec{c}) = \pi^{n-2} \sum_{\boldsymbol{\alpha}} (-1)^{\boldsymbol{\alpha}} \left(\sum_{j=0}^{|\boldsymbol{\alpha}|}a_{\boldsymbol{\alpha},j}\int_0^\infty \frac{\xi^{2j+1}\sinh{2\pi\xi}}{\cosh^n\pi\xi}\mathrm{d}\xi\right) \boldsymbol{s}^{\boldsymbol{\alpha}}.
\]
Since
\[
\frac{\sinh 2\pi\xi}{\cosh^n\pi\xi} = - \frac{2}{(n-2)\pi} \frac{d}{d\xi}\text{sech}^{n-2}\pi\xi,
\]
we obtain via integration-by-parts that the coefficients of the Taylor germ of $V_{\DD_n}(\vec{c})$ in $s$-variables (hence in $c$-variables) are $\QQ$-linear combinations of
\[\pi^{n-3}\int_0^\infty \xi^{2j}\text{sech}^{n-2}\pi\xi\,\mathrm{d}\xi, \quad j\geqslant0,
\]
or, in other words, of $\pi^{n-3}\mathcal{M}\{\text{sech}^{n-2}\pi\xi\}(2j+1)$ for $j\geqslant 0$. Then by \cref{lem:mellin-sechn-values}, he Taylor germ of $V_{\DD_n}(\vec{c})$ has coefficients in $\QQ[\pi^2]$.
\medskip

(ii) c) Let $\Sigma = \AA_{a_1,a_2}$. By the neck recursion (\cref{prop:neck-recursion}),
\begin{equation}
V_{\AA_{a_1,a_2}}(\vec{c},\vec{c}\,') = \int_0^\infty d\cdot V_{\AA_{a_1}}(\vec{c}\;|d)\cdot V_{\AA_{a_2}}(\vec{c}\,'|d)\,\mathrm{d}d.
\end{equation}
From \cref{thm:vol-n-crown-1-int} observe that 
\begin{equation}
V_{\AA_n}(\vec{c}\;|d) = \sum_{\boldsymbol{\alpha}}(-1)^{|\boldsymbol{\alpha}|}A_{\boldsymbol{\alpha}}(-\partial^2_d) V_{\AA_n}(\vec{0}|d)\boldsymbol{s}^{\boldsymbol{\alpha}}.
\end{equation}
Therefore
\begin{align*}
&V_{\AA_{a_1,a_2}}(\vec{c},\vec{c}\,') 
\\\smash{\mathllap{=}}\,&  \sum_{\boldsymbol{\alpha},\boldsymbol{\beta}} (-1)^{|\boldsymbol{\alpha}|+|\boldsymbol{\beta}|}\left(\int_0^\infty d\cdot A_{\boldsymbol{\alpha}}(-\partial^2_d) V_{\AA_{a_1}}(\vec{0}|d)\cdot A_{\boldsymbol{\beta}}(-\partial^2_d) V_{\AA_{a_2}}(\vec{0}|d)\,\mathrm{d}d \right) \boldsymbol{s}^{\boldsymbol{\alpha}}(\boldsymbol{s}')^{\boldsymbol{\beta}}.
\end{align*}
The explicit expression for $V_{\AA_n}(\vec{0}|d)$ is derived in \cite[Theorem~1.3]{HT25}. Let $q=e^{-d}$, then one has
\[
V_{\AA_n}(\vec{0}|d) = \frac{Q_n(d)q^{1/2}}{1-(-1)^nq},
\]
where $Q_n(d) \in \QQ[\pi^2,d], \deg Q_n = n-1$. Applying $A_{\boldsymbol{\alpha}}(-\partial^2_d)$ produces a rational linear combination of terms
\[
d^\ell q^{1/2}\frac{R_\ell(q)}{(1+(-1)^nq)^N}
\]
with $R_\ell(q) \in \QQ[\pi^2,q]$. After multiplying the two factors, two ``$q^{1/2}$''s cancel the $q^{-1}$ from $\mathrm{d}d = -\frac{\mathrm{d}q}{q}$. Hence the coefficient becomes a finite rational linear combination of integrals
\[
\int_0^1 (-\log q)^r\frac{R(q)}{(1-(-1)^{a_1}q)^u(1-(-1)^{a_2}q)^v}\mathrm{d}q
.\]
Partial fractions and integration by parts reduce  these to 
\[
\int_0^1 \frac{(-\log q)^r}{1-q}\mathrm{d}q = r!\zeta(r+1)
\]
and 
\[
\int_0^1 \frac{(-\log q)^r}{1+q}\mathrm{d}q = r!\eta(r+1)
\]
with $\eta(1)=\log 2$. Thus the Taylor germ of $V_{\AA_{a_1,a_2}}(\vec{c},\vec{c}\,')$ has coefficients in 
\[\QQ[\pi^2,\log 2,\zeta(2r+1)],\] where $2r+1\leqslant a_1+a_2$.
\medskip

(ii) d) Finally, let $\Sigma = \AA_n$. By \cref{cor:vol-n-crown-no-neck}, it suffices to consider the Taylor expansion of $\frac{K(\tanh(c/2))}{\cosh(c/2)}$. Since $\frac{K(\tanh(c/2))}{\cosh(c/2)} = \frac{\pi}{2}P_{-1/2}(\cosh c)$  \cite[Equation~7.10.12]{MR350075} and  by \cref{eq:conical-func-expansion} the expansion of $P_{-1/2}(\cosh c)$ has rational coefficients, the Taylor germ of $V_{\AA_n}(\vec{c})$ has coefficients in $\QQ\pi^n$.
\end{proof}

\begin{remark}
The constant term of the Taylor germ in part (ii) of \cref{thm:vol-str} equals $V_\Sigma({\vec{b};\vec{0})}$ and this recovers Theorem~1.5 in \cite{HT25} (modulo positivity and homogeneity statements).
\end{remark}
\medskip

\subsection{Mirzakhani volumes are periods}

\begin{theorem}
\label{thm:vol-periods}
For any $\Sigma$, the volumes $V_\Sigma(\vec{b};\vec{c}\;|\vec{d})$ and $V_\Sigma(\vec{b};\vec{c})$ are periods in the sense of Kontsevich-Zagier when $b_i$, $\cosh c_j$ and $\cosh d_k$ are algebraic numbers. 
\end{theorem}

\begin{proof}
For $\Sigma = \DD_n$, the integral \cref{eq:vol-n-fon-big-int} is a period on the nose by letting $t_i = \cosh\ell_i$. Indeed, the domain of integration in $\RR^{n-3}$ is given by the inequalities $t_i\geqslant 1$, and the associated $(n-3)$-form is algebraic, since it is given by $f_{n-3}(t_1,\dotsc,t_{n-3})\mathrm{d}t_1\wedge\cdots\wedge \mathrm{d}t_{n-3}$ with $P_{n-3}f_{n-3}^2-1=0$, where $P_{n-3}$ is an explicit polynomial with algebraic coefficients.  

For $\Sigma = \AA_n$, make the following change of variables in \cref{eq:vol-n-crown-conv}: $q=e^{-d}, x_i=e^{\ell_i}, t_i=e^{c_i}$ for $i=1,\dotsc,n-1$, and let $X=x_1\cdots x_{n-1}$. Then we compute \[\frac{\mathrm{d}\ell_i}{\sqrt{2\cosh\ell_i+2\cosh c_i}} = \frac{\mathrm{d}x_i}{\sqrt{x_i}\sqrt{(x_i+t_i)(x_i+t_i^{-1})}}\] and \[\frac{1}{\sqrt{2\cosh(d-\sum_{i=1}^{n-1}\ell_i)+2\cosh c_n}} = \frac{\sqrt{qX}}{\sqrt{(1+t_nqX)(1+t_n^{-1}qX)}},\] and therefore 
\begin{equation}
\label{eq:vol-n-crown-period}
\begin{split}
V_{\AA_n}(\vec{c}\;|d) &= \int_{\RR^{n-1}}\frac{\mathrm{d}\ell_1\wedge\cdots\wedge\mathrm{d}\ell_{n-1}}{\prod_{i=1}^{n-1}\sqrt{2\cosh\ell_i+2\cosh c_i}\cdot\sqrt{2\cosh(d-\sum_{i=1}^{n-1}\ell_i)+2\cosh c_n}}
\\&=\sqrt{q}\int_{(0,\infty)^{n-1}} \frac{\mathrm{d}x_1\wedge\cdots\wedge\mathrm{d}x_{n-1}}{\prod_{i=1}^{n-1} \sqrt{(x_i+t_i)(x_i+t_i^{-1})}\cdot\sqrt{(1+t_nqX)(1+t_n^{-1}qX)}},
\end{split}
\end{equation}
which is a period, since $\cosh x \in \bar{\QQ}$ if and only if $e^x \in \bar{\QQ}$. The volume $V_{\AA_n}(\vec{c})$ is a period by \cref{cor:vol-n-crown-no-neck}.

For $\Sigma = \AA_{a_1,a_2}$, the volume $V_\Sigma(\vec{c},\vec{c}\,'|d)$ is a period by the neck recursion (\cref{prop:neck-recursion}) combined with the fact that periods form a ring, that the logarithms of algebraic numbers are periods and because $V_{\AA_n}(\vec{c}\;|d)$ is a period. For $V_\Sigma(\vec{c},\vec{c}\,')$, rewrite \cref{eq:vol-n-crown-period} as $V_{\AA_n}(\vec{c}\;|d) = \sqrt{q}\int_{\vec{x}>0}\frac{\mathrm{d}\vec{x}}{E(\vec{x},\vec{t},q)}$, then
\begin{equation*}
\begin{split}
V_{\AA_{a_1,a_2}}(\vec{c},\vec{c}\,') &= \int_1^\infty \log q\left(\int_{\vec{x}>0,\vec{x}'>0}\frac{\mathrm{d}\vec{x}\wedge\mathrm{d}\vec{x}'}{E(\vec{x},\vec{t},q)E(\vec{x}',\vec{t}',q)}\right)\mathrm{d}q 
\\& = \int_{\mathcal{D}} \frac{\mathrm{d}q\wedge\mathrm{d}u\wedge\mathrm{d}\vec{x}\wedge\mathrm{d}\vec{x}'}{uE(\vec{x},\vec{t},q)E(\vec{x}',\vec{t}',q)},
\end{split}
\end{equation*}
where $\mathcal{D} = \{q>1,\quad 1<u<q, \quad x_i>0, \quad x'_i>0\}$. Hence it is a period.

For $\Sigma \neq \DD_n, \AA_n, \AA_{a_1,a_2}$, the volume $V_\Sigma(\vec{b};\vec{c}\;|\vec{d})$ is a period by the same argument as for $V_\Sigma(\vec{c},\vec{c}\,'|d)$ and because $\pi$ is a period. For $V_\Sigma(\vec{b};\vec{c})$, by \cref{eq:vol-str} it suffices to show that $\mu_{k,n}(\vec{c})$ is a period. Writing $(\log q)^N = N! \int_{1<u_1<\dotsc<u_N<q}\frac{\mathrm{d}\vec{u}}{u_1\cdots u_N}$, for $N=2k+1$ we have
\begin{equation*}
\begin{split}
\mu_{k,n}(\vec{c}) &= \int_{1}^\infty(\log q)^N\left(\sqrt{q}\int_{\vec{x}>0}\frac{\mathrm{d}\vec{x}}{E(\vec{x},\vec{t},q)}\right) \frac{\mathrm{d}q}{q}
\\& = N! \int_{\mathcal{D}_N} \frac{\mathrm{d}q\wedge \mathrm{d}\vec{u}\wedge \mathrm{d}\vec{x}}{\sqrt{q}u_1\cdots u_NE(\vec{x},\vec{t},q) },
\end{split}
\end{equation*}
where $\mathcal{D}_N = \{q>1,\quad 1<u_1<\dotsc <u_N< q, \quad x_i>0\}$. Hence it is also a period.
\end{proof}

\section{Closed formulae}
\label{sec:explicit-volumes}

In this section, we:
\begin{itemize}
    \item compute the volume $V_{\AA_2}(c_1,c_2|d)$ of the moduli space of flippered $2$-crowns (\cref{thm:vol-flip-2-crown}),
    \item compute the volume $V_{\DD_4}(c_1,c_2,c_3,c_4)$ of the moduli space of flippered quadrilaterals  (\cref{thm:vol-flip-quadr}),
    \item compute the volume $V_{\DD_5}(c_1,c_2,0,0,0)$ of the moduli space of flippered pentagons (\cref{prop:vol-5-gon-2-flip}),
    \item compute the volume $V_{\AA_3}(c_1,0,0|d)$ of the moduli space of flippered $3$-crowns (\cref{prop:vol-3-crown-formula}), and 
    \item compute the volume $V_{\AA_{1,1}}(c,c)$ of the moduli space of flippered $(1,1)$-annuli (\cref{prop:vol-flip-(1-1)-ann}). 
    
\end{itemize}

\subsection{Volumes of the moduli spaces of flippered 2-crowns}

\begin{theorem}
\label{thm:vol-flip-2-crown}
The Mirzakhani volume of the moduli space of flippered 2-crowns with strap lengths $c_1,c_2$ and neck length $d$ is as follows.

Let $\xi_1=\cosh(c_1+c_2), \xi_2=\cosh(c_1-c_2), \xi_3=\cosh d$. Suppose $\xi_i>\xi_j>\xi_k$, where $\{i,j,k\}=\{1,2,3\}$.
Then as a function of $\xi_1,\xi_2,\xi_3$, the Mirzakhani volume $V_{\AA_2}$ is the following elliptic integral of the first kind:
\begin{equation}
\begin{split}
&V_{\AA_2}(\xi_1,\xi_2|\xi_3)  = \sqrt{\frac{2}{\xi_i-\xi_k}}F\left(\arcsin\sqrt{\frac{\xi_i-\xi_k}{\xi_i+1}},\sqrt{\frac{\xi_i-\xi_j}{\xi_i-\xi_k}}
\right). 
\end{split}
\end{equation}
Suppose $\xi_i>\xi_j=\xi_k$. Then
\begin{equation}
\begin{split}
{V}_{\AA_2}(\xi_1,\xi_2|\xi_3)=\sqrt{\frac{2}{\xi_i-\xi_j}}\text{arccoth}\sqrt{\frac{\xi_i+1}{\xi_i-\xi_j}}.
\end{split}
\end{equation}
Suppose $\xi_i=\xi_j>\xi_k$. Then
\begin{equation}
\begin{split}
{V}_{\AA_2}(\xi_1,\xi_2|\xi_3)=\sqrt{\frac{2}{\xi_i-\xi_k}}\text{arccot} \sqrt{\frac{1+\xi_k}{\xi_i-\xi_k}}.
\end{split}
\end{equation}
Suppose $\xi_i=\xi_j=\xi_k$. Then
\begin{equation}
\begin{split}
{V}_{\AA_2}(\xi_1,\xi_2|\xi_3)=\sqrt{\frac{2}{1+\xi_1}}.
\end{split}
\end{equation}
\end{theorem}
\begin{proof}
By \cref{eq:vol-2-crown-disk-rec},
\begin{equation*}
\begin{split}
&{V}_{\AA_2}(c_1,c_2|d) 
\\
\smash{\mathllap{=}}\,& \frac{1}{2}\int_0^\infty  \frac{1}{\sqrt{2\,\text{ch}\,\ell+2\,\text{ch}\, d}} \cdot \frac{2}{\sqrt{\text{ch}^2c_1+\text{ch}^2c_2+\text{ch}^2\ell+2\,\text{ch}\, c_1 \text{ch}\,c_2 \text{ch}\, \ell-1}}\,\mathrm{d}\cosh\ell.
\end{split}
\end{equation*} 
Let $t=\cosh \ell$. By \cref{eq:alg-id-cosh}, we can write
\begin{equation*}
\begin{split}
{V}_{\AA_2}(\xi_1,\xi_2|\xi_3)  = 
\frac{1}{\sqrt{2}}\int_1^\infty \frac{\mathrm{d}t }{\sqrt{(t+\xi_1)(t+\xi_2)(t+\xi_3)}}.
\end{split}
\end{equation*} 
Suppose $\xi_i>\xi_j>\xi_k$. By \cite[Equation~237.00]{EllipticIntegrals} (or \cite[Equation~3.131(7)]{MR2360010}) and the addition formula for elliptic integrals of the first kind \cite[Equation~116.01]{EllipticIntegrals} (for more details, see \cref{app:add-formula}), we have: 
\begin{equation}
\label{eq:vol-2-crown-gen-case}
\begin{split}
{V}_{\AA_2}(\xi_1,\xi_2|\xi_3) &=\frac{1}{\sqrt{2}}\int_1^\infty \frac{\mathrm{d}t }{\sqrt{(t+\xi_1)(t+\xi_2)(t+\xi_3)}} \\&= \sqrt{\frac{2}{\xi_i-\xi_k}}F\left(\arcsin\sqrt{\frac{t+\xi_k}{t+\xi_j}},\sqrt{\frac{\xi_i-\xi_j}{\xi_i-\xi_k}}
\right)\Bigg|_1^\infty
\\& = \sqrt{\frac{2}{\xi_i-\xi_k}}F\left(\arcsin\sqrt{\frac{\xi_i-\xi_k}{\xi_i+1}},\sqrt{\frac{\xi_i-\xi_j}{\xi_i-\xi_k}}
\right). 
\end{split}
\end{equation}
Next, suppose $\xi_i>\xi_j=\xi_k$. Then 
\begin{equation}
\begin{split}
{V}_{\AA_2}(\xi_1,\xi_2|\xi_3) &= \frac{1}{\sqrt{2}}\int_1^\infty \frac{\mathrm{d}t }{(t+\xi_j)\sqrt{(t+\xi_i)}} \\&= \frac{1}{\sqrt{2}}\int_{1+\xi_i}^\infty \frac{\mathrm{d}y}{(y+\xi_j-\xi_i)\sqrt{y}} \\& = -\sqrt{\frac{2}{\xi_i-\xi_j}} \text{arccoth} \sqrt{\frac{y}{\xi_i-\xi_j}}\Bigg|_{1+\xi_i}^\infty
\\& = \sqrt{\frac{2}{\xi_i-\xi_j}}\text{arccoth}\sqrt{\frac{\xi_i+1}{\xi_i-\xi_j}}.
\end{split}
\end{equation}

Suppose $\xi_i=\xi_j>\xi_k$. Then
\begin{equation}
\begin{split}
{V}_{\AA_2}(\xi_1,\xi_2|\xi_3) & = \frac{1}{\sqrt{2}}\int_1^\infty \frac{\mathrm{d}t }{(t+\xi_i)\sqrt{(t+\xi_k)}} \\&=\frac{1}{\sqrt{2}} \int_{1+\xi_k}^\infty \frac{\mathrm{d}y}{(y+\xi_i-\xi_k)\sqrt{y}} 
\\&= \sqrt{\frac{2}{\xi_i-\xi_k}} \arctan \sqrt{\frac{y}{\xi_i-\xi_k}}\Bigg|_{1+\xi_k}^\infty
\\& = \sqrt{\frac{2}{\xi_i-\xi_k}} \text{arccot} \sqrt{\frac{1+\xi_k}{\xi_i-\xi_k}}.
\end{split}
\end{equation}

Finally, suppose $\xi_i=\xi_j=\xi_k$. Then
\begin{equation}
\begin{split}
{V}_{\AA_2}(\xi_1,\xi_2|\xi_3) = \frac{1}{\sqrt{2}}\int_1^\infty \frac{\mathrm{d}t }{(t+\xi_1)^{3/2}} &= -\sqrt{\frac{2}{t+\xi_1}}\Bigg|_1^\infty 
= \sqrt{\frac{2}{1+\xi_1}}.
\end{split}
\end{equation}
\end{proof}

\subsection{Volumes of the moduli spaces of flippered quadrilaterals}

\begin{theorem}
\label{thm:vol-flip-quadr}
The Mirzakhani volume of the moduli space of flippered quadrilaterals with strap lengths $c_1,c_2,c_3,c_4$ (listed counterclockwise) is as follows.

Let $\xi_1=\cosh(c_1+c_2), \xi_2=\cosh(c_1-c_2), \xi_3=\cosh (c_3+c_4),\xi_4=\cosh(c_3-c_4)$. Suppose $\xi_i>\xi_j>\xi_k>\xi_\ell$, where $\{i,j,k,\ell\}=\{1,2,3,4\}$. Then as a function of $\xi_1,\xi_2,\xi_3,\xi_4$, the Mirzakhani volume ${V}_{\DD_4}$ is the following elliptic integral of the first kind:
\begin{equation}
\begin{split}
&{V}_{\DD_4}(\xi_1,\xi_2,\xi_3,\xi_4)\\
\smash{\mathllap{=}}\,& \frac{4}{\sqrt{(\xi_i-\xi_k)(\xi_j-\xi_\ell)}}
\\& \cdot
F\left(\arcsin\frac{\sqrt{(\xi_i-\xi_k)(\xi_j-\xi_\ell)}}{\sqrt{(1+\xi_i)(1+\xi_j)}+\sqrt{(1+\xi_k)(1+\xi_\ell)}},\sqrt{\frac{(\xi_i-\xi_\ell)(\xi_j-\xi_k)}{(\xi_i-\xi_k)(\xi_j-\xi_\ell)}}\right).
\end{split}
\end{equation}
Suppose $\xi_i>\xi_j>\xi_k=\xi_\ell$. Then 
\begin{equation}
\label{eq:vol-quadr-3-values}
\begin{split}
&V_{\DD_4}(\xi_1,\xi_2,\xi_3,\xi_4)= \frac{4}{\sqrt{(\xi_i-\xi_k)(\xi_j-\xi_k)}} \log\left(\frac{\sqrt{(\xi_i-\xi_k)(1+\xi_j)}+\sqrt{(\xi_j-\xi_k)(1+\xi_i)}}{(\sqrt{\xi_i-\xi_k}+\sqrt{\xi_j-\xi_k})\sqrt{1+\xi_k}}\right). 
\end{split}
\end{equation}
Suppose $\xi_i>\xi_j=\xi_k>\xi_\ell$. Then
\begin{equation}
\begin{split}
&V_{\DD_4}(\xi_1,\xi_2,\xi_3,\xi_4) =  \frac{4}{\sqrt{(\xi_i-\xi_j)(\xi_j-\xi_\ell)}} \left(\arctan\sqrt{\frac{\xi_i-\xi_j}{\xi_j-\xi_\ell}}-\arctan\sqrt{\frac{(\xi_i-\xi_j)(1+\xi_\ell)}{(\xi_j-\xi_\ell)(1+\xi_i)}} \right). 
\end{split}
\end{equation}
Suppose $\xi_i=\xi_j>\xi_k>\xi_\ell$. Then 
\begin{equation}
\begin{split}
V_{\DD_4}(\xi_1,\xi_2,\xi_3,\xi_4) =\frac{4}{\sqrt{(\xi_i-\xi_k)(\xi_i-\xi_\ell)}} \log\left(\frac{(\sqrt{\xi_i-\xi_k}+\sqrt{\xi_i-\xi_\ell})\sqrt{1+\xi_i}}{\sqrt{(\xi_i-\xi_k)(1+\xi_\ell)}+\sqrt{(\xi_i-\xi_\ell)(1+\xi_k)}}\right). 
\end{split}
\end{equation}
Suppose $\xi_i=\xi_j>\xi_k=\xi_\ell$. Then
\begin{equation}
\label{eq:vol-flip-quadr-2-values}
\begin{split}
{V}_{\DD_4}(\xi_1,\xi_2,\xi_3,\xi_4)= \frac{2}{\xi_i-\xi_k}\log \left(\frac{1+\xi_i}{1+\xi_k}\right).
\end{split}
\end{equation}
Suppose $\xi_i>\xi_j=\xi_k=\xi_\ell$. Then ${V}_{\DD_4}(\xi_1,\xi_2,\xi_3,\xi_4) = \frac{4}{\xi_i-\xi_j}\left(\sqrt{\frac{1+\xi_i}{1+\xi_j}}-1\right)$.

\noindent
Suppose $\xi_i=\xi_j=\xi_k>\xi_\ell$. Then ${V}_{\DD_4}(\xi_1,\xi_2,\xi_3,\xi_4) = \frac{4}{\xi_i-\xi_\ell}\left(1-\sqrt{\frac{1+\xi_\ell}{1+\xi_i}}\right)$.

\noindent
Suppose $\xi_i=\xi_j=\xi_k=\xi_\ell$. Then ${V}_{\DD_4}(\xi_1,\xi_2,\xi_3,\xi_4) = \frac{2}{1+\xi_1}$.
\end{theorem}

\begin{proof}
Consider the essential arc on $\DD_4$ that separates boundary punctures $q_1,q_2$ from $q_3,q_4$. Then by \cref{thm:vol-fli-triang}, the disk recursion (\cref{thm:disk-recursion}) and \cref{eq:alg-id-cosh}:
\begin{equation*}
\begin{split}
{V}_{\DD_4}(c_1,c_2,c_3,c_4) &= \frac{1}{2} \int_0^\infty {V}_{\DD_3}(c_1,c_2,\ell)\cdot {V}_{\DD_3}(c_3,c_4,\ell) \,\mathrm{d}\cosh \ell
\\& = \frac{1}{2} \int_0^\infty \frac{2}{\sqrt{{D}(c_1,c_2,\ell)}}\cdot \frac{2}{\sqrt{{D}(c_3,c_4,\ell)}} \,\mathrm{d}\cosh\ell
\\& = 2 \int_1^\infty \frac{\mathrm{d}t}{\sqrt{(t+\xi_1)(t+\xi_2)(t+\xi_3)(t+\xi_4)}},
\end{split}
\end{equation*}
where $t=\cosh\ell$.
By \cite[Equation~258.00]{EllipticIntegrals} (or \cite[Equation~3.147(8)]{MR2360010}) and the addition formula for elliptic integrals of the first kind \cite[Equation~116.01]{EllipticIntegrals} (for more details, see \cref{app:add-formula}), we have: 
\begin{equation}
\label{eq:vol-4-gon-gen-case}
\begin{split}
&{V}_{\DD_4}(\xi_1,\xi_2,\xi_3,\xi_4) \\\smash{\mathllap{=}}\,&  2\int_1^\infty \frac{\mathrm{d}t}{\sqrt{(t+\xi_1)(t+\xi_2)(t+\xi_3)(t+\xi_4)}}
\\\smash{\mathllap{=}}\,&  \frac{4}{\sqrt{(\xi_i-\xi_k)(\xi_j-\xi_\ell)}} F\left(\arcsin\sqrt{\frac{(t+\xi_\ell)(\xi_i-\xi_k)}{(t+\xi_k)(\xi_i-\xi_\ell)}},\sqrt{\frac{(\xi_i-\xi_\ell)(\xi_j-\xi_k)}{(\xi_i-\xi_k)(\xi_j-\xi_\ell)}}\right) \Bigg|_1^\infty
\\\smash{\mathllap{=}}\,& \frac{4}{\sqrt{(\xi_i-\xi_k)(\xi_j-\xi_\ell)}} \\& \cdot F\left(\arcsin\frac{\sqrt{(\xi_i-\xi_k)(\xi_j-\xi_\ell)}}{\sqrt{(1+\xi_i)(1+\xi_j)}+\sqrt{(1+\xi_k)(1+\xi_\ell)}},\sqrt{\frac{(\xi_i-\xi_\ell)(\xi_j-\xi_k)}{(\xi_i-\xi_k)(\xi_j-\xi_\ell)}}\right).
\end{split}
\end{equation}
Suppose $\xi_i>\xi_j>\xi_k=\xi_\ell$. Then
\begin{equation*}
\begin{split}
&{V}_{\DD_4}(\xi_1,\xi_2,\xi_3,\xi_4) \\\smash{\mathllap{=}}\,&   2\int_1^\infty \frac{\mathrm{d}t}{(t+\xi_k)\sqrt{(t+\xi_i)(t+\xi_j)}} \\\smash{\mathllap{=}}\,& 2\int_{1+\xi_k}^\infty \frac{\mathrm{d}y}{y\sqrt{(y+\xi_i-\xi_k)(y+\xi_j-\xi_k)}}
\\\smash{\mathllap{=}}\,&  \frac{4}{\sqrt{(\xi_i-\xi_k)(\xi_j-\xi_k)}} \log\left(\frac{\sqrt{y}}{\sqrt{(\xi_i-\xi_k)(y+\xi_j-\xi_k)}+\sqrt{(\xi_j-\xi_k)(y+\xi_i-\xi_k)}}\right)\Bigg|_{1+\xi_k}^\infty
\\\smash{\mathllap{=}}\,&  \frac{4}{\sqrt{(\xi_i-\xi_k)(\xi_j-\xi_k)}} \log\left(\frac{\sqrt{(\xi_i-\xi_k)(1+\xi_j)}+\sqrt{(\xi_j-\xi_k)(1+\xi_i)}}{(\sqrt{\xi_i-\xi_k}+\sqrt{\xi_j-\xi_k})\sqrt{1+\xi_k}}\right).
\end{split}
\end{equation*}
Suppose $\xi_i>\xi_j=\xi_k>\xi_\ell$. Then
\begin{equation*}
\begin{split}
&{V}_{\DD_4}(\xi_1,\xi_2,\xi_3,\xi_4) \\\smash{\mathllap{=}}\,&   2\int_1^\infty \frac{\mathrm{d}t}{(t+\xi_j)\sqrt{(t+\xi_i)(t+\xi_\ell)}} \\\smash{\mathllap{=}}\,& 2\int_{1+\xi_j}^\infty \frac{\mathrm{d}y}{y\sqrt{(y+\xi_i-\xi_j)(y+\xi_\ell-\xi_j)}}
\\\smash{\mathllap{=}}\,&  \frac{4}{\sqrt{(\xi_i-\xi_j)(\xi_j-\xi_\ell)}} \arctan \sqrt{\frac{(\xi_i-\xi_j)(y+\xi_\ell-\xi_j)}{(\xi_j-\xi_\ell)(y+\xi_i-\xi_j)}}\Bigg|_{1+\xi_j}^\infty
\\\smash{\mathllap{=}}\,&  \frac{4}{\sqrt{(\xi_i-\xi_j)(\xi_j-\xi_\ell)}} \left(\arctan\sqrt{\frac{\xi_i-\xi_j}{\xi_j-\xi_\ell}}-\arctan\sqrt{\frac{(\xi_i-\xi_j)(1+\xi_\ell)}{(\xi_j-\xi_\ell)(1+\xi_i)}} \right).
\end{split}
\end{equation*}
Suppose $\xi_i=\xi_j>\xi_k>\xi_\ell$. Then
\begin{equation*}
\begin{split}
&{V}_{\DD_4}(\xi_1,\xi_2,\xi_3,\xi_4) \\\smash{\mathllap{=}}\,&   2\int_1^\infty \frac{\mathrm{d}t}{(t+\xi_i)\sqrt{(t+\xi_k)(t+\xi_\ell)}} \\\smash{\mathllap{=}}\,& 2\int_{1+\xi_i}^\infty \frac{\mathrm{d}y}{y\sqrt{(y+\xi_k-\xi_i)(y+\xi_\ell-\xi_i)}}
\\\smash{\mathllap{=}}\,&  \frac{4}{\sqrt{(\xi_i-\xi_k)(\xi_i-\xi_\ell)}} \log\left(\frac{\sqrt{(\xi_i-\xi_k)(1+\xi_\ell)}+\sqrt{(\xi_i-\xi_\ell)(1+\xi_k)}}{\sqrt{y}}\right)\Bigg|_{1+\xi_i}^\infty
\\\smash{\mathllap{=}}\,&  \frac{4}{\sqrt{(\xi_i-\xi_k)(\xi_i-\xi_\ell)}} \log\left(\frac{(\sqrt{\xi_i-\xi_k}+\sqrt{\xi_i-\xi_\ell})\sqrt{1+\xi_i}}{\sqrt{(\xi_i-\xi_k)(1+\xi_\ell)}+\sqrt{(\xi_i-\xi_\ell)(1+\xi_k)}}\right).
\end{split}
\end{equation*}
Suppose $\xi_i=\xi_j>\xi_k=\xi_\ell$.
Then 
\begin{equation*}
\begin{split}
{V}_{\DD_4}(\xi_1,\xi_2,\xi_3,\xi_4) &= 2\int_1^\infty \frac{\mathrm{d}t}{(t+\xi_i)(t+\xi_k)} \\& = \frac{2\log\frac{t+\xi_k}{t+\xi_i}}{\xi_i-\xi_k}\Bigg|_1^\infty
\\& = \frac{2\log\frac{1+\xi_i}{1+\xi_k}}{\xi_i-\xi_k}.
\end{split}
\end{equation*}
Suppose $\xi_i>\xi_j=\xi_k=\xi_\ell$. Then
\begin{equation*}
\begin{split}
{V}_{\DD_4}(\xi_1,\xi_2,\xi_3,\xi_4) &= 2\int_1^\infty \frac{\mathrm{d}t}{(t+\xi_j)^{3/2}\sqrt{t+\xi_i}} \\&= -\frac{4}{\xi_i-\xi_j}\sqrt{\frac{t+\xi_i}{t+\xi_j}}\Bigg|_1^\infty
\\& = \frac{4}{\xi_i-\xi_j}\left(\sqrt{\frac{1+\xi_i}{1+\xi_j}}-1\right).
\end{split}
\end{equation*}
Suppose $\xi_i=\xi_j=\xi_k>\xi_\ell$. Then
\begin{equation*}
\begin{split}
{V}_{\DD_4}(\xi_1,\xi_2,\xi_3,\xi_4)& = 2\int_1^\infty \frac{\mathrm{d}t}{(t+\xi_i)^{3/2}\sqrt{t+\xi_\ell}} \\&= \frac{4}{\xi_i-\xi_\ell}\sqrt{\frac{t+\xi_\ell}{t+\xi_i}}\Bigg|_1^\infty
\\& = \frac{4}{\xi_i-\xi_\ell}\left(1-\sqrt{\frac{1+\xi_\ell}{1+\xi_i}}\right).
\end{split}
\end{equation*}
Finally, if $\xi_i=\xi_j=\xi_k=\xi_\ell$, then
\begin{equation*}
\begin{split}
{V}_{\DD_4}(\xi_1,\xi_2,\xi_3,\xi_4) = 2 \int_1^\infty \frac{\mathrm{d}x}{(x+\xi_1)^2} = -\frac{2}{x+\xi_1}\Big|_1^\infty
 = \frac{2}{1+\xi_1}.
\end{split}
\end{equation*}
\end{proof}

\subsection{Volumes of the moduli spaces of flippered pentagons with at most two straps of non-zero length}

We will need the following lemma.
\begin{lemma}
\label{lem:logx/x-int}
For any $a,c \in \RR$ and $b>0$:
\begin{equation}
\label{eq:logx/x-int}
\int \frac{\log(\frac{x+a}{b})}{x+c} \mathrm{d}x = \begin{cases} \text{Li}_2\left(\frac{c-a}{x+c}\right) + \frac{1}{2}\log^2(\frac{x+c}{b})+C, & \text{if } x+a>0, x+c > 0 \\ -\text{Li}_2\left(\frac{x+c}{c-a}\right) + \log(\frac{a-c}{b})\cdot\log(-x-c)+C', & \text{if } x+a>0, x+c<0 \end{cases}
\end{equation}
Furthermore, if $b=a-c$, letting $C-C'=\frac{\pi^2}{6}$ yields a continuous antiderivative on $x+a>0$.
\end{lemma}
\begin{proof}
The first part is a straightforward calculation using the derivative of the dilogarithm:
\[\text{Li}'_2(x) = -\frac{\log(1-x)}{x}.\]
The second part follows from the asymptotics of the dilogaritm as $x\to -\infty$:
\begin{equation}
\label{eq:dilog-asympt}
\text{Li}_2(x) = -\frac{\pi^2}{6}-\frac{\log^2(-x)}{2}+o(1),
\end{equation}
which itself is a corollary of the reflection properties of the dilogarithm (\cite{MR2290758}).
\end{proof}

\begin{proposition}
\label{prop:vol-5-gon-2-flip}
\begin{equation}
\begin{split}
&V_{\DD_5}(c_1,c_2,0,0,0)
=\frac{2\mathrm{Li}_2(1)+\log^2\frac{\cosh c_1+1}{\cosh c_2+1}+2\sum_{i=1,2}\mathrm{Li}_2\left(\frac{\cosh c_i-1}{\cosh c_i+1}\right)}{\cosh c_1+\cosh c_2}.
\end{split}
\end{equation}
\end{proposition}
\begin{proof}
Consider the essential arc on $\DD_5$ that separates boundary punctures $q_1$ and $q_2$, such that the component containing $q_1$ is homeomorphic to $\DD_4$.
Then by the disk recursion (\cref{thm:disk-recursion}) and \cref{eq:vol-flip-quadr-2-values} in \cref{thm:vol-flip-quadr}, 
\begin{equation*}
\begin{split}
{V}_{\DD_5}(c_1,c_2,0,0,0)&= \frac{1}{2} \int_0^\infty \frac{2\log \left(\frac{1+\text{ch}\, c_1}{1+\text{ch}\,\ell}\right)}{(\text{ch}\, c_1-\text{ch}\,\ell)}\cdot\frac{2}{\text{ch}\, c_2+\text{ch}\,\ell}\,\mathrm{d}\,\text{ch}\, \ell \\& = \frac{2}{\text{ch}\, c_1+\text{ch}\, c_2}\int_1^\infty \log \left(\frac{1+t}{1+\text{ch}\, c_1}\right)  \left(\frac{1}{t-\text{ch}\, c_1}-\frac{1}{t+\text{ch}\,c_2}\right) \mathrm{d}t,
\end{split}
\end{equation*}
where $t = \text{ch}\,\ell$. Let $s=\frac{1+t}{1+\text{ch}\, c_1}$. Then $t=(1+\text{ch}\, c_1)s-1$, and $\mathrm{d}t=(1+\text{ch}\, c_1)\mathrm{d}s$. Furthermore, by \cref{lem:logx/x-int},
\begin{equation}
\label{eq:vol-flip-5-gon-calc}
\begin{split}
&{V}_{\DD_5}(c_1,c_2,0,0,0)
\\\smash{\mathllap{=}}\,&  \frac{2}{\text{ch}\,  c_1+\text{ch}\,  c_2} \int_{\frac{2}{1+\text{ch}\,  c_1}}^\infty \log s \cdot \left(\frac{1}{s-1}-\frac{1}{s+\frac{\text{ch}\, c_2-1}{\text{ch}\, c_1+1}}\right) \mathrm{d}s
\\\smash{\mathllap{=}}\,&  \frac{-2\text{Li}_2(1-s)-2\text{Li}_2\left(\frac{1}{\frac{\text{ch}\, c_1+1}{\text{ch}\,  c_2-1}s+1}\right)-\log^2\left(s+\frac{\text{ch}\,  c_2-1}{\text{ch}\,  c_1+1}\right)\Bigg|_{\frac{2}{1+\text{ch}\,  c_1}}^\infty}{\text{ch}\,  c_1+\text{ch}\, c_2} .
\end{split}
\end{equation}
As $s\to \infty$, $\text{Li}_2(1-s) = -\frac{\pi^2}{6}-\frac{1}{2}\log^2(s-1)+o(1)$ by \cref{eq:dilog-asympt}, and 
\begin{equation}
\label{eq:lim-log^2}
\log^2(s+\alpha)-\log^2(s+\beta)\approx\frac{2(\alpha-\beta)\log(s)}{s+\beta}\to 0
\end{equation}
for every $\alpha,\beta\in\RR$,  hence the upper limit of integration in \cref{eq:vol-flip-5-gon-calc} is $\frac{\pi^2/3}{\cosh c_1+\cosh c_2}$. As $s\to \frac{2}{1+\cosh c_1}$, the lower limit of integration in \cref{eq:vol-flip-5-gon-calc} equals
\[
\frac{1}{\text{ch}\, c_1+\text{ch}\, c_2}\left(-2\text{Li}_2\left(\frac{\text{ch}\, c_1-1}{\text{ch}\, c_1+1}\right)-2\text{Li}_2\left(\frac{\text{ch}\, c_2-1}{\text{ch}\, c_2+1}\right)-\log^2\left(\frac{\text{ch}\, c_2+1}{\text{ch}\, c_1+1}\right)\right).
\]
After subtraction, we obtain that ${V}_{\DD_5}(c_1,c_2,0,0,0)$ equals 
\[
\frac{2\text{Li}_2(1)+\log^2\left(\frac{\text{ch}\, c_1+1}{\text{ch}\, c_2+1}\right)+2\text{Li}_2\left(\frac{\text{ch}\, c_1-1}{\text{ch}\, c_1+1}\right)+2\text{Li}_2\left(\frac{\text{ch}\,c_2-1}{\text{ch}\, c_2+1}\right)}{\text{ch}\,c_1+\text{ch}\,c_2}.
\]
\end{proof}
\begin{remark}
Observe that if $c_1=c_2=0$, then ${V}_{\DD_5} = \frac{\pi^2}{6}$, which agrees with \cite[Example~6.2.1]{chekhov2024}.
\end{remark}

\begin{proposition}
\label{prop:vol-3-crown-formula}
\begin{equation}
\begin{split}
&{V}_{\AA_{3}}(c,0,0|d)\cdot\sqrt{2(\text{ch} c+\text{ch} d)}
\\\smash{\mathllap{=}}\,&  2\text{Li}_2(1)+\text{Li}_2\left(\frac{\sqrt{\text{ch} d+\text{ch} c}-\sqrt{\text{ch} d+1}}{\sqrt{\text{ch} d+\text{ch} c}+\sqrt{\text{ch} d-1}}\right)+\text{Li}_2\left(\frac{\sqrt{\text{ch} d+\text{ch} c}-\sqrt{\text{ch} d-1}}{\sqrt{\text{ch} d+\text{ch} c}+\sqrt{\text{ch} d+1}}\right)
\\&\phantom{2\text{Li}_2(1)}+\text{Li}_2\left(\frac{\sqrt{\text{ch} d+\text{ch} c}-\sqrt{\text{ch} d+1}}{\sqrt{\text{ch} d+\text{ch} c}-\sqrt{\text{ch} d-1}}\right)+\text{Li}_2\left(\frac{\sqrt{\text{ch} d+\text{ch} c}+\sqrt{\text{ch} d-1}}{\sqrt{\text{ch} d+\text{ch} c}+\sqrt{\text{ch} d+1}}\right)
\\&\phantom{2\text{Li}_2(1)}+\frac{1}{2}\log^2\left(\frac{\sqrt{\text{ch} d+\text{ch} c}+\sqrt{\text{ch} d+1}}{\sqrt{\text{ch} d+\text{ch} c}+\sqrt{\text{ch} d-1}}\right)+\frac{1}{2}\log^2\left(\frac{\sqrt{\text{ch} d+\text{ch} c}+\sqrt{\text{ch} d+1}}{\sqrt{\text{ch} d+\text{ch} c}-\sqrt{\text{ch} d-1}}\right).
\end{split}
\end{equation}

\end{proposition}
\begin{proof}
Consider an essential arc on $\AA_3$ that decomposes it into $\AA_1$ and $\DD_4$. Then by the disk recursion (\cref{thm:disk-recursion}), \cref{eq:vol-flip-quadr-2-values} in \cref{thm:vol-flip-quadr} and \cref{thm:vol-1-crown},  the volume equals
\begin{equation*}
\begin{split}
{V}_{\AA_{3}}(c,0,0|d)
= \frac{1}{2} \int_0^\infty \frac{1}{\sqrt{2(\cosh \ell+\cosh d)}}\cdot \frac{2\log \left(\frac{1+\cosh \ell}{1+\cosh c}\right)}{\cosh \ell-\cosh c} \, \mathrm{d}
\cosh\ell
\end{split}
\end{equation*}
Let $s=\sqrt{\cosh\ell+\cosh d}$. Then $\cosh\ell=s^2-\cosh d$ and $\mathrm{d}\cosh\ell=2s\,\mathrm{d}s$. Then
\begin{equation}
\label{eq:vol-3-crown-calc}
\begin{split}
&{V}_{\AA_{3}}(c,0,0|d)  \cdot \sqrt{2(\text{ch}\,c+\text{ch}\,d)}
\\\smash{\mathllap{=}}\,& 2\sqrt{\text{ch}\,c+\text{ch}\,d}\int_{\sqrt{1+\text{ch}\,d}}^\infty \frac{\log\left(\frac{s^2-\text{ch}\,d+1}{1+\text{ch}\,c}\right)}{s^2-\text{ch}\,d -\text{ch}\,c} \mathrm{d}s
\\\smash{\mathllap{=}}\,&  \int_{\sqrt{1+\text{ch}d}}^\infty \left(\log\left(\frac{s+\sqrt{\text{ch}d-1}}{\sqrt{\text{ch}d+\text{ch}c}+\sqrt{\text{ch}d-1}}\right)+\log\left(\frac{s-\sqrt{\text{ch}d-1}}{\sqrt{\text{ch}d+\text{ch}c}-\sqrt{\text{ch}d-1}}\right)\right)
\\& \phantom{ \int_{\sqrt{1+\text{ch}d}}^\infty} \left(\frac{1}{s-\sqrt{\text{ch}d+\text{ch}c}}-\frac{1}{s+\sqrt{\text{ch}d+\text{ch}c}}\right) \mathrm{d}s
\\\smash{\mathllap{=}}\,& 
\int_{\sqrt{1+\text{ch}d}}^\infty \frac{\log\left(\frac{s+\sqrt{\text{ch}d-1}}{\sqrt{\text{ch}d+\text{ch}c}+\sqrt{\text{ch}d-1}}\right)}{s-\sqrt{\text{ch}d+\text{ch}c}} \mathrm{d}s - \int_{\sqrt{1+\text{ch}d}}^\infty \frac{\log\left(\frac{s+\sqrt{\text{ch}d-1}}{\sqrt{\text{ch}d+\text{ch}c}+\sqrt{\text{ch}d-1}}\right)}{s+\sqrt{\text{ch}d+\text{ch}c}}\mathrm{d}s \\& +\int_{\sqrt{1+\text{ch}d}}^\infty \frac{\log\left(\frac{s-\sqrt{\text{ch}d-1}}{\sqrt{\text{ch}d+\text{ch}c}-\sqrt{\text{ch}d-1}}\right)}{s-\sqrt{\text{ch}d+\text{ch}c}} \mathrm{d}s-\int_{\sqrt{1+\text{ch}d}}^\infty \frac{\log\left(\frac{s-\sqrt{\text{ch}d-1}}{\sqrt{\text{ch}d+\text{ch}c}-\sqrt{\text{ch}d-1}}\right)}{s+\sqrt{\text{ch}d+\text{ch}c}} \mathrm{d}s
\end{split}
\end{equation}
Applying \cref{lem:logx/x-int} with $C=0, C'=-\frac{\pi^2}{6}$, we find that the upper limit of integrations in \cref{eq:vol-3-crown-calc} is 
\begin{equation}
\begin{split}
\lim_{s\to\infty}\Biggl(&\frac{1}{2}\log^2\left(\frac{s-\sqrt{\text{ch}d+\text{ch}c}}{\sqrt{\text{ch}d+\text{ch}c}+\sqrt{\text{ch}d-1}}\right)-\frac{1}{2}\log^2\left(\frac{s+\sqrt{\text{ch}d+\text{ch}c}}{\sqrt{\text{ch}d+\text{ch}c}+\sqrt{\text{ch}d-1}}\right)
\\&+\frac{1}{2}\log^2\left(\frac{s-\sqrt{\text{ch}d+\text{ch}c}}{\sqrt{\text{ch}d+\text{ch}c}-\sqrt{\text{ch}d-1}}\right)-\frac{1}{2}\log^2\left(\frac{s+\sqrt{\text{ch}d+\text{ch}c}}{\sqrt{\text{ch}d+\text{ch}c}-\sqrt{\text{ch}d-1}}\right)\Biggr)
\end{split}
\end{equation}
which, similarly to \cref{eq:lim-log^2}, equals $0$.
Then by evaluating the antiderivative given by  \cref{lem:logx/x-int} at $\sqrt{1+\text{ch}\,d}$, we find that ${V}_{\AA_{3}}(c,0,0|d) \cdot \sqrt{2(\text{ch}\,c+\text{ch}\,d)}$ equals
\begin{equation}
\begin{split}
&\frac{\pi^2}{3}+\text{Li}_2\left(\frac{\sqrt{\text{ch}d+1}-\sqrt{\text{ch}d+\text{ch}c}}{-\sqrt{\text{ch}d+\text{ch}c}-\sqrt{\text{ch}d-1}}\right)+\text{Li}_2\left(\frac{\sqrt{\text{ch}d+\text{ch}c}-\sqrt{\text{ch}d-1}}{\sqrt{\text{ch}d+1}+\sqrt{\text{ch}d+\text{ch}c}}\right)
\\&\phantom{\frac{\pi^2}{3}}+\text{Li}_2\left(\frac{\sqrt{\text{ch}d+1}-\sqrt{\text{ch}d+\text{ch}c}}{-\sqrt{\text{ch}d+\text{ch}c}+\sqrt{\text{ch}d-1}}\right)+\text{Li}_2\left(\frac{\sqrt{\text{ch}d+\text{ch}c}+\sqrt{\text{ch}d-1}}{\sqrt{\text{ch}d+1}+\sqrt{\text{ch}d+\text{ch}c}}\right)
\\&\phantom{\frac{\pi^2}{3}}+\frac{1}{2}\log^2\left(\frac{\sqrt{\text{ch}d+1}+\sqrt{\text{ch}d+\text{ch}c}}{\sqrt{\text{ch}d+\text{ch}c}+\sqrt{\text{ch}d-1}}\right)+\frac{1}{2}\log^2\left(\frac{\sqrt{\text{ch}d+1}+\sqrt{\text{ch}d+\text{ch}c}}{\sqrt{\text{ch}d+\text{ch}c}-\sqrt{\text{ch}d-1}}\right).
\end{split}
\end{equation}
\end{proof}

\begin{remark}
If $c=0$, one can check that ${V}_{\AA_3}(0|d)=\frac{d^2+\pi^2}{4\cosh d/2}$, which agrees with \cite[Theorem~1.3]{HT25}.
\end{remark}

\subsection{Volumes of moduli spaces of flippered (1,1)-annuli with equal strap lengths}
\begin{proposition}
\label{prop:vol-flip-(1-1)-ann}
The volume of the moduli space $\mathcal{M}_{\AA_{1,1}}(c,c)$ of flippered $(1,1)$-annuli with straps of length $c,c$ is
\[
V_{\AA_{1,1}}(c,c)  =  \frac{\text{Li}_2(-e^{-c})-\text{Li}_2(-e^c)}{2\sinh c}.
\]
\end{proposition}

\begin{proof}
By \cref{thm:vol-1-crown} and the neck recursion (\cref{prop:neck-recursion}), 
\begin{equation*}
\begin{split}
V_{\AA_{1,1}}(c,c) &= \int_0^\infty \left(V_{\AA_1}(c|d)\right)^2 \cdot d\,\, \mathrm{d}d = \int_0^\infty \frac{ d}{2\cosh c + 2\cosh d} \, \mathrm{d}d \\& =  \frac{1}{2\sinh{c}}\left( d\cdot\log\left(\frac{1+e^{d+c}}{1+e^{d-c}}\right)+\text{Li}_2\left(-e^{d+c}\right)-\text{Li}_2\left(-e^{d-c}\right)\right)\Bigg|_0^\infty \\&= \frac{1}{2\sinh c}\left(\text{Li}_2(-e^{-c})-\text{Li}_2(-e^c)\right).
\end{split}
\end{equation*}
In the last line we have used that 
\begin{equation}
\label{eq:dilog-asymp}
\text{Li}_2(-e^t) = -\frac{\pi^2}{6}-\frac{t^2}{2}+o(1)
\end{equation}
as $t\to\infty$, see \cref{eq:dilog-asympt}.

\end{proof}

\newpage
\appendix

\section{An algebraic identity}
\label{app:identity}

\begin{lemma}
\label{lem:identity}
For every $z_1,z_2,z_3\geqslant1$, the following identity holds:
\begin{equation}
\label{eq:unexpected-identity}
\prod_{i=1}^{3}
\left(z_i z_{i+1}+z_{i+2}+\sqrt{D}\right) = \frac{1}{2} \prod_{i=1}^{3}(z_i+1)\cdot \left(z_1+z_2+z_3-1+\sqrt{D}\right)^2,
\end{equation}
where $D = z_1^2+z_2^2+z_3^2+2z_1z_2z_3-1,$ and the product index is taken $3$-cyclically.
\end{lemma}
\begin{proof}
Let
\[
e_1=z_1+z_2+z_3, \quad e_2=z_1z_2+z_2z_3+z_3z_1, \quad e_3=z_1z_2z_3
\]
be the elementary symmetric polynomials. Then
\[
D = e_1^2-2e_2+2e_3-1.
\]
For $1\leqslant i\leqslant 3$, define $A_i = z_iz_{i+1}+z_{i+2}$, where the index is taken $3$-cyclically.
Then we compute
\[
B_1 := A_1+A_2+A_3=e_2+e_1,
\]
\[
B_2 := A_1A_2+A_2A_3+A_3A_1= e_1e_3 + (e_1e_2-3e_3)+e_2,
\]
\[
B_3 := A_1A_2A_3= e_3^2 + (e_1^2-2e_2) e_3 + (e_2^2-2e_1e_3) +e_3.
\]
Now expand the left‑hand side of \cref{eq:unexpected-identity}: 
\begin{equation*}
\begin{split}
\prod_{i=1}^{3}
\left(A_i+\sqrt{D}\right)
&= B_3+B_2 \sqrt{D}+B_1\left(\sqrt{D}\right)^2+\left(\sqrt{D}\right)^3
\\&=(B_3+B_1D)+(B_2+D)\sqrt{D}. 
\end{split}
\end{equation*}
Next, we factorize
\begin{equation*}
\begin{split}
B_3+B_1D & = e_3^2 + (e_1^2-2e_2) e_3 + (e_2^2-2e_1e_3) +e_3+(e_2+e_1)(e_1^2-2e_2+2e_3-1)
\\& = e_3^2+(e_1^2+1)e_3-(e_2^2+(1+2e_1-e_1^2)e_2+(1+e_1)(e_1-e_1^2)),
\end{split}
\end{equation*}
and observing that the bracket $(e_2^2+(1+2e_1-e_1^2)e_2+(1+e_1)(e_1-e_1^2))$ factorizes as $(e_2+1+e_1)(e_2+e_1-e_1^2)$, we obtain
\begin{equation*}
\begin{split}
B_3+B_1D & = e_3^2+(e_1^2+1)e_3+(e_1^2-e_1-e_2)(e_1+e_2+1)
\\& = \left(e_3+e_1^2-e_1-e_2\right)(e_3+e_1+e_2+1),
\end{split}
\end{equation*}
and we factorize 
\begin{equation*}
\begin{split}
B_2+D & = (e_1e_3+e_1e_2-3e_3+e_2) + (e_1^2-2e_2+2e_3-1)
\\& = (e_1^2-1)+(e_1-1)e_2+(e_1-1)e_3
\\& = (e_1-1)(e_1+e_2+e_3+1).
\end{split}
\end{equation*}
Therefore, the left‑hand side of  \cref{eq:unexpected-identity} equals 
\begin{equation*}
\begin{split}
&\prod_{i=1}^{3}
\left(A_i+\sqrt{D}\right)
\\\smash{\mathllap{=}}\,&(B_3+B_1D)+(B_2+D)\sqrt{D} 
\\\smash{\mathllap{=}}\,&  (e_1+e_2+e_3+1)\left((e_1-1)e_1-e_2+e_3\right)+(e_1+e_2+e_3+1)(e_1-1)\sqrt{D}
\\\smash{\mathllap{=}}\,&  \frac{1}{2}(e_1+e_2+e_3+1)\left(e_1-1+\sqrt{D}\right)^2,
\end{split}
\end{equation*}
which matches the right‑hand side of  \cref{eq:unexpected-identity}.
\end{proof}

\section{Explicit calculations for $\mathbb{D}_3$ and $\mathbb{A}_1$}
\label{app:D3-A1-action}

\subsection{Chekhov action for flippered triangles}
\label{subsec:action-triang}

We now compute the Chekhov action for flippered triangles. The moduli spaces $\mathcal{M}_{\mathbb{D}_3}(\vec{c})$ of such triangles are singletons, but we will show in \cref{subsec:vol-form-flip-triang} that neither the action nor the volume form is constant with respect to $\vec{c}=(c_1,c_2,c_3)$, and the final volume expressions are non-constant (see \cref{subsec:vol-flip-triang}).\medskip

\begin{proposition}
\label{prop:action-triang}
The Chekhov action for the flippered triangle with strap lengths $\vec{c}=(c_1, c_2, c_3)$, is equal to
\begin{equation}
\label{eq: triangle with 3 flippers}
e^{-S}
= 
\frac{
4
}
{\cosh c_1+\cosh c_2+\cosh c_3-1+\sqrt{D}
},
\end{equation}
where 
\[
D = \cosh^2 c_1 + \cosh^2 c_2 + \cosh^2 c_3 +2\cosh c_1 \cosh c_2 \cosh c_3 -1.
\]
\end{proposition}

\begin{proof}
By \cref{defn: action for flippered surface}, we have 
\begin{equation}
\label{eq:flip-triang-action}
S = \frac{1}{2}\left(\ell(\mathbb{I}_1)+\ell(\mathbb{I}_2)+\ell(\mathbb{I}_3)\right)+\sum_{i=1}^3 \log \cosh (c_i/2).
\end{equation}

\begin{figure}[H]
    \centering
    \begin{tikzpicture}[scale=1.2]
\usetikzlibrary{decorations.pathreplacing}

  \draw[very thick]
    (0,2) ++(-30:1) arc[start angle=-30, end angle=-150, radius=1];

  \draw[very thick]
    (-2,0) ++(-80:1.2) arc[start angle=-80, end angle=70, radius=1.2];

  \draw[very thick]
    (2,0) ++(120:0.9) arc[start angle=120, end angle=250, radius=0.9];


\draw[thick, blue] (-1.1,0.83) to [out=40,in=-130] (-0.65,1.25);


\draw[thick, blue] (1.37,0.66) to [out=140,in=-50] (0.74,1.32);


\draw[thick, blue] (-0.8,0) to [out=0,in=180] (1.1,0);

\draw[thick, blue] (-1.79,-1.19) to [out=10,in=-100] (-0.1,1);

\draw[thick, blue] (1.7,-0.85) to [out=160,in=-75] (0.25,1.05);


\draw [decorate,decoration={brace,amplitude=4pt}]
    (-0.1,1) -- (0.25,1.05) node [midway,yshift=10pt] {$\mathbb{I}_1$};

    (-0.1,1) -- (-0.65,1.25);

    (0.74,1.32) -- (0.25,1.05);

\draw (-1,1.15) node {$c_1$};

\draw (1.2,1.05) node {$c_2$};

\draw (0.1,0.13) node {$c_3$};



\end{tikzpicture}
    \caption{Essential interval in a flippered triangle.}
    \label{fig:flip-triang-ess-int}
\end{figure}
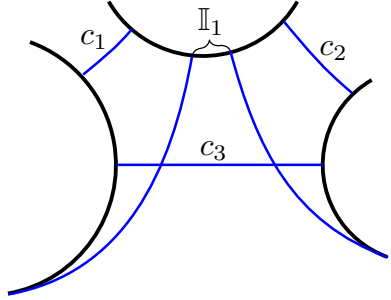

From \cref{fig:flip-triang-ess-int}, we compute $\ell(\mathbb{I}_1)$ using a right-angled hexagon formula \cite[Equation 2.6.10]{MR4554426} and the tri-rectangle relation \cite[Theorem 2.3.1]{Buser}:
\begin{equation}
\label{eq:ess-int-flip-triang}
\ell(\mathbb{I}_1) = \text{arccosh}\left(\frac{\cosh c_3+\cosh c_1 \cosh c_2}{\sinh c_1 \sinh c_2}\right) - \text{arcsinh}\left(\frac{1}{\sinh c_1}\right)-\text{arcsinh}\left(\frac{1}{\sinh c_2}\right).
\end{equation} 
Rewriting the summands in \cref{eq:ess-int-flip-triang} as logarithms, we obtain:
\begin{equation}
\label{eq:ess-int-flip-triang-log}
\begin{split}
\ell(\mathbb{I}_1) &= \log \frac{\cosh c_3+\cosh c_1 \cosh c_2+\sqrt{D}}{\sinh c_1\sinh c_2} - \log\frac{\cosh c_1+1}{\sinh c_1}-\log\frac{\cosh c_2+1}{\sinh c_2} 
\\&= \log \frac{\cosh c_3+\cosh c_1 \cosh c_2+\sqrt{D}}{(\cosh c_1+1)(\cosh c_2+1)}. 
\end{split}
\end{equation}
Then, by \cref{eq:flip-triang-action} and \cref{eq:ess-int-flip-triang-log},
\begin{equation}
e^{-S}
= 
\frac{
(\cosh c_1 +1)(\cosh c_2 +1)(\cosh c_3 +1)
}
{\prod_{i=1}^{3}
(\cosh c_i \cosh c_{i+1}+\cosh c_{i+2}+\sqrt{D})^{\frac{1}{2}}
}\cdot\frac{1}{\prod_{i=1}^3\cosh(c_i/2)},
\end{equation}
where the product index is taken $3$-cyclically. To obtain the final expression  of the action in \cref{eq: triangle with 3 flippers}, we apply the algebraic identity of \cref{lem:identity}:
\begin{equation}
\begin{split}
e^{-S}
& = 
\frac{
(\cosh c_1 +1)(\cosh c_2 +1)(\cosh c_3 +1)
}
{\prod_{i=1}^{3}
(\cosh c_i \cosh c_{i+1}+\cosh c_{i+2}+\sqrt{D})^{\frac{1}{2}}
} \cdot\frac{1}{\prod_{i=1}^3\cosh(c_i/2)},
\\& 
= \frac{
\sqrt{2(\cosh c_1+1)(\cosh c_2+1)(\cosh c_3+1)}
}
{\cosh c_1+\cosh c_2+\cosh c_3-1+\sqrt{D}
} \cdot\frac{1}{\prod_{i=1}^3\cosh(c_i/2)},
\\& 
= \frac{
4
}
{\cosh c_1+\cosh c_2+\cosh c_3-1+\sqrt{D}
}.
\end{split}
\end{equation}
\end{proof}
The following corollary is immediate.
\begin{corollary}
\label{cor:action-flip-triang}
The Chekhov's action for the triangle with three flippers, whose straps have lengths $c_1$, $c_2$ and $c_3=0$, satisfies  
\begin{equation}
e^{-S} = \frac{2
}{
\cosh c_1+\cosh c_2
}.    
\end{equation}
The Chekhov's action for the triangle with three flippers, whose straps have lengths $c_1$ and $c_2=c_3=0$, satisfies
\begin{equation}
e^{-S} = \frac{1}{\cosh^2 \!\big(\frac{c_1}{2}\big)}.
\end{equation}
\end{corollary}
\medskip

\subsection{Chekhov action for flippered 1-crowns}
\label{sec:1crown}

\begin{proposition}
\label{prop:1-crown-action}
The Chekhov action for the 1-crown with neck length $d$ and the flipper with strap length $c$ equals
\begin{equation}
e^{-S}
=
\frac{1}{\sqrt{\frac{\cosh c +\cosh d}{2}}+\cosh d/2} .
\end{equation}
\end{proposition}

\begin{proof}
Suppose that the strap $\sigma_1$ of a flippered $1$-crown $X$ has positive length, i.e. that $c>0$. Consider the flipper-clipped surface $\text{Clip}(X)$. Triangulating it by two ideal arcs emanating from the interior puncture, the associated shearing coordinates $\{s_1,s_2\}$ satisfy: $s_1+s_2=d$, and by \cref{eq:1-flip-crown-system},
\begin{equation}
\label{eq:eq:1-flip-crown-system}
(1+e^{s_1})(1+e^{-s_2}) = \cosh^2(c/2).
\end{equation}
By \cite[Equation~5.1]{chekhov2024}, the Chekhov action equals
\[
e^{-S(\text{Clip}(X))}  = \frac{1}{4\cosh(s_1/2)\cosh(s_2/2)},
\]
and we express it in $c,d$. Let $u=e^{(s_1-s_2)/2}$, then \cref{eq:eq:1-flip-crown-system} becomes
\[
(1+e^{d/2}u)(1+e^{-d/2}u) = u^2+2\cosh(d/2)u+1=\cosh^2(c/2),
\]
whose positive root is $u=\sqrt{\frac{\cosh c+\cosh d}{2}}-\cosh d/2$. Meanwhile, 
\[
(1+e^{s_1})(1+e^{-s_2}) = 4 u \cosh(s_1/2)\cosh(s_2/2), 
\]
hence 
\[
e^{-S(\text{Clip}(X))}  = \frac{u}{\cosh^2(c/2)} = \frac{\tanh^2(c/2)}{\sqrt{\frac{\cosh c+\cosh d}{2}}+\cosh d/2}.
\]
Finally, by \cref{cor:action-clip}:
\[
e^{-S(X)} = e^{-S(\text{Clip}(X))} \cdot \coth^2(c/2) = \frac{1}{\sqrt{\frac{\cosh c +\cosh d}{2}}+\cosh d/2}.
\]
The case $c=0$ is obtained by letting $c\to0$ and making use of the continuity of the action.
\end{proof}

\section{Generalities about the volume forms}
\label{app:codimvolume}

Following \cite{MR350769} with minor modifications, we state a few well-known results on the volume forms on orientable smooth manifolds and their submanifolds. In \cref{subsec:flippervolform}, we apply them in our particular setting, to define the volume form $\Omega^{\mathrm{WP}}_{\Sigma}(\vec{b};\vec{c})$ on the Teichm\"uller space of flippered hyperbolic surfaces with fixed holonomy $\mathcal{T}(\vec{b};\vec{c})$. 

\subsection{Submersions and volume forms}
\label{subsec:submersions}
Let $M$ and $N$ be smooth (orientable) manifolds of dimensions $m$ and $n$, respectively. Let $f:M\to N$ be a surjective submersion. Fix $x\in M$, and set $y = f(x)$. Then the fiber $f^{-1}(y)$ is a smooth submanifold of $M$ of dimension $m-n$. Let $j: f^{-1}(y) \to M$ denote the canonical injection. Consider the covectors $\omega_x \in \Lambda^m(T^*_xM), \mu_y\in\Lambda^n (T^*_yN)$, so that $\mu_y$ does not vanish. Then 
\begin{lemma}[{\cite[16.21.7]{MR350769}}] 
There is a unique $(m-n)$-covector 
$$\sigma_x \in \Lambda^{m-n}(T^*_x(f^{-1}(y)))$$ such that 
\begin{equation*}
\omega_x = (f^*\mu_y) \wedge \sigma_x',    
\end{equation*}
where $\sigma_x' \in \Lambda^{m-n}(T^*_xM)$ is any $(m-n)$-covector for which $\sigma_x = j^*\sigma_x'$.    
\end{lemma}
We denote by $\omega_x/\mu_y$ the $(m-n)$-covector $\sigma_x$, whose existence has just been established. Moreover, the above proof shows that, for fixed $y$ and $\mu_y\neq0$, we have:
\begin{lemma}[{\cite[16.21.8]{MR350769}}]
\label{lem:vol-form-division}
If $x\to\omega_x$ is a smooth differential $m$-form on $M$, then $x\to\omega_x/\mu_y$ is a smooth differential $(m-n)$-form on $f^{-1}(y)$.   
\end{lemma}
We denote by $\omega/\mu_y$ the resulting $(m-n)$-form on $f^{-1}(y)$. It follows that if $M$ is orientable, then for each $y \in N$, the fiber $f^{-1}(y)$ is orientable \cite[(16.21.9)]{MR350769}.

\subsection{Integration along fibers}
In the setting of \cref{subsec:submersions}, there is an analogue of Fubini's theorem as follows. Suppose that $\omega$ is an integrable differential $m$-form $\geqslant 0$ on $M$, and $\mu$ is a (locally) integrable differential $n$-form on $N$ such that $\mu>0$ almost everywhere on $N$. Then
\begin{lemma}[{\cite[16.24.8]{MR350769}}]
\label{lem:Fubini}
For almost all $y\in N$ the $(m-n)$-form $\omega/\mu_y$ is integrable over $f^{-1}(y)$ (endowed with the orientation induced by $f$ from the orientations of $M$ and $N$). The form $y \mapsto  (\int_{f^{-1}(y)} \omega/\mu_y) \mu_y$ is integrable over $N$, and 
\begin{equation}
    \int_M \omega = \int_N \left(\int_{f^{-1}(y)} \omega/\mu_y \right) \mu_y.
\end{equation}    
\end{lemma}

\section{Shearing to Fenchel--Nielsen for $(1,1)$-annuli}
\label{app:shear-FN}

\begin{lemma}\label{lem:conversion}
    Consider $\Sigma=\Sigma_{0,0,(1,1)}$, i.e.: $\Sigma$ is an annulus with a single puncture on each of its two boundaries. Let $\triangle$ denote an ideal triangulation of $\Sigma$, and let $s_0,s_1$ denote the shearing parameters for $\triangle=\{\mu_0,\mu_1\}$ (see \cref{fig:D_fig1}). Further let $\gamma$ denote the unique simple closed curve (up to isotopy) on $\Sigma$, and let $(\ell,\tau)$ denote Fenchel--Nielsen coordinates with respect to $\gamma$, noting that $\tau$ is non-canonical, but well-defined up to a choice of addition by some constant. Then,
    \[
    \mathrm{d}s_0\wedge\mathrm{d}s_1=\pm\mathrm{d}\ell\wedge\mathrm{d}\tau.
    \]
\end{lemma}

\begin{proof}
Our rough strategy is as follows: we express lambda-length coordinates for $\calT_\Sigma$ (taken with respect to $\triangle$) in terms of some choice of Fenchel--Nielsen coordinates and use this as an intermediary coordinate system to express $s_0,s_1$ as functions of $\ell$ and $\tau$. We then explicitly calculate $\mathrm{d}s_0\wedge\mathrm{d}s_1$ and show that it simplifies to $\mathrm{d}\ell\wedge\mathrm{d}\tau$.

\begin{center}
\begin{figure}[ht!]
    \includegraphics[scale=0.8]{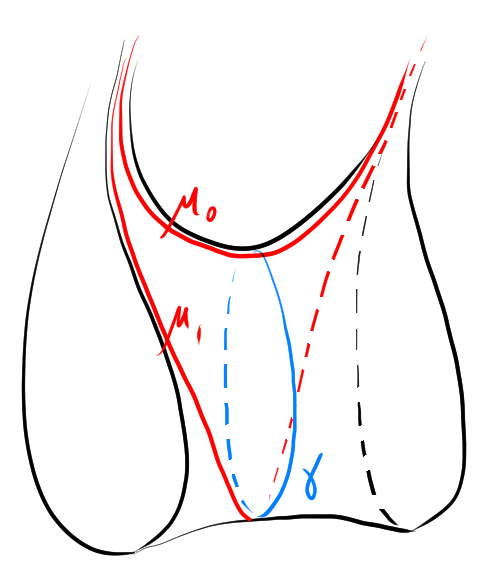}
\caption{An ideal triangulation {\color{red}$\triangle=\{\mu_0,\mu_1\}$} and a simple closed geodesic {\color{blue}$\gamma$} on a $(1,1)$-annulus.}\label{fig:D_fig1}
\end{figure}
\end{center}

Let $\lambda_0$ and $\lambda_1$ denote the lambda-lengths for $\triangle$ with both horocycles set to $1$. Concretely, this means that given a marked hyperbolic structure $X\in\calT_\Sigma$, $\lambda_i$ is defined as $\exp(u_i/2)$, where $u_i$ is the length of the ideal geodesic on $X$ representing $\mu_i$ truncated at the length $1$ horocycles for the two tines. Further let $\lambda_a$ denote the lambda-length of the two boundary arches, noting that they are both equal to $2\cosh(\ell/2)$ (see, e.g.: \cite[Equation~(3.7)]{huang_thesis}). Consider the unique $X_0\in\calT_\Sigma$ where $\mu_0$ is orthogonal to $\gamma$. We define $\tau$ so that $\tau(X_0)=0$.

\begin{center}
\begin{figure}[ht!]
    \includegraphics[scale=0.8]{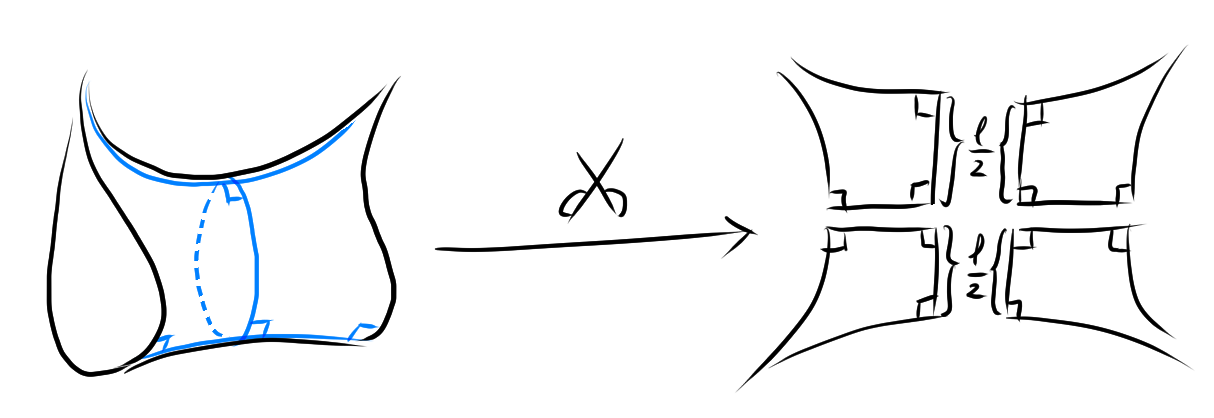}
\caption{Cutting $X_0$ into four isometric trirectangles.}\label{fig:D_fig2}
\end{figure}
\end{center}

Cutting $X_0$ along $\gamma$ and $\mu_0$ and the (unique) shortest geodesic arc joining the two boundary arches of $X_0$ yields four isometric trirectangles (\cref{fig:D_fig2}) with finite sides of lengths $\frac{\ell}{2}$ and $\mathrm{arcsinh}(1/\sinh\frac{\ell}{2})$ (using \cite[2.3.1.(i)]{Buser}), meaning that 
\[
\lambda_0(X_0)
=
2\cosh(\mathrm{arcsinh}(1/\sinh\tfrac{\ell}{2}))
=
2\sqrt{\frac{1+\sinh^2\frac{\ell}{2}}{\sinh^2\frac{\ell}{2}}}
=
2\coth\tfrac{\ell(X_0)}{2}.
\]
Now consider what happens for a general $X\in\calT_\Sigma$. The process of Fenchel--Nielsen twisting by $\tau$ produces a $\pi$-rotation symmetric figure as depicted in \cref{fig:D_fig3} (left figure). 

\begin{center}
\begin{figure}[ht!]
    \includegraphics[width=0.75\textwidth]{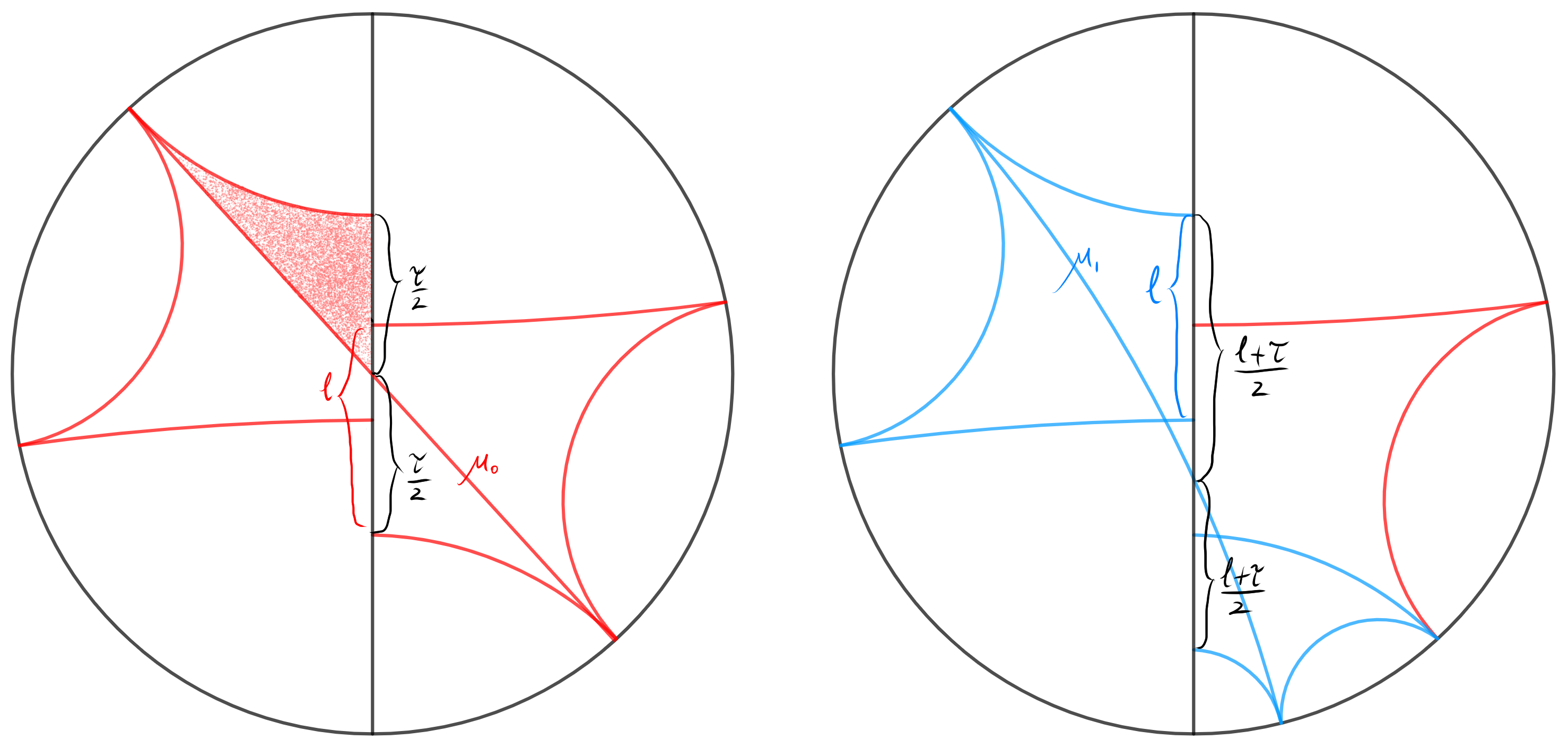}
\caption{Left: a depiction of $\mu_0$ and how it helps to detect the twist parameter $\tau$. Right: similar depiction for $\mu_1$.}\label{fig:D_fig3}
\end{figure}
\end{center}

In particular, the shaded right-angled triangle has one finite side of length $\frac{\tau}{2}$, and two infinite sides whose horocycle-truncated lengths are $\log\lambda_0(X_0)$ and $\log\lambda_0(X)$. Elementary calculations\footnote{See, for example, the proof of \cite[Proposition~3.7]{huang_thesis} with $r=\lambda_0(X_0)$. One can also argue this via more general principles combining \cite[2.3.1.(v)]{Buser} and the fact that (the boundary) of a standard collar neighborhood ``stabilizes'' as geodesic lengths tend to $0$ (to clarify, the $\beta$ in Buser's result is tending to $0$).} show that 
\[
\log\lambda_0(X)-\log\lambda_0(X_0)=\log\cosh\tfrac{\tau}{2},
\]
and hence
\begin{align}
    \lambda_0(X)=\frac{2\cosh\frac{\tau}{2}}{\tanh\frac{\ell}{2}}.\label{eq:lambda0}
\end{align}
To determine $\lambda_1(X)$, we note that the displacement between the perpendiculars between endpoints of $\mu_1$ and $\gamma$ is $\tau+\ell$ (see the right figure in \cref{fig:D_fig3}), and so
\begin{align}
    \lambda_1(X)=\frac{2\cosh\frac{\ell+\tau}{2}}{\tanh\frac{\ell}{2}}.\label{eq:lambda1}
\end{align}
Now, we convert from lambda-length (see, e.g.: \cite[Paragraph before Figure~5]{gekhtman2005cluster}, which cites \cite{pennercoords} and \cite{fock1997dual}) to shearing parameters to get
\begin{align}
s_0&=2\log\lambda_1-2\log(2\cosh\tfrac{\ell}{2})=2\log\cosh\tfrac{\ell+\tau}{2}-2\log\sinh\tfrac{\ell}{2},
\text{ and }
\\
s_1&=2\log(2\cosh\tfrac{\ell}{2})-2\log\lambda_0=2\log\sinh\tfrac{\ell}{2}-2\log\cosh\tfrac{\tau}{2}.
\end{align}
We finish off the proof by computing $\mathrm{d}s_0\wedge\mathrm{d}s_1$. Since the choice of how to assign the ordering for $\triangle$ was arbitrary, we really only consider this $2$-form up to sign:
\begin{align*}
    \mathrm{d}s_0\wedge\mathrm{d}s_1
    &=
    \left(\tanh(\tfrac{\ell+\tau}{2})\mathrm{d}\ell+\tanh(\tfrac{\ell+\tau}{2})\mathrm{d}\tau-\coth\tfrac{\ell}{2}\mathrm{d}\ell\right)
    \wedge
    \left(\coth\tfrac{\ell}{2}\mathrm{d}\ell-\tanh\tfrac{\tau}{2}\mathrm{d}\tau\right)\\
    &=\frac{\tanh\tfrac{\tau}{2}-\tanh\tfrac{\ell+\tau}{2}-\tanh\tfrac{\ell}{2}\tanh\tfrac{\tau}{2}\tanh\tfrac{\ell+\tau}{2}}{\tanh\tfrac{\ell}{2}}
    \mathrm{d}\ell\wedge\mathrm{d}\tau\\
    &=-\mathrm{d}\ell\wedge\mathrm{d}\tau,
\end{align*}
where the last equality is obtained via the compound angle formula for the hyperbolic tangent function. Finally, we note that the equality holds for every choice of Fenchel--Nielsen coordinate, since $\tau$ is well-defined up to adding a constant and hence $\mathrm{d}\tau$ is well-defined.
\end{proof}

Using the above result, we prove a formula for the Weil--Petersson volume form that ensures that our shearing-coordinates-based definition of the Weil--Petersson volume for Teichm\"uller spaces of crowned hyperbolic surfaces satisfies the neck recursion formula. This is used in the proof of \cref{prop:WP-vol-form-scc}.

\begin{theorem}
\label{app:conversion}
Let $\gamma$ be an essential simple closed curve on $\Sigma$, recall that $\Sigma\setminus\gamma$ is a surface obtained from cutting along $\gamma$. This process produces two new interior punctures of $\Sigma$, which we previously labeled as $p_{m+1}$ and $p_{m+2}$, and for the upcoming result, we denote the corresponding cuffs for these two new punctures on $\Sigma\setminus\gamma$ by $\gamma'$ and $\gamma''$. Then,

\begin{align*}
\Omega^{\mathrm{WP},\varnothing}_{\Sigma}
= 
\Omega \wedge \mathrm{d}\ell_\gamma\wedge \mathrm{d}\tau_\gamma,
\end{align*}
where $\Omega$ is a differential whose restriction to 
\[
\{X\in\calT^{\varnothing}_\Sigma\mid \ell_\gamma(X)=x,\ \tau_\gamma(X)=y\}
\cong
\{X\in\calT^{\varnothing}_{\Sigma\setminus\gamma}\mid \ell_{\gamma'}(X)=\ell_{\gamma''}(X)=x\}
\]
is equal to $\Omega^{\mathrm{WP},\varnothing}_{\Sigma\setminus\gamma}/(\mathrm{d}\ell_{\gamma'}\wedge\mathrm{d}\ell_{\gamma''})$. 
\end{theorem}

\begin{proof}
Given an arbitrary ideal triangulation $\triangle$ of $\Sigma$, the union of the geodesic ideal triangles in $X\in\calT_\Sigma(\vec{b};\vec{0})$ which intersect the geodesic representative of $\gamma$ is the immersed image of an annulus $\mathbb{A}_{a_1,a_2}$, where injectivity fails at most on the boundary arches. Since this is an embedding on the interior (and also since $\Sigma$ is not a $1$-punctured torus), we may replace $\triangle$ with an ideal triangulation where precisely two ideal triangles intersect $\gamma$. Replace $\triangle$ with the new triangulation. This means that the union of all ideal triangles in (the new) $\triangle$ which intersect $\gamma$ is an embedded $\mathbb{A}_{1,1}$ in $\Sigma$. Label the edges in $\triangle$ which lie on $\mathbb{A}_{1,1}$ by
\begin{itemize}
    \item $\mu_0,\mu_1$ for the ideal arcs which intersect $\gamma$, and 
    \item $\mu_L,\mu_R$ for the boundary arches. 
\end{itemize}
Thanks to \cref{lem:conversion}, we first obtain that $\Omega_\Sigma^{\mathrm{WP},\varnothing}=\pm\Omega\wedge\mathrm{d}\ell_\gamma\wedge\mathrm{d}\tau_\gamma$, where
\begin{align}
\Omega
=
\bigwedge_{\mu\in\triangle\setminus\{\mu_0,\mu_1\}} \mathrm{d}s_\mu
\end{align}

The ideal triangulation $\triangle$ induces an ideal triangulation $\hat{\triangle}$ on $\Sigma\setminus\gamma$ in the following way: the curve $\gamma$ cuts $\mu_0$ into $\mu_{0,L},\mu_{0,R}$, and we define
\begin{align}
\hat{\triangle}:=
\hat{\triangle}'\cup
\{\mu_L,\mu_R,\mu_{0,L},\mu_{0,R}\}
,\quad\text{where}\quad
\hat{\triangle}':=\triangle\setminus\{\mu_L,\mu_R,\mu_0,\mu_1\}.
\end{align}
Note that the choice to use $\mu_0$ was arbitrary, and using $\mu_1$ instead would have yielded ideal arcs $\mu_{1,L},\mu_{1,R}$ respectively isotopic to $\mu_{0,L},\mu_{0,R}$. We set shearing coordinates $\{\hat{s}_\mu\}_{\mu\in\hat{\triangle}}$ on $\calT_{\Sigma\setminus\gamma}$ with the spiraling direction for $\{\mu_{0,L},\mu_{0,R}\}$ chosen to yield non-negative shearing coordinates $\hat{s}_{\mu_{0,L}}$ and $\hat{s}_{\mu_{0,R}}$. Firstly, standard results about the length of a boundary component equality the absolute value of the sum of the shearing parameters for ideal arcs (counted with multiplicity) incident to the relevant boundary means that
\begin{align}
\hat{s}_{\mu_{0,R}}=\ell_{\gamma'}
\quad\text{and}\quad
\hat{s}_{\mu_{0,R}}=\ell_{\gamma''}.\label{eq:conversion1}
\end{align}
Next, we observe that the geometries of most ideal triangles (and how they're attached) remain unchanged when passing from $\calT^{\varnothing}_\Sigma$ to $\calT^{\varnothing}_{\Sigma\setminus\gamma}$. Specifically, 
\begin{align}
\text{for all $\mu\in\triangle'$,}\quad s_\mu=\hat{s}_\mu. 
\end{align}
We next examine $\hat{s}_{\mu_L}$ and $\hat{s}_{\mu_R}$. Given $X\in\calT_\Sigma$, let the lambda-lengths of the cyclic quadrilateral (taken with respect to $\triangle$) surrounding $\mu_L$ be given by $\lambda_0,\lambda_1,\lambda_2,\lambda_3$, and recall that
\[
s_{\mu,L}
=
\log\frac{\lambda_2}{\lambda_3}+\log\frac{\lambda_0}{\lambda_1}
=
\log\frac{\lambda_2}{\lambda_3}+\log\frac{\cosh\frac{\tau_\gamma}{2}}{\cosh\frac{\ell_\gamma+\tau_\gamma}{2}}.
\]
By the continuity of cross-ratios, \cref{eq:lambda0} and \cref{eq:lambda1} from the previous lemma, we have
\begin{align}
\hat{s}_{\mu_L}
=
\lim_{\tau\to\infty}
s_{\mu_L}
=
\log\frac{\lambda_2}{\lambda_3}+\tfrac{1}{2}\mathrm{d}\ell_\gamma
=
s_{\mu_L}-\log\frac{\cosh\frac{\tau_\gamma}{2}}{\cosh\frac{\ell_\gamma+\tau_\gamma}{2}}+\tfrac{1}{2}\mathrm{d}\ell_\gamma.
\end{align}
By symmetry (or using the same arguments), 
\[
s_{\mu_R}
=
s_{\mu_R}-\log\frac{\cosh\frac{\tau_\gamma}{2}}{\cosh\frac{\ell_\gamma+\tau_\gamma}{2}}+\tfrac{1}{2}\mathrm{d}\ell_\gamma.
\]

We use the above formulae to show that $\Omega$ is equal to $\Omega^{\mathrm{WP},\varnothing}_{\Sigma\setminus\gamma}/(\mathrm{d}\ell_{\gamma'}\wedge\mathrm{d}\ell_{\gamma''})$ when restricted to $\{X\in\calT^{\varnothing}_{\Sigma\setminus\gamma}\mid \ell_{\gamma'}(X)=\ell_{\gamma''}(X)=x\}$. \cref{eq:conversion1} tells us that 
\[
\mathrm{d}\ell_{\gamma'}\wedge\mathrm{d}\ell_{\gamma''}
=
\mathrm{d}\hat{s}_{\mu_{0,L}}\wedge\mathrm{d}\hat{s}_{\mu_{0,R}},
\]
and so
\[
\Omega^{\mathrm{WP},\varnothing}_{\Sigma\setminus\gamma}/(\mathrm{d}\ell_{\gamma'}\wedge\mathrm{d}\ell_{\gamma''})
=
\pm
\left(
\bigwedge_{\mu\in\triangle'}\mathrm{d}\hat{s}_\mu
\right)
\wedge
\mathrm{d}\hat{s}_{\mu_L}\wedge\mathrm{d}\hat{s}_{\mu_R}.
\]
When restricted to 
\[
\{X\in\calT^{\varnothing}_\Sigma\mid \ell_\gamma(X)=x,\ \tau_\gamma(X)=y\}
\cong
\{X\in\calT^{\varnothing}_{\Sigma\setminus\gamma}\mid \ell_{\gamma'}(X)=\ell_{\gamma''}(X)=x\},
\]
the parameters $\ell_\gamma$ and $\tau_\gamma$ are constant and hence $\mathrm{d}\hat{s}_{\mu_L}=\mathrm{d}s_{\mu_L}$ and $\mathrm{d}\hat{s}_{\mu_R}=\mathrm{d}s_{\mu_R}$. Therefore,
\[
\Omega^{\mathrm{WP},\varnothing}_{\Sigma\setminus\gamma}/(\mathrm{d}\ell_{\gamma'}\wedge\mathrm{d}\ell_{\gamma''})
=
\left(
\bigwedge_{\mu\in\triangle'}\mathrm{d}s_\mu
\right)
\wedge
\mathrm{d}s_{\mu_L}\wedge\mathrm{d}s_{\mu_R}
=
\Omega,
\]
as required.
\end{proof}

\begin{remark}
A more-or-less equivalent way of thinking about the above result is that we have obtained an expression for the Weil--Petersson volume $\Omega^{\mathrm{WP},\varnothing}_{\Sigma}$ in terms of ``mixed coordinates'' which combine Fenchel--Nielsen and shearing coordinates. Specifically,
\begin{align*}
\Omega^{\mathrm{WP},\varnothing}_{\Sigma}
= 
\pm
\left(
\bigwedge_{\mu\in\triangle'}\mathrm{d}s_\mu
\right)
\wedge
\mathrm{d}\hat{s}_{\mu_L}\wedge\mathrm{d}\hat{s}_{\mu_R}
\wedge \mathrm{d}\ell_\gamma\wedge \mathrm{d}\tau_\gamma,
\end{align*}
with if we allow for a little notation abuse.
\end{remark}
\medskip

\section{Another proof of \cref{prop:prod-formula-con-func}}
\label{app:another-proof}
Let $x=\cosh \alpha, y=\cosh \beta, z= \cosh \gamma$. Observe the following identity:
\begin{equation}
\label{eq:alg-id-cosh}
x^2+y^2+z^2+2xyz-1 = (\cosh \gamma + \cosh(\alpha+\beta))(\cosh \gamma + \cosh (\alpha-\beta)).
\end{equation}
Thus the left-hand side of \cref{eq:prod-form-con-func} equals 
\[\int_0^\infty \frac{P_{-1/2+\sqrt{-1}\xi}(\cosh \gamma)\sinh \gamma\,\mathrm{d}\gamma}{\sqrt{(\cosh \gamma+p)(\cosh \gamma+q)}}\] for $p=\cosh(\alpha+\beta),q=\cosh(\alpha-\beta)$. Using the identity
\begin{equation}
\frac{1}{\sqrt{AB}} = \frac{2}{\pi} \int_0^{\pi/2}\frac{\mathrm{d}\phi}{A\sin^2\phi+B\cos^2\phi}, \quad A,B>0
\end{equation}
with $A=\cosh \gamma+p, B=\cosh \gamma+q$, the left hand side of \cref{eq:prod-form-con-func}  equals 
\begin{equation}
\label{eq:prod-form-con-fun-step}
\frac{2}{\pi}\int_0^\infty\int_0^{\pi/2} \frac{P_{-1/2+\sqrt{-1}\xi}(\cosh \gamma)\sinh \gamma}{\cosh \gamma + r(\phi)}\,\mathrm{d}\phi\,\mathrm{d}\gamma,
\end{equation}
where $r(\phi) = p\sin^2\phi+q\cos^2\phi$. Since $r(\phi)\geqslant 0$ and $P_{-1/2+\sqrt{-1}\xi}(z)=O(z^{-1/2})$ as $z\to \infty$ \cite[Equation~1.9.7]{MR2254107}, by Fubini--Tonelli theorem, we can switch the order of integration in \cref{eq:prod-form-con-fun-step}. Next, by Mehler's identity \cite[Equation~44]{HT25}, \[\int_1^\infty \frac{P_{-1/2+\sqrt{-1}\xi}(z)\mathrm{d}z}{z+r(\phi)}= \frac{\pi}{\cosh \pi\xi}\cdot P_{-1/2+\sqrt{-1}\xi}(r(\phi)).\]
Let's unpack $r(\phi)$:
\begin{equation*}
r(\phi) = \cosh\alpha\cosh \beta-\sinh\alpha\sinh\beta\cos2\phi.
\end{equation*}
Let $\psi=\pi-2\phi$. We invoke the following product formula for conical functions (see \cite[Page~325]{MR229863}, where the geometric meaning of the formula is also given):
\begin{equation*}
\begin{split}
&\frac{1}{\pi} \int_0^\pi P_{-1/2+\sqrt{-1}\xi}(\cosh \alpha\cosh\beta+\sinh\alpha\sinh\beta\cos\psi) \mathrm{d}\psi \\\smash{\mathllap{=}}\,& P_{-1/2+\sqrt{-1}\xi}(\cosh\alpha) P_{-1/2+\sqrt{-1}\xi}(\cosh \beta).
\end{split}
\end{equation*}
Putting it together, we obtain
\begin{equation*}
\begin{split}
\int_1^\infty \frac{P_{-1/2+\sqrt{-1}\xi}(z)\mathrm{d}z}{\sqrt{x^2+y^2+z^2+2xyz-1}} & = \frac{2}{\cosh\pi\xi}\cdot \frac{\pi}{2} P_{-1/2+\sqrt{-1}\xi}(x)P_{-1/2+\sqrt{-1}\xi}(y) \\&=
\frac{\pi}{\cosh \pi\xi}\cdot P_{-1/2+\sqrt{-1}\xi}(x) P_{-1/2+\sqrt{-1}\xi}(y).
\end{split}
\end{equation*}
\medskip

\section{Abel regularization of $\mu_{k,n}$}
\label{app:abel}

Set $m=2k+2$ and define

\[
\mu_{k,n}^{(\varepsilon)}(\vec{c}) = \pi^{n-1}\int_0^\infty e^{-\varepsilon d} d^{m-1 }\left(\int_0^\infty \cos (\omega d) \Phi_n(\omega,\vec{c})\,\mathrm{d}\omega\right)\,\mathrm{d}d,  \quad \varepsilon>0.
\]
Then the Fubini-Tonelli theorem applies, because

\begin{equation*}
\begin{split}
&\int_0^\infty \int_0^\infty e^{-\varepsilon d} d^{m-1 }|\cos(\omega d)|\cdot|\Phi_{n}(\omega,\vec{c})|\,\mathrm{d}\omega\,\mathrm{d}d
\\\smash{\mathllap{\leqslant}}\,&\int_0^\infty \int_0^\infty e^{-\varepsilon d} d^{m-1 }|\Phi_{n}(\omega,\vec{c})|\,\mathrm{d}\omega\,\mathrm{d}d
\\\smash{\mathllap{=}}\,& \frac{\Gamma(m)}{\varepsilon^{m}} \|\Phi_n(\cdot,\vec{c})\|_{L^1}<\infty.
\end{split}
\end{equation*}
Therefore,

\begin{equation*}
\begin{split}
\mu_{k,n}^{(\varepsilon)}(\vec{c})& = \pi^{n-1}\int_0^\infty \Phi_{n}(\omega,\vec{c}) \left( \int_0^\infty e^{-\varepsilon d}d^{m-1}\cos(\omega d)\,\mathrm{d}d \right) \,\mathrm{d}\omega
\\& = \pi^{n-1}\Gamma(m)\int_0^\infty \Phi(\omega,\vec{c})\cdot\text{Re}\left(\varepsilon-\sqrt{-1}\omega\right)^{-m}\,\mathrm{d}\omega,
\end{split}
\end{equation*}
since $e^{-\varepsilon d}\cos(\omega d) = \text{Re}\Bigl(e^{-\left(\varepsilon-\sqrt{-1}\omega\right)d}\Bigr).$

Since $\Phi_n(\omega,\vec{c})$ is an even function of $\omega$, let 
\[
T_{2k}\Phi(\omega,\vec{c}) = \sum_{j=0}^k \frac{\partial^{2j}_\omega \Phi({0,\vec{c})}}{(2j)!}\omega^{2j}
\]
be the $2k$-th order Taylor polynomial of $\Phi(\omega,\vec{c})$ at $0$.

Note that for $0\leqslant j\leqslant k$, the beta-integral gives
\[
\int_0^\infty \omega^{2j}(\varepsilon-\sqrt{-1}\omega)^{-m} = \varepsilon^{2j+1-m} (-\sqrt{-1})^{-(2j+1)}B(2j+1,m-2j-1),
\] 
which is purely imaginary. 
Hence the Taylor polynomial contributes trivially to the real part, and 
\[
\mu_{k,n}^{(\varepsilon)}(\vec{c}) = \pi^{n-1}\Gamma(m)\int_0^\infty(\Phi_n-T_{2k}\Phi_n)(\omega,\vec{c} )\cdot\text{Re}\left(\varepsilon-\sqrt{-1}\omega\right)^{-m}\,\mathrm{d}\omega.
\]
Since for fixed $\vec{c}$,  
\begin{equation*}
\label{eq:Phi_n-minus-taylor}
\Phi_n(\omega,\vec{c} )-T_{2k}\Phi_n(\omega,\vec{c} ) = O(\omega^{2k+2})
\end{equation*}
as $\omega\to0$, dominated convergence applies: indeed, for small $\omega$, the integrand is $O(\omega^{2k+2})\cdot\omega^{-m}=O(1)$, and for large $\omega$, we get $O(\omega^{-2})$ since $\Phi_n$ decays rapidly. 

Moreover, $\text{Re}\left(-\sqrt{-1}\omega\right)^{-m}=(-1)^{k+1}\omega^{-2k-2}.$ Letting $\varepsilon\to0^+,$ we obtain
\[
\mu_{k,n}(\vec{c}) = (-1)^{k+1}(2k+1)!\pi^{n-1}\int_0^\infty\frac{\Phi_n(\omega,\vec{c} )-T_{2k}\Phi_n(\omega,\vec{c} )}{\omega^{2k+2}}\,\mathrm{d}\omega.
\]
This last integral agrees with  the analytically continued Mellin transform $\mathcal{M}\{\Phi_n\}(s)$ at $s=-2k-1$. Indeed, if we split the integral $\int_0^\infty \omega^{s-1}\Phi_n\,\mathrm{d}\omega$ into $\int_0^1+\int_1^\infty$, substitute Taylor truncation, and evaluate meromorphic pieces, we obtain:
\begin{equation*}
\begin{split}
 \mathcal{M}\{f\}(s) &= \int_0^1 \omega^{s-1}(f(\omega)-T_{2k}f(\omega))\,\mathrm{d}\omega + \sum_{j=0}^k a_{2j}\int_0^1\omega^{s+2j-1}\mathrm{d}\omega 
 \\&\phantom{=\int_0^1 \omega^{s-1}(f(\omega)-T_{2k}f(\omega))\,\mathrm{d}\omega}+\int_1^\infty \omega^{s-1}f(\omega)\mathrm{d}\omega
 \\& = \int_0^1 \omega^{s-1}(f(\omega)-T_{2k}f(\omega))\,\mathrm{d}\omega + \sum_{j=0}^k \frac{a_{2j}}{s+2j}+\int_1^\infty \omega^{s-1}f(\omega)\mathrm{d}\omega,
\end{split}
\end{equation*}
where $f=\Phi_n, T_{2k}f(\omega)=\sum_{j=0}^k a_{2j}\omega^{2j}.$ This defines a meromorphic continuation of $\mathcal{M}\{f\}(s)$ on $s>-2k-2$, with possible poles at $s=0,-2,\dotsc,-2k$. Substituting $s=-2k-1$:
\begin{equation*}
\begin{split}
\mathcal{M}\{f\}(-2k-1) = \int_0^1 \frac{f(\omega)-T_{2k}f(\omega)}{\omega^{2k+2}}\mathrm{d}\omega - \sum_{j=0}^k \frac{a_{2j}}{2k+1-2j}+\int_0^\infty \frac{f(\omega)}{\omega^{2k+2}}\mathrm{d}\omega
\end{split}
\end{equation*}
But 
\[
\int_1^\infty \frac{T_{2k}f(\omega)}{\omega^{2k+2}}\mathrm{d}\omega = \sum_{j=0}^k a_{2k}\int_1^\infty \omega^{2j-m}\mathrm{d}\omega = \sum_{j=0}^k \frac{a_{2j}}{2k+1-2j},
\]
Hence we conclude that
\[\mathcal{M}\{f\}(-2k-1)  = \int_0^\infty \frac{f(\omega)-T_{2k}f(\omega)}{\omega^{2k+2}}\mathrm{d}\omega.\]
\medskip

\section{Mellin transform and the hyperbolic secant function}
The following lemma describes the structure of the Mellin transform values of $\text{sech}^n\pi\omega$ at odd integers.

\begin{lemma}
\label{lem:mellin-sechn-values}
If $2p+1>0$ and $n>0$, then
\begin{equation}
\pi^{n-1}\mathcal{M}\{\text{sech}^{n}\pi\omega\}(2p+1) \in \sum_{i=0}^{\min\{p,\lfloor\frac{n-1}{2}\rfloor\}} \QQ\pi^{2(\lfloor\frac{n-1}{2}\rfloor-i)}.
\end{equation}
If $2p+1<0$ and $n=2h+1>0$ is odd, then
\begin{equation}
\pi^{n-1}\mathcal{M}\{\text{sech}^{n}\pi\omega\}(2p+1) \in \sum_{i=0}^{h} \QQ\pi^{2(h-i)}\beta(2i-2p).
\end{equation}
If $2p+1<0$ and $n=2h+2>0$ is even, then
\begin{equation}
\pi^{n-1}\mathcal{M}\{\text{sech}^{n}\pi\omega\}(2p+1) \in \sum_{i=0}^{h} \QQ\pi^{2(h-i)}\zeta(2i-2p+1).
\end{equation}
\end{lemma}
\begin{proof}
For $n\geqslant1$, an explicit computation shows that
\begin{equation}
\left(\text{sech}^n(\pi\omega)\right)'' = n^2\pi^2\text{sech}^n(\pi\omega)-n(n+1)\pi^2\text{sech}^{n+2}(\pi\omega).
\end{equation}
Since $\mathcal{M}\{f''\}(s)=(s-1)(s-2)\mathcal{M}\{f\}(s-2)$ \cite[Equation~17.42.2(i)]{MR2360010}, we have the following recursion:
\begin{equation}
\label{eq:mellin-sech-rec}
\mathcal{M}\{\text{sech}^{n+2}\pi\omega\}(s) = \frac{n}{n+1}\mathcal{M}\{\text{sech}^{n}\pi\omega\}(s)-\frac{(s-1)(s-2)}{n(n+1)\pi^2}\mathcal{M}\{\text{sech}^{n}\pi\omega\}(s-2).
\end{equation}
Hence it is sufficient to compute for $n=1,2.$  For $n=1$, assuming $\text{Re}\,s>0$:
\begin{equation}
\label{eq:mellin-sech}
\begin{split}
\mathcal{M}\{\text{sech}\pi\omega\}(s) &= \mathcal{M}\left\{\frac{2e^{-\pi\omega}}{1+e^{-2\pi\omega}}\right\}(s)  
\\& = \mathcal{M}\left\{2\sum_{n=0}^\infty (-1)^n e^{-(2n+1)\pi\omega} \right\}(s)
\\& = 2\sum_{n=0}^\infty (-1)^n\int_0^\infty \omega^{s-1}  e^{-(2n+1)\pi\omega}  \,\mathrm{d}\omega 
\\& = \frac{2\Gamma(s)}{\pi^s}\sum_{n=0}^\infty \frac{(-1)^n}{(2n+1)^s} 
\\& = \frac{2\Gamma(s)}{\pi^s}\beta(s).
\end{split}
\end{equation}
Since $\beta(2k+1)\in\QQ\pi^{2k+1}$ \cite{wikipedia:dirichlet_beta}, then by \cref{eq:mellin-sech}, $\mathcal{M}\{\text{sech}\pi\omega\}(2k+1)\in\QQ$. Furthermore, from the  functional equation for $\beta(s)$ \cite{wikipedia:dirichlet_beta}, the meromorphic continuation of $\mathcal{M}\{\text{sech}\pi\omega\}(s)$ at $s=-2k-1$ equals
\begin{equation}
\label{eq:mellin-sech-values}
\mathcal{M}\{\text{sech}\pi\omega\}(-2k-1) = (-1)^{k+1}2^{2k+2}\beta(2k+2).
\end{equation}
Suppose  $n=2h+1$ is odd and  $2p+1>0$. Then from the recursion \cref{eq:mellin-sech-rec}, $\mathcal{M}\{\text{sech}^n\pi\omega\}(2p+1)$
is a rational linear combination of 
\begin{equation}
\begin{split}
&\mathcal{M}\{\text{sech}\pi\omega\}(2p+1),\pi^{-2}
\mathcal{M}\{\text{sech}\pi\omega\}(2p-1),
\\&\dotsc,\pi^{-2\min\{p,h\}}\mathcal{M}\{\text{sech}\pi\omega\}(\max\{1,2p-2h+1\}),
\end{split}
\end{equation}
and there are no smaller values of Mellin transform since the factor $(s-1)(s-2)$ in \cref{eq:mellin-sech-rec} equals zero for $s=1$. Therefore,
\[
\pi^{n-1}\mathcal{M}\{\text{sech}^{n}\pi\omega\}(2p+1) \in \sum_{i=0}^{\min\{p,h\}} \QQ\pi^{2(h-i)}.
\]
Suppose $2p+1<0$. Then from the recursion \cref{eq:mellin-sech-rec}, $\mathcal{M}\{\text{sech}^n\pi\omega\}(2p+1)$
is a rational linear combination of 
\begin{equation*}
\begin{split}
&\mathcal{M}\{\text{sech}\pi\omega\}(2p+1),\pi^{-2}
\mathcal{M}\{\text{sech}\pi\omega\}(2p-1),
\dotsc,\pi^{-2h}\mathcal{M}\{\text{sech}\pi\omega\}(2p-2h+1),
\end{split}
\end{equation*}
and by \cref{eq:mellin-sech-values}
\[
\pi^{n-1}\mathcal{M}\{\text{sech}^{n}\pi\omega\}(2p+1) \in \sum_{i=0}^{h} \QQ\pi^{2(h-i)}\beta(2i-2p).
\]
Next, for $n=2$, assuming $\text{Re}\,s>1$:
\begin{equation}
\label{eq:mellin-sech2}
\begin{split}
\mathcal{M}\{\text{sech}^2\pi\omega\}(s) &= \mathcal{M}\left\{\frac{4e^{-2\pi\omega}}{(1+e^{-2\pi\omega})^2}\right\}(s)  
\\& = \mathcal{M}\left\{4\sum_{n=1}^\infty (-1)^{n-1}n e^{-2n\pi\omega} \right\}(s)
\\& = 4\sum_{n=1}^\infty (-1)^{n-1}n\int_0^\infty \omega^{s-1}  e^{-2n\pi\omega}  \,\mathrm{d}\omega 
\\& = \frac{4\Gamma(s)}{(2\pi)^s}\sum_{n=1}^\infty \frac{(-1)^{n-1}}{n^{s-1}} 
\\& = \frac{2^{2-s}(1-2^{2-s})\Gamma(s)}{\pi^s}\zeta(s-1).
\end{split}
\end{equation}
Since $\zeta(2k)\in\QQ\pi^{2k}$ \cite[Equation~9.542(1)]{MR2360010}, then $\mathcal{M}\{\text{sech}^2\pi\omega\}(2k+1) \in \QQ\frac{1}{\pi}$ by \cref{eq:mellin-sech-values}. Furthermore, from the  functional equation for $\zeta(s)$ \cite[Equation~9.535(3)]{MR2360010}, the meromorphic continuation of $\mathcal{M}\{\text{sech}^2\pi\omega\}(s)$ at $s=-2k-1$ equals
\begin{equation}
\label{eq:mellin-sech2-values}
\mathcal{M}\{\text{sech}^2\pi\omega\}(-2k-1) = (-1)^{k+1}\frac{(2k+2)(2^{2k+3}-1)}{\pi}\zeta(2k+3).
\end{equation}
Suppose  $n=2h+2$ is even and  $2p+1>0$. Then similarly, $\mathcal{M}\{\text{sech}^n\pi\omega\}(2p+1)$
is a rational linear combination of 
\begin{equation}
\begin{split}
&\mathcal{M}\{\text{sech}^2\pi\omega\}(2p+1),\pi^{-2}
\mathcal{M}\{\text{sech}^2\pi\omega\}(2p-1),
\\&\dotsc,\pi^{-2\min\{p,h\}}\mathcal{M}\{\text{sech}^2\pi\omega\}(\max\{1,2p-2h+1\}),
\end{split}
\end{equation}
and then
\[
\pi^{n-1}\mathcal{M}\{\text{sech}^{n}\pi\omega\}(2p+1) \in \sum_{i=0}^{\min\{p,h\}} \QQ\pi^{2(h-i)}.
\]
If $2p+1<0$, then by \cref{eq:mellin-sech2-values}, we similarly have
\[
\pi^{n-1}\mathcal{M}\{\text{sech}^{n}\pi\omega\}(2p+1) \in \sum_{i=0}^{h} \QQ\pi^{2(h-i)}\zeta(2i-2p+1).
\]
\end{proof}

\section{An addition formula for elliptic integrals}
\label{app:add-formula}
Let 
\begin{equation}
F(x,k) = \int_0^x \frac{\mathrm{d}t}{\sqrt{(1-t^2)(1-k^2t^2)}}
\end{equation}
denote the incomplete elliptic integral of the first kind.

The following \textit{addition formula} holds (\cite[Equation~116.01]{EllipticIntegrals}):
\begin{equation}
F(x,k)\pm F(y,k) = F(z,k),    
\end{equation}
where
\begin{equation}
\label{eq:ellip-int-add-formula}
z = \arccos \left(\frac{\cos x\cos y \mp  \sin x\sin y \sqrt{(1-k^2\sin^2x)(1-k^2\sin^2y)}}{1-k^2\sin^2x\sin^2y} \right).    
\end{equation}
We apply this formula to the elliptic integrals that appear in \cref{thm:vol-flip-2-crown} (flippered 2-crowns) and \cref{thm:vol-flip-quadr} (flippered quadrilaterals). 

\subsection{Flippered 2-crowns}
In \cref{eq:vol-2-crown-gen-case}, we find
\begin{equation}
\begin{split}
&F\left(\arcsin\sqrt{\frac{t+\xi_k}{t+\xi_j}},\sqrt{\frac{\xi_i-\xi_j}{\xi_i-\xi_k}}
\right)\Bigg|_1^\infty  \\\smash{\mathllap{=}}\,& F\left(\frac{\pi}{2},\sqrt{\frac{\xi_i-\xi_j}{\xi_i-\xi_k}}
\right)-F\left(\arcsin\sqrt{\frac{1+\xi_k}{1+\xi_j}},\sqrt{\frac{\xi_i-\xi_j}{\xi_i-\xi_k}}
\right).  
\end{split}
\end{equation}
Then by \cref{eq:ellip-int-add-formula} with $x=\frac{\pi}{2}, y = \arcsin \sqrt{\frac{1+\xi_k}{1+\xi_j}}$, we compute
\begin{equation}
\begin{split}
z = \arccos\left( \frac{0+\sin y \sqrt{(1-k^2)(1-k^2\sin^2y)}}{1-k^2\sin^2y} \right) &= \arccos\frac{\sin y\sqrt{1-k^2}}{\sqrt{1-k^2\sin^2y}}
\\& = \arccos \frac{\sqrt{\frac{1+\xi_k}{1+\xi_j}}\cdot\sqrt{\frac{\xi_j-\xi_k}{\xi_i-\xi_k}}}{\sqrt{\frac{(1+\xi_i)(\xi_j-\xi_k)}{(1+\xi_j)(\xi_i-\xi_k)}}}
\\& = \arccos \sqrt{\frac{1+\xi_k}{1+\xi_i}}.
\end{split}
\end{equation}
Hence $z = \arcsin \sqrt{\frac{\xi_i-\xi_k}{\xi_i+1}}$ as in \cref{eq:vol-2-crown-gen-case}.

\subsection{Flippered quadrilaterals}
In \cref{eq:vol-4-gon-gen-case}, we find
\begin{equation}
\begin{split}
&F\left(\arcsin\sqrt{\frac{(t+\xi_\ell)(\xi_i-\xi_k)}{(t+\xi_k)(\xi_i-\xi_\ell)}},\sqrt{\frac{(\xi_i-\xi_\ell)(\xi_j-\xi_k)}{(\xi_i-\xi_k)(\xi_j-\xi_\ell)}}\right) \Bigg|_1^\infty 
\\\smash{\mathllap{=}}\,& F\left(\arcsin\sqrt{\frac{\xi_i-\xi_k}{\xi_i-\xi_\ell}},\sqrt{\frac{(\xi_i-\xi_\ell)(\xi_j-\xi_k)}{(\xi_i-\xi_k)(\xi_j-\xi_\ell)}}\right) 
\\&  - F\left(\arcsin\sqrt{\frac{(1+\xi_\ell)(\xi_i-\xi_k)}{(1+\xi_k)(\xi_i-\xi_\ell)}},\sqrt{\frac{(\xi_i-\xi_\ell)(\xi_j-\xi_k)}{(\xi_i-\xi_k)(\xi_j-\xi_\ell)}}\right).
\end{split}
\end{equation}
Then by \cref{eq:ellip-int-add-formula} with $x=\arcsin\sqrt{\frac{\xi_i-\xi_k}{\xi_i-\xi_\ell}}, y = \arcsin\sqrt{\frac{(1+\xi_\ell)(\xi_i-\xi_k)}{(1+\xi_k)(\xi_i-\xi_\ell)}}$, we compute
\begin{equation}
\begin{split}
z &= \arccos \left(
\frac{\splitfrac{\sqrt{\frac{\xi_k-\xi_\ell}{\xi_i-\xi_\ell}}\sqrt{\frac{(1+\xi_i)(\xi_k-\xi_\ell)}{(1+\xi_k)(\xi_i-\xi_\ell)}}}
{+\sqrt{\frac{\xi_i-\xi_k}{\xi_i-\xi_\ell}}\sqrt{\frac{(1+\xi_\ell)(\xi_i-\xi_k)}{(1+\xi_k)(\xi_i-\xi_\ell)}}\sqrt{\left(1-\frac{\xi_j-\xi_k}{\xi_j-\xi_\ell}\right)\left(1-\frac{(1+\xi_\ell)(\xi_j-\xi_k)}{(1+\xi_k)(\xi_j-\xi_\ell)}\right)}}}{1-\frac{(1+\xi_\ell)(\xi_i-\xi_k)(\xi_j-\xi_k)}{(1+\xi_k)(\xi_i-\xi_\ell)(\xi_j-\xi_\ell)}}
\right)
\\& = \arccos
\left(\frac{\frac{\xi_k-\xi_\ell}{\xi_i-\xi_{\ell}}\sqrt{\frac{1+\xi_i}{1+\xi_k}}+\frac{(\xi_i-\xi_k)(\xi_k-\xi_\ell)\sqrt{(1+\xi_j)(1+\xi_\ell)}}{(\xi_i-\xi_\ell)(\xi_j-\xi_\ell)(1+\xi_k)}}{1-\frac{(1+\xi_\ell)(\xi_i-\xi_k)(\xi_j-\xi_k)}{(1+\xi_k)(\xi_i-\xi_\ell)(\xi_j-\xi_\ell)}}\right)
\\& = \arccos\left(\frac{(\xi_j-\xi_\ell)\sqrt{(1+\xi_i)(1+\xi_k)}+(\xi_i-\xi_k)\sqrt{(1+\xi_j)(1+\xi_\ell)}}{(1+\xi_i)(1+\xi_j)-(1+\xi_k)(1+\xi_\ell)}\right)
\\& = \arccos\left(\frac{\sqrt{(1+\xi_i)(1+\xi_\ell)}+\sqrt{(1+\xi_j)(1+\xi_k)}}{\sqrt{(1+\xi_i)(1+\xi_j)}+\sqrt{(1+\xi_k)(1+\xi_\ell)}}\right).
\end{split}
\end{equation}
Hence $z=\arcsin\frac{\sqrt{(\xi_i-\xi_k)(\xi_j-\xi_\ell)}}{\sqrt{(1+\xi_i)(1+\xi_j)}+\sqrt{(1+\xi_k)(1+\xi_\ell)}}$ as in \cref{eq:vol-4-gon-gen-case}.

\newpage

\bibliographystyle{plain}
\bibliography{bibliography}
\end{document}